\documentclass[english]{article}
\usepackage[T1]{fontenc}
\usepackage[latin9]{inputenc}
\usepackage{xcolor}
\usepackage{babel}
\usepackage{verbatim}
\usepackage{amsmath}
\usepackage{amsthm}
\usepackage{amssymb}
\usepackage{graphicx}
\usepackage[unicode=true,pdfusetitle,
 bookmarks=true,bookmarksnumbered=false,bookmarksopen=false,
 breaklinks=false,pdfborder={0 0 0},pdfborderstyle={},backref=false,colorlinks=true]{hyperref}
\newcommand{\be}{\begin{equation}}
\newcommand{\ee}{\end{equation}}
\newcommand{\ba}{\begin{array}}
\newcommand{\ea}{\end{array}}
\newcommand{\dsp}{\displaystyle}
\newcommand{\N}{\mathbb{N}}
\newcommand{\Z}{\mathbb{Z}}
\newcommand{\R}{\mathbb{R}}
\newcommand{\Exp}{\mathbb{E}}
\newcommand{\Prob}{\mathbb{P}}
\newcommand{\supinf}{><}
\newcommand{\speed}{\sigma}
\newtheorem{lemma}{Lemma}[section]
\newtheorem{example}{Example}[section]
\newtheorem{corollary}{Corollary}[section]
\newtheorem{definition}{Definition}[section]
\newtheorem{remark}{Remark}[section]
\newcommand{\thmref}[1]{Theorem \ref{thm:#1}}

\makeatletter

\AtBeginDocument{}
\AtBeginDocument{\providecommand\thmref[1]{\ref{thm:#1}}}
\AtBeginDocument{}
\AtBeginDocument{}

  \theoremstyle{remark}
  
\theoremstyle{plain}

  \theoremstyle{plain}
  
	\newtheorem{proposition}{Proposition}[section]
\newtheorem{assumption}{Assumption}[section]
\newtheorem{theorem}{Theorem}[section]

\AtBeginDocument{}

\AtBeginDocument{}

\makeatother

  \providecommand{\lemmaname}{Lemma}
  \providecommand{\remarkname}{Remark}
\providecommand{\theoremname}{Theorem}

\begin{document}

\title{ Invariant measures for multilane exclusion process}

\author{G. Amir$^a$, C. Bahadoran$^b$, O. Busani$^c$, E. Saada$^d$}
\maketitle
$$ \ba{l}
^a\,\mbox{\small Department of Mathematics, Bar Ilan University,} \\
\quad \mbox{\small 5290002 Ramat Gan, Israel. E-mail: amirgi@macs.biu.ac.il}\\
^b\,\mbox{\small Laboratoire de Math\'ematiques Blaise Pascal, Universit\'e Clermont Auvergne,} \\
\quad \mbox{\small 63177 Aubi\`ere, France. E-mail: christophe.bahadoran@uca.fr}\\
^c\, 
\mbox{\small  University of Edinburgh,
5321, James Clerk Maxwell Building,} \\
\quad \mbox{\small Peter Guthrie Tait Road,
Edinburgh, United Kingdom.}\\
\quad\mbox{\small E-mail: 	obusani@ed.ac.uk}\\
^d\, \mbox{\small 
CNRS, UMR 8145, MAP5, Universit\'e Paris Cit\'e, Campus Saint-Germain-des-Pr\'es,} \\
\quad \mbox{\small  75270 Paris cedex 06, France.
 E-mail: Ellen.Saada@mi.parisdescartes.fr}\\
\ea
$$
\begin{abstract}
 We consider the simple exclusion process on $\Z\times\{0,1\}$, 
 that is,  an ``horizontal ladder''  composed of $2$ lanes,
 depending on 6 parameters. 
 Particles can jump according 
to a lane-dependent translation-invariant nearest neighbour jump kernel,
i.e. ``horizontally'' along each lane, and
``vertically'' along the scales of the ladder.
We prove that  generically, the set of extremal invariant measures 
consists of \textit{(i)} translation-invariant product Bernoulli measures; 
and, modulo translations along $\Z$: \textit{(ii)} at most two 
shock measures (i.e. asymptotic to Bernoulli measures at $\pm\infty$) 
with asymptotic densities $0$ and $2$; \textit{(iii)} at most 
 one (outside degenerate cases) shock measure  with a  density jump 
 of magnitude $1$. We fully determine this set for a range of  parameter values.   
 Our results can be   generalized  in several directions using 
the same approach
and answer certain open questions formulated in \cite{ligd} 
as a  step towards the  process on $\mathbb{Z}^2$. 
\end{abstract}
\section{Introduction}\label{sec:intro}
The simple exclusion process, introduced in \cite{spi}, 
is a fundamental model in statistical mechanics.
In this markovian process, particles hop on a countable lattice 
following a certain random walk kernel subject to the exclusion rule, 
that allows at most one particle per site.  
As usual for Markov processes, 
the characterization of its invariant measures 
is one of the basic questions to address.
Still today, outside the case of a symmetric kernel (\cite{ligsym}), 
the problem is far from being completely solved. In fact, it has been 
mostly studied for translation invariant kernels. We briefly recall  
known results in this situation. \\ \\
For the exclusion process on $\Z^d$, the set of  extremal 
\textit{translation invariant}  (also called {\em homogeneous})  
stationary  probability  measures consists 
(\cite{liggett2012interacting}) of homogeneous Bernoulli product measures. 
However, for a non-symmetric  kernel, there may exist 
extremal invariant  probability 
measures that are \textit{not} translation invariant.  
These are fairly well  (though not completely)  understood 
in one-space dimension  (\cite{Liggett1976, fls, BMM, Bramson2002};  
see also \cite{ligd} for open questions): 
under suitable assumptions, there is a unique (up to translations) 
such extremal  probability  measure, called either a \textit{blocking} 
or a \textit{profile} measure (the latter being 
a weakened version of the former); 
its main feature (for a kernel with, say, a positive drift) 
  is that the asymptotic particle density 
is $0$ to the left and $1$ to the right  of the origin.  \\ \\
In several space dimensions, 
although  analogues of  blocking or profile measures 
can be exhibited (\cite{ligd}), 
 the complete characterization of invariant  probability 
 measures remains an open question.  The paper \cite{ligd} 
 initiated a program in this direction. The authors
introduced so-called {\em $v$-homogeneous} measures, that is, 
measures invariant by translations in directions orthogonal 
to a given vector $v$, and {\em $v$-profile} measures, 
that is $v$-homogeneous measures with asymptotic density $0$ 
at $-\infty$ and $1$ at $+\infty$ parallel to $v$. They showed 
that when $v$ is orthogonal to the drift, extremal stationary
$v$-homogeneous measures are homogeneous Bernoulli measures.
They proved that under some conditions on the jump kernel and 
vector $v$, extremal $v$-profile measures are given by an explicit 
family of product measures analogous to  those in  
\cite{Liggett1976}. Finally, 
they decomposed the problem of characterizing all invariant measures 
into a series of open questions. The first of these are  (BL1)  whether any 
non-homogeneous extremal stationary measure is $v$-homogeneous 
for some $v$, and  (BL2)  whether it is $v$-profile for some $v$. 
These questions were also formulated for the so-called 
{\em ladder process},  where one among two dimensions is cyclic, 
mentioned in \cite{ligd} as an interesting step towards the process on $\Z^2$. 
In this context, $v$-homogeneity is interpreted as cyclic rotational invariance. 
\\  \\
In the present paper,  we obtain characterization results  
(Theorems \ref{thm:characterization_lemma} 
 to \ref{thm:Characterization_of_invariant_measures_gen}), 
 for  intermediate models between dimensions 1 and 2 containing 
  the above ladder process.  As explained below, we exhibit new phenomena 
and a richer behaviour as compared to the one  dimensional
single-lane exclusion process.  
We consider  first  the simple exclusion process 
on $\Z\times\{0,1\}$,  that is  an ``horizontal  
ladder''  composed of $2$ lanes. 
Particles can jump ``horizontally''  to nearest neighbour sites 
along each lane according to a lane-dependent translation-invariant 
jump kernel, and ``vertically'' along the scales of the ladder 
according to another kernel. 
In the totally asymmetric case, this can be interpreted 
as traffic-flow on a highway, 
with two lanes on which cars have different speeds and 
different directions. \\ \\ 
We  now  
describe our results  for the two-lane model.  
Let $\gamma_0,\gamma_1$ denote mean drifts on each lane, 
$p$ the jump rate 
from lane $0$ to lane $1$ and  $q$ the jump rate 
from lane $1$ to lane $0$.  
The drifts may be of 
equal or opposite signs; one or both of them may also vanish. 
We assume that $p+q>0$, so that both lanes are indeed connected. 
We prove that 
the set $\mathcal I_e$ of extremal invariant  probability  measures
can be decomposed as a disjoint union
\begin{equation}\label{decomp_extremal}
\mathcal I_e=\mathcal I_0\cup\mathcal I_1\cup\mathcal I_2
\end{equation}
In this decomposition, $\mathcal I_0:=\{\nu_\rho,\rho\in[0,2]\}$ is 
the set of extremal invariant  probability  measures that are 
\textit{translation invariant} along lanes. 
The parameter $\rho$ represents the total density over the two lanes. 
Under  $\nu_\rho$, the mean densities 
$\rho_0,\rho_1$ on each lane are functions of $\rho$, 
and they are different 
when $p\neq q$. In the sequel, we refer to these  probability  measures 
as ``Bernoulli measures''. For $k\in\{1,2\}$, $\mathcal I_k$ denotes a 
(possibly empty) set of extremal invariant  
probability  measures that  we call  
\textit{shock measures} of \textit{amplitude} $k$. 
By a shock measure, we mean a  probability  
measure that is asymptotic to two Bernoulli measures 
of different densities
$\rho^-$, resp. $\rho^+$, when viewed from faraway left, resp. right
 (w.r.t. the origin).   
 We define the amplitude of the shock 
to be  $k:=|\rho^+-\rho^-|$. 
The set $\mathcal I_2$ contains only shocks such that 
$(\rho^-,\rho^+)=(0,2)$ or $(\rho^-,\rho^+)=(2,0)$. 
These measures are the analogue in our context of 
blocking measures or profile measures. 
In some cases, $\mathcal I_1$ may contain  what we call 
\textit{partial} blocking measures, i.e., measures 
whose restriction to one lane is a blocking measure, 
and whose restriction to
 the other lane is  a Dirac measure concentrated either on
 the full configuration or on the empty one.  \\ \\
We show that the following generic picture holds 
outside some  degenerate cases:  up to translations along $\Z$, 
 \textit{(i)} the set $\mathcal I_1$ contains at most  one 
 probability  measure;  \textit{(ii)}  the set  
$\mathcal I_2$  contains at most two   probability  measures.  
In particular, these sets are at most countable. 
We can fully determine 
$\mathcal I_1$ and $\mathcal I_2$, 
and thus obtain a complete characterization of invariant  
probability  measures,  for a subset of parameter values 
including the following situations: 
 (a)  when $\gamma_0$, $\gamma_1$ are close enough, the ratio 
$q/p$ small enough or large enough,  and
\begin{equation}\label{cond_blocking}
d_0/l_0=d_1/l_1\neq 1
\end{equation}
 where $d_i$, resp. $l_i$, denotes the jump rate to the right, 
 resp. left, on lane $i\in\{0,1\}$;
(b)  when $p$ or $q$ vanishes and $\gamma_0\neq\gamma_1$;  (c) 
when $\gamma_0=\gamma_1=0$ and $p,q$ are arbitrary.  \\ \\
 In case (b), we exhibit {\em partial} blocking measures 
(where only one lane has a blocking measure), a new phenomenon 
with respect to single-lane  asymmetric simple exclusion process (ASEP).  Another
result in sharp contrast with the 
one-dimensional case is that $\mathcal I_2$ may be empty when $pq=0$
 even if  the drifts are both 
strictly positive (or both strictly negative); and when it is not, 
it is described by two integer parameters representing 
two independent shock locations instead of a single parameter 
in the usual ASEP.  
 In case (a), our characterization can be viewed in this context 
as a positive answer to open question (BL2) above from \cite{ligd}. 
The set $\mathcal I_2$ is then derived from a family of 
two-dimensional product blocking measures that are analogues 
in this context (see Remark \ref{remark_v}) of certain $v$-profile 
measures constructed on $\Z^d$ in \cite{ligd}. 
We observe here some structural similarity 
between elements of $\mathcal I_2$ and extremal blocking measures 
constructed in \cite{bfj} for the single-lane 
Misanthrope's process.   
It would be interesting to know if  two-dimensional  blocking measures 
can lead to remarkable combinatorial identities as in \cite{bfj}.  \\ \\
The following questions are left open. First, we can 
show that $\mathcal I_1$ is 
indeed nonempty in cases where it contains \textit{only} 
partial blocking measures,  and that it is empty 
on a set of parameter values
for which $\gamma_0$ and $\gamma_1$ are close enough, 
and the ratio between $p$ and $q$ small enough (or large enough).
We do not know if  for certain parameter values  
it is possible to have $\mathcal I_1$ nonempty with 
a shock of amplitude  $1$ that is
not a partial blocking measure. 
In the case $p=q$ (and more generally for the vertically  
cyclic ladder process, 
see below), it is believed in \cite{ligd} that 
this probably does not occur. 
 Next,  we conjecture that when 
$pq>0$,  under a suitable assumption 
(see Remark \ref{remark_open}) verified in particular when both drifts are 
strictly positive, $\mathcal I_2$ is nonempty,
even without the assumption 
\eqref{cond_blocking}.  We believe that this could be 
proved in the spirit of \cite{BMM} 
by means of the hydrodynamic limit. We shall investigate 
the hydrodynamic behaviour 
of our model  and extensions thereof (see below) in \cite{abbs2}.  \\ \\
 In the assumptions of  Theorems \ref{thm:characterization_lemma}--\ref{cor:entire},
to avoid cumbersome statements and proofs, we have not aimed at fullest possible 
generality. Nevertheless, we stress that our approach is robust enough to handle 
more general or related situations without substantial changes. 
In Appendix \ref{app:ext}, 
we provide a detailed discussion of such extensions with
precise assumptions and conclusions, and explain why the ideas of proofs 
developed in the body of the paper carry over to such situations.
These include non-nearest neighbour jump kernels, multilane processes 
with more than two lanes, and Misanthrope's processes. We point out that 
although the latter 
are single-lane generalizations of the simple exclusion process, 
the characterization 
of their invariant measures (outside translation invariant ones) 
is still an open problem. 
We realized along the way that, though this question was not our initial motivation, 
it could be partly solved by our methods.
Among the above extensions, the vertically cyclic multilane ladder process 
from \cite{ligd} 
is however treated in  Subsection \ref{subsec:rotation} rather than in the appendix, 
because our corresponding Theorem \ref{thm:Characterization_of_invariant_measures_gen} 
answers question (BL1) above from \cite{ligd}, namely all invariant measures 
are rotationally invariant.  \\ \\
 One of the difficulties of our models is that available approaches 
(\cite{Liggett1976, Bramson2002}) to classify invariant measures 
for the one-dimensional single-lane  asymmetric simple exclusion 
process  rely heavily on the fact that at most one particle 
is allowed on each site. 
 In the aforementioned works, the line of argument is to show that 
 for a non translation-invariant stationary measure, 
 the  {\em mean} density difference between $-\infty$ and $+\infty$ 
 is at least $1$. Since the possible density range is $[0,1]$, 
 this {\em automatically} implies that the measure is a shock 
 with asymptotic densities $0$ and $1$ at $\pm\infty$ (see Remark \ref{remark:diff}). 
In our case,  the range of global densities is no longer restricted to $1$ 
but to the number of lanes. A different and more complex scheme of proof 
(see outline in Subsection \ref{subsec:ideas}) is imposed by this, 
but also by the interplay of several parameters leading to a wider variety 
of behaviours. One key point is to show 
{\em a priori} that an invariant measure is a shock. This is done thanks to a  
novel and robust argument (Proposition \ref{prop:step2})  using  extremality 
and attractiveness, which can be transposed to other attractive models.  
{\em Then} we carry out an analysis of possible shocks based on the macroscopic 
flux function of the model.  Note that this density range problem arises also 
for the Misanthrope's process and similarly makes the characterization problem
for this model  different than for the simple exclusion process.    \\ \\
Another difficulty that occurs when interlane jumps are possible 
only in one direction is the lack of irreducibility for the jump kernel. 
Usual arguments  (in the line of \cite{Liggett1976})  
based on attractiveness and irreducibility, 
showing that discrepancies between two coupled processes 
eventually disappear  (see e.g. \cite{Liggett1976}),  
 are not sufficient in this case. \\ \\  
 We finally mention that while revising this manuscript, we became aware 
 that the case where all lanes are symmetric (corresponding to 
 $\gamma_0=\gamma_1=0$ in the basic two-lane model) had been recently studied 
in \cite{rvw} by different duality methods.  \\ \\
 This paper is organized as follows.  Models are introduced 
 in Section \ref{sec:models}.  
We then state our results on invariant measures for the two-lane  
simple exclusion process:  
Theorem \ref{thm:characterization_lemma} for the 
invariant and translation 
invariant probability measures,  
Theorems \ref{thm:Characterization_of_invariant_measures}
and \ref{cor:entire}
for the invariant probability measures; 
finally Theorem \ref{thm:Characterization_of_invariant_measures_gen}
 deals with the multilane  simple exclusion process,  
 and in particular with the
ladder process from \cite{ligd}. 
Section \ref{sec:proofs}
is devoted to the proof of Theorem \ref{thm:characterization_lemma}, 
and Section
\ref{sec:proofs-2} to the proofs of Theorems 
\ref{thm:Characterization_of_invariant_measures}, \ref{cor:entire}
and \ref{thm:Characterization_of_invariant_measures_gen}.  
In order to make the general schemes of proofs more visible, 
 the main ideas are first explained in 
Subsection \ref{subsec:ideas}, and  
most intermediate results used to establish 
Theorems \ref{thm:Characterization_of_invariant_measures} 
and \ref{cor:entire} are proved in the separate 
Section \ref{subsec:proofs_lemmas}. 
 Extensions of our results are discussed in  Appendix \ref{app:ext}.  
\section{Models and results}\label{sec:models}
 In this section,  we  present and state our results
 for our basic model, the \textit{two-lane}  SEP 
(motivated by traffic-flow considerations), and for its generalization
to a  \textit{multilane}  SEP.
Before that, we first recall the definition of 
the simple exclusion process on a countable set $V$.
The two-lane  and the multilane SEP indeed belong to this class, 
but they have  specific properties 
due to the structure of the set $V$.
\subsection{Simple exclusion process}\label{subsec:single_lane}
 Throughout the paper, $\Z$ denotes the set of integers 
 and $\N$ the set of nonnegative integers. 
Let $V$ be a nonempty countable set.
The state space of the process is 
\be\label{state_space}\mathcal{X}:=\{0,1\}^V
\ee
that is a compact polish space with respect to product topology. 
One can think of   $\eta\in\mathcal X$  as a configuration 
of particles on $V$, 
i.e. for which a site $x\in V$ is occupied by a particle if and only if
$\eta\left(x\right)=$1. \newline\newline
 We call \textit{kernel} on $V$ a function  
 $p:V\times V\to[0,+\infty)$ such that
\begin{equation}\label{cond_lig}
\sup_{x\in V}\left\{
\sum_{y\in V}p(x,y)+\sum_{y\in V}p(y,x)
\right\}<+\infty
\end{equation}
The   $(V,p)$-simple exclusion process (in short: SEP) 
 is a Markov process  
$\left(\eta_{t}\right)_{t\geq0}$ on $\mathcal{X}$ 
(see \cite[Chapter VIII]{liggett2012interacting})  with generator
\begin{equation}
Lf\left(\eta\right)
=\sum_{x,y\in V}p\left(x,y\right)\eta\left(x\right)
\left(1-\eta\left(y\right)\right)
\left(f\left(\eta^{x,y}\right)-f\left(\eta\right)\right),
\label{eq:generator of the Exclusion}
\end{equation}
where  $\eta^{x,y}$, given by
\[
\eta^{x,y}\left(w\right)=\left\{ \begin{array}{cc}
\eta\left(w\right) & w\neq x,y\\
\eta(x)-1 & w=x\\
\eta(y)+1 & w=y
\end{array}\right.,
\] 
is the new configuration after a particle has jumped from $x$ to $y$, 
and $f$ is a cylinder (or local) function, that is, 
a function that depends
only on the value of $\eta$ on a finite number of sites in $V$. 
We denote by $(S_t)_{t\geq 0}$ the semigroup generated by 
\eqref{eq:generator of the Exclusion},  and by $\Exp_\mu$, 
resp. $\Exp_\eta$, 
the expectation for the process with initial distribution 
a probability measure
$\mu$ on $\mathcal{X}$, resp. with initial configuration  
$\eta\in\mathcal{X}$.\newline\newline 
The (nearest-neighbour) SEP  on $\Z$ is the particular case of 
\eqref{eq:generator of the Exclusion} with $(V,p)$ given  by 
\be\label{def_tasep} 
V=\Z, \quad p(x,y)=d{\bf 1}_{\{y-x=1\}}+l{\bf 1}_{\{y-x=-1\}};
\quad d,l\geq 0,\,d+l>0
\ee
Within this category we distinguish 
the \textit{symmetric}, resp. \textit{asymmetric} exclusion process 
(SSEP, resp. ASEP), for which $d=l$, resp. $d\neq l$; and the 
\textit{totally asymmetric} simple exclusion process (TASEP) on $\Z$, 
for which $dl=0<d+l$. \newline\newline
A probability measure $\mu$ on $\mathcal X$ is said to be 
\textit{invariant} for the Markov process generated by 
\eqref{eq:generator of the Exclusion} if it is invariant 
with respect to the semigroup $(S_t)_{t\geq 0}$, which is equivalent to
\begin{equation}\label{def_inv}
\int Lf(\eta)d\mu(\eta)=0
\end{equation}
for every cylinder function $f$. The set of invariant probability 
measures is denoted by $\mathcal I$. 
Since $\mathcal{I}$ is convex, by Choquet-Deny Theorem, 
in order to know $\mathcal I$,  it is enough to determine 
the subset of its extremal elements,  denoted by $\mathcal I_e$. \newline\newline
   A probability measure $\mu$ is said to be {\em reversible} 
if $L$ is a self-adjoint operator 
in $L^2(\mathcal X,\mu)$. Reversible measures are invariant; 
when they exist, they are usually easier to compute explicitely 
than non-reversible invariant measures. For instance, 
the following general result, which will be helpful, 
can be found (in a slightly different formulation) 
in \cite[Chapter VIII]{liggett2012interacting}.
\begin{proposition}\label{prop:lig}
Let $S$ be a countable subset and $\pi(.,.)$ a kernel on $S$ 
satisfying \eqref{cond_lig}.
Let $\rho_.=(\rho_i)_{i\in S}$ be a $[0,1]$-valued family
such that, for every  $i,j\in S$,  the following condition holds:
\begin{equation}
 \rho_i(1-\rho_j)\pi(i,j)=\rho_j(1-\rho_i)\pi(j,i)
\label{eq:detailed balance_alpha}
\end{equation}
 Define the product measure  ${\mu}_{S,\rho_.}$ on $\{0,1\}^S$ by
\begin{equation}
%
\mu_{S,\rho_.}(d\eta)=\bigotimes_{i\in S}\mathcal B(\rho_i)(d\eta_i),
\label{alpha_x}
\end{equation}
where $\mathcal B(\rho)$ denotes the Bernoulli measure with parameter $\rho$. 
Then  ${\mu}_{S,\rho_.}$  is reversible with respect to the   $(S,\pi)$ 
simple exclusion process. 
\end{proposition}
\begin{remark}\label{rk_prod}
When the family $\rho_.$ has constant value $\rho\in[0,1]$, the product measure defined by \eqref{alpha_x} will be denoted by $\mu_{S,\rho}$.
The subscript $S$ will be dropped whenever there is no ambiguity.
\end{remark}
\subsection{The general setup}
\label{subsec:two_lane}
In the sequel, we shall focus on special choices of $V$ and  $p(.,.)$ 
for which the model has an interesting structure.  
First, we consider a lattice $V$ of the form 
\be\label{structure}
V=\Z\times W
\ee
for some 
nonempty finite set $W$. An element $x$ of $V$ will be 
generically written 
in the form $x=(x(0),x(1))$, with $x(0)\in\Z$ and $x(1)\in W$. 
In traffic-flow modeling, 
we may think of  $V$ as a highway, of $\Z$ as a lane, 
and of $x$ as site $x(0)$ on lane 
$x(1)$. For $i\in W$,
\begin{equation}\label{eq:general lane}
\mathbb{L}_{i}:=
\left\{
 x\in V:\,x(0)\in\mathbb{Z},\,x(1)=i
\right\} 
\end{equation}
denotes the $i$'th lane of $V$,
 and $\eta^i$  the particle configuration on $\Z$,  defined by
\be\label{config_lane}\eta^{i}\left(z\right)=\eta\left(z,i\right)\ee
 for $z\in\Z$.  We can view $\eta^i$ as the configuration on lane $i$.
 Another interpretation is that $i\in W$ represents a particle species, 
 then $\eta(z,i)=\eta^i(z)$ is the number of particles of species $i$ 
 at site  $z\in\Z$.  We also denote by 
\be\label{config_total}
\overline\eta(z)=\sum_{i\in W}\eta^i(z)
\ee
the total number of particles at $z\in\Z$. \newline\newline
 Next, we consider kernels $p(.,.)$  of the form 
\be\label{restrict_kernel}
p(x,y)=\left\{
\ba{lll}
0\mbox{ if } x(0)\neq y(0)\mbox{ and }x(1)\neq y(1)\\ 
 q_i(x(0),y(0))=:Q_i[y(0)-x(0)] & \mbox{if} & x(1)=y(1)=i\\ 
q(x(1),y(1)) & \mbox{if} & x(0)=y(0)
\ea
\right.
\ee
 for $x,y\in V$,  where $q(.,.)$ is a kernel on $W$  
 (that will be given afterwards for the $W$ of interest), 
and for each $i\in W$, $q_i(.,.)$ is a translation invariant kernel 
on $\Z$ given by
\be\label{restrict_2}
q_i(u,v)= d_i{\bf 1}_{\{v-u=1\}}+l_i{\bf 1}_{\{v-u=-1\}}
,\quad  Q_i(z)=d_i{\bf 1}_{\{z=1\}}+l_i{\bf 1}_{\{z=-1\}}
\ee
 for $u,v\in\Z$,  where $d_i\geq 0$  and $l_i\geq 0$ 
 are such that $d_i+l_i>0$. \newline\newline  
We shall be interested in translations along $\Z$, but the set $W$
 is in general not endowed with a translation operator.
We denote by $(\tau_k)_{k\in\mathbb{Z}}$ the group of space shifts 
on $\Z$. 
The shift operator $\tau_k$ acts on a particle configuration
$\eta\in\mathcal X$ through
\begin{equation}\label{shift_config}
(\tau_k\eta)(z,w):=\eta(z+k,w),\quad\forall (z,w)\in\mathbb{Z}\times W
\end{equation}
It acts on a function $f:\mathcal X\to\mathbb{R}$ via
\begin{equation}\label{shift_fct-f}
(\tau_k f )(\eta):=f(\tau_k\eta),\quad\forall \eta\in\mathcal X
\end{equation}
If   $\mu$ is a probability measure on $\mathcal X$, 
 then  $\tau_k$ acts on $\mu$ via
\begin{equation}\label{shift_fct-mu}
\int_{\mathcal X}f(\eta) d(\tau_k\mu)(\eta) 
:=\int_{\mathcal X}(\tau_k f)(\eta)d\mu(\eta)
\end{equation}
for every bounded continuous function $f:\mathcal X\to\mathbb{R}$.
Last, if $\mathcal L$ is a linear operator acting on functions 
$f:\mathcal X\to\mathbb{R}$, 
 then  $\tau_k$ acts on $\mathcal L$ via
\begin{equation}\label{shift_fct-L}
(\tau_k \mathcal L)f:={\mathcal L}(\tau_k f)
\end{equation}
By an abuse of notation, in what follows,
we write $\tau$ instead of $\tau_{1}$.
We define
$\mathcal{S}$ to be the set of all probability measures 
on $\mathcal{X}$ that
are invariant under the translations $\tau_k$,  $k\in\Z$. 
\subsection{The two-lane  SEP }
\label{subsec:two_lane2}
 In the sequel, we shall sometimes refer to 
  SEP (resp. SSEP, ASEP, TASEP) as \textit{single-lane} or 
 \textit{one-dimensional} SEP (resp. SSEP, ASEP, TASEP).  
Our  basic  model is  \textit{the two-lane} SEP,  which  corresponds to
\be\label{def_ladder}
%
W=\{0,1\}
\ee
We can view this model as a dynamics on an infinite 
horizontal ladder, with vertical
steps separating its two bars 
$\mathbb{L}_{0}$ and $\mathbb{L}_{1}$, namely:
\begin{align}\begin{array}{cc}
\mathbb{L}_{0} & = \, \left\{ x\in V:x=\left(z,0\right),
z\in\mathbb{Z}\right\} \\
\mathbb{L}_{1} & = \, \left\{ x\in V:x=\left(z,1\right),
z\in\mathbb{Z}\right\}  \label{L0L1}
\end{array}\end{align}
In the traffic interpretation,  we call $\mathbb{L}_{0}$ 
and $\mathbb{L}_{1}$  respectively
the {\em upper} and {\em lower} lane,  
and the steps between them the 
direction a car can follow to change lane. 
 Thus we shall henceforth call {\em downward} jump a jump from 
$\mathbb{L}_0$ to $\mathbb{L}_1$, and {\em upward} jump a jump from 
$\mathbb{L}_1$ to $\mathbb{L}_0$.     \newline\newline
Let $p,q\geq 0$ and  $d_0,l_0,d_1,l_1\geq 0$. 
The two-lane  SEP  is the dynamics on
$\mathcal{X}$ defined by the generator 
\eqref{eq:generator of the Exclusion}
 with kernel \eqref{restrict_kernel}--\eqref{restrict_2}, 
in which  $q(.,.)$ is given by  
\be\label{restrict_1}
q(0,1)=p,\quad q(1,0)=q
\ee  
This means that,  for $x,y\in V$,   
\begin{equation}
p\left(x,y\right)=\left\{ \begin{array}{ccc}
 d_0 & \mbox{if} & x,y\in\mathbb{L}_{0},\ensuremath{y(0)-x(0)=1}\\
 l_0 &  \mbox{if} &  x,y\in\mathbb{L}_{0},\ensuremath{y(0)-x(0)=-1}\\
d_1  & \mbox{if} & x,y\in\mathbb{L}_{1},\ensuremath{y(0)-x(0)=1}\\
 l_1  &  \mbox{if} &  x,y\in\mathbb{L}_{1},\ensuremath{y(0)-x(0)=-1}\\
  p & \mbox{if} & x\in\mathbb{L}_{0},y\in\mathbb{L}_{1},x(0)=y(0)\\
q & \mbox{if} &  x\in\mathbb{L}_{1},y\in\mathbb{L}_{0},x(0)=y(0)\\
0 & \mbox{otherwise} &
\end{array}\right.\label{eq:intensity_c}
\end{equation}
In other words,  particles move one step to the right  
or to the left  on each lane 
at a rate depending on the lane, and we allow the rate 
$p$ at which particles go down 
to be different than the rate $q$ of going up.   
We shall assume in the sequel that (cf. \eqref{restrict_2})  
\be\label{always_move}(d_0+l_0)(d_1+l_1)>0\ee
so that particles can always move on both lanes. However
they  cannot go from  $\mathbb{L}_{0}$ to $\mathbb{L}_{1}$ if $p=0$, 
nor from  $\mathbb{L}_{1}$ to $\mathbb{L}_{0}$ if $q=0$. If $p=q=0$, 
the dynamics reduces to two independent   SEP's  on each lane. Thus, 
 $p+q\neq 0$ introduces interaction between the two lanes. 
For $i\in W$, we let 
\be\label{def_drift}
\gamma_i:=d_i-l_i
\ee
denote the mean drift on lane $i$. 
The following symmetry properties of the two-lane SEP will be useful. 
Define the \textit{lane symmetry} operator 
$\sigma:\mathcal X\to\mathcal X$,
 the \textit{lane exchange} operator $\sigma':\mathcal X\to\mathcal X$, 
 and the \textit{particle-hole} symmetry operator 
 $\sigma'':\mathcal X\to\mathcal X$ by
\be\label{sym_op}
(\sigma\eta)(z,i)=\eta(-z,i);\,
(\sigma'\eta)(z,i)=\eta(z,1-i);\,
(\sigma''\eta)(z,i)=1-\eta(z,i)
\ee
for $\eta\in\mathcal X$ and $(z,i)\in V$.
Let us call the process defined by  the generator
\eqref{eq:generator of the Exclusion} with transition kernel  
\eqref{eq:intensity_c} the $(d_0,l_0);(d_1,l_1);(p,q)$-two-lane SEP.
The definition of the two-lane SEP dynamics implies the following.
\begin{lemma}\label{lemma_sym}
Let $(\eta_t)_{t\geq 0}$ be a $(d_0,l_0);(d_1,l_1);(p,q)$-two-lane SEP. 
\newline
Then the image of this process by $\sigma$, 
resp. $\sigma',\sigma''$, is a
$(l_0,d_0);(l_1,d_1);(p,q)$, resp. $(d_1,l_1);(d_0,l_0);(q,p)$, 
resp. $(l_0,d_0);(l_1,d_1);(q,p)$-two-lane SEP.
\end{lemma}
Thus, without loss of generality, we shall assume in the sequel that
\be\label{wlog}
\gamma_0\geq 0,
\quad\gamma_0+\gamma_1\geq 0,\quad  p\geq q, \quad p>0 
\ee
If we view $i\in\{0,1\}$ as a species rather than a lane, 
the interpretation is as follows: the dynamics within each species 
is a  SEP  of $\Z$, and a lane change becomes a spin flip whereby 
a particle may change its species. The exclusion rule within species 
implies that a particle cannot change its species if there is already 
a particle of the other species sitting at the same site. 
This is the only point where an interaction occurs between the two species.
\subsection{Invariant measures  for two-lane SEP }\label{subsec:inv}
Let us start with translation invariant measures.  
Recalling \eqref{decomp_extremal}, this corresponds to $\mathcal I_0$; 
the complete description of $\mathcal I_e$
will be given in Subsection \ref{subsub:Ie}. 
\subsubsection{Translation invariant stationary measures 
 for two-lane SEP}\label{subsub:tinv}
 The following two-parameter Bernoulli product probability measures 
 will be central. Let us   define  
 $\nu^{\rho_{0},\rho_{1}}$  for $(\rho_{0},\rho_{1})\in [0,1]^2$,
as the product probability measure on $\mathcal{X}$ such that 
\begin{align}
 \nu^{\rho_{0},\rho_{1}} 
\left(\eta\left(x\right)=1\right) & =\left\{ \begin{array}{cc}
\rho_{0} & x\in\mathbb{L}_{0}\\
\rho_{1} & x\in\mathbb{L}_{1}
\end{array}\right..\label{eq:two-rate Bernoulli measure}
\end{align}
In words, the two lanes are independent, 
 and for $i\in\{0,1\}$,  the projection of  
 $\nu^{\rho_{0},\rho_{1}}$  on lane $\mathbb{L}_i$ 
is the product Bernoulli measure $\mu_{\mathbb{L}_i,\rho_i}$ 
with parameter $\rho_i$, see \eqref{alpha_x} 
and Remark \ref{rk_prod}.  \newline\newline
When $p=q=0$, as mentioned after \eqref{always_move}, 
the two lanes evolve as independent  SEP's,   hence 
 $\nu^{\rho_{0},\rho_{1}}$  is an invariant measure 
 for every $(\rho_0,\rho_1)\in[0,1]^2$.
 We look for a relation between $\rho_{0}$ and $\rho_{1}$ 
 under which we could have 
 $\nu^{\rho_{0},\rho_{1}}\in\mathcal{I}$ 
when  $p+q\neq 0$.  
To this end, we define the following subset $\mathcal F$ of $[0,1]^2$:
\be\label{def_f_always}
{\mathcal F}:=\left\{
(\rho_0,\rho_1)\in[0,1]^2:\, p\rho_0(1-\rho_1)-q\rho_1(1-\rho_0)=0
\right\}
\ee
 The set $\mathcal F$ expresses an equilibrium 
  detailed balance  relation 
 for vertical jumps: 
 it states that under $\nu^{\rho_0,\rho_1}$, 
 the mean algebraic ``creation rate'' 
 on each lane (i.e. resulting from jumps from/to the other lane) 
 has to be $0$. 
 Similarly to the  single-lane SEP,  we have
the following theorem, proved in  Section \ref{sec:proofs}. 
\begin{theorem}
\label{thm:characterization_lemma} We have that 
\begin{eqnarray}\label{eq:charac}
\left(\mathcal{I}\cap\mathcal{S}\right)_{e} & = &
\left\{
\nu^{\rho_{0},\rho_{1}}:\,(\rho_0,\rho_1)\in\mathcal F
\right\}\\
\label{sameset} & = & \left\{
\nu_{\rho}:\,\rho\in[0,2]
\right\}
\end{eqnarray}
for a one-parameter family $\left\{\nu_{\rho}:\,0\leq\rho\leq2\right\}$
of  probability  measures on $\mathcal{X}$, where the parameter 
$\rho$ represents
the total mean density over the two lanes:
\begin{equation}
\mathbb{E}_{\nu_{\rho}}\left\{
\eta^{0}(0)+\eta^{1}(0)
\right\}
=\rho
\label{eq:density at a point}
\end{equation}
\end{theorem}
\begin{remark}\label{rem:pq=0}
 When $q=0$,  
the invariant measures $\nu_\rho$ can be guessed naturally.
Indeed in this case, particles cannot move upwards
 from $\mathbb{L}_1$ to $\mathbb{L}_0$.  
Thus if lane $0$ is empty, 
it remains empty and lane $1$ behaves as an autonomous   SEP. 
 Hence, for $\rho\in[0,1]$, the measure $\nu^{0,\rho}$ 
 (which has global density 
 $\rho$ over the two lanes) is invariant for the two-lane SEP, because 
 its restriction to lane $1$ is invariant for the SEP on this lane.
Similarly, if lane $1$ is full, it remains full and lane  
$0$ evolves as an autonomous  SEP. 
 Hence, for $\rho\in[1,2]$, the measure $\nu^{\rho-1,1}$ 
 (which has global density 
 $\rho$ over the two lanes)  is invariant for the two-lane SEP. 
This is consistent with the fact that for $q=0$, 
\eqref{def_f_always} yields 
 (see \eqref{tilderho_2b} later on) 
\[
\mathcal F=\{(0,\rho):\rho\in[0,1]\}\cup\{(\rho-1,1):\,\rho\in[1,2]\}
\]
\end{remark}
\begin{remark}
Definition $\eqref{def_f_always}$ was interpreted above as 
a detailed balance condition, but this is related only
 to the vertical part of the dynamics. 
The measure $\nu^{\rho_0,\rho_1}$ is in general not reversible, unless
$d_i=l_i$ for every $i\in W$, in which case the 
relations \eqref{eq:detailed balance_alpha} hold.
\end{remark}
 \subsubsection{Structure of invariant measures for two-lane SEP}\label{subsub:Ie}
 We are now interested in $\mathcal I_e$ rather than $(\mathcal I\cap\mathcal S)_e$, 
 and need to consider {\em blocking-type} configurations adapted to our setting. 
 Blocking configurations for simple exclusion on a general countable set of sites $S$ 
 were defined in \cite{Liggett1976}. There, for $S=\Z$, 
 the set of blocking configurations is given by 
\be\label{def_x0}
\mathcal X_1:=\left\{\eta\in\mathcal\{0,1\}^\Z:\,
\sum_{x>0}[1-\eta(x)]+\sum_{x\leq 0}\eta(x)<+\infty\right\}
\ee 
 and an invariant probability measure supported on $\mathcal X_
 1$ is called a {\em blocking measure}. 
For the two-lane model, we must define the following set:
\begin{eqnarray}
{\mathcal X}_2 & := & \left\{\eta\in\mathcal X:\,
\sum_{ x\in V:\, x(0)>0}[1-\eta(x)]
+\sum_{ x\in V:\, x(0)\leq 0}\eta(x)<+\infty\right\}\label{def_x2}
\end{eqnarray}
 In our model, a {\em blocking measure} will be an invariant probability 
 measure supported on $\mathcal X_2$. 
In  Appendix \ref{app:ext},  we discuss other settings where our approach 
also yields characterization results. The set of blocking configurations 
has to be adapted to each model. Among these models are the Misanthrope's process, 
a {\em single-lane} particle system with several possible particles per site, 
for which the definition of blocking configurations can be found in \cite{bfj}.  
Let  
\be\label{def_shocksets}
\begin{array}{lll}
&\mathcal B_1:=\{(0,1),(1,0),(1,2),(2,1)\}, 
    & \mathcal B_2:=\{(0,2)\}  \\
&\mathcal B:=\mathcal B_1\cup\mathcal B_2, 
    & \mathcal D:=\{(\rho,\rho):\,\rho\in[0,2]\} \\
\end{array}\ee
 Let $(\rho^-,\rho^+)\in[0,2]^2\setminus\mathcal D$,  
 that we call a \textit{shock.}  A probability measure 
 $\mu$ on $\mathcal X$
is called a $(\rho^-,\rho^+)$-\textit{shock measure} if
\be\label{limits_mu}
\lim_{n\to-\infty}\tau_n\mu=\nu_{\rho^-},\quad 
\lim_{n\to+\infty}\tau_n\mu=\nu_{\rho^+}
\ee
in the sense of weak convergence,  for $\nu_{\rho}$
defined in \eqref{sameset}.   
The \textit{amplitude} of the shock (or of the shock measure) 
is by definition $|\rho^+-\rho^-|$. \newline\newline
 We can now state the results of this section. Since they include 
 many different cases, for the sake of readability, 
 we will state them in several steps. 
The following theorem is proved in 
 Section \ref{sec:proofs-2}. 
\begin{theorem}
\label{thm:Characterization_of_invariant_measures}
 (i)  There exist a (possibly empty) subset $\mathcal R$ of  
$[0,2]^2\setminus(\mathcal D\cup \mathcal B)$  
containing only shocks of amplitude $1$, 
 a (possibly empty) subset $\mathcal R'$ of $\mathcal B_1$, and
for each $(\rho^-,\rho^+)\in\mathcal R\cup\mathcal R'$, 
a $(\rho^-,\rho^+)$-shock measure denoted $\nu_{\rho^-,\rho^+}$,
such that 
\begin{eqnarray}\label{charac_ext}
{\mathcal I}_e & = &
\left\{
 \nu_{\rho}:0\leq\rho\leq2\right
\}\cup  (Bl_1\cup Bl_2)  
\cup\left\{
\tau_z\nu_{\rho^-,\rho^+}:z\in\Z,\,(\rho^-,\rho^+)\in\mathcal R
\right\}\\
\label{def_bl_1}
 Bl_1  &  =   &  \{\tau_z\nu_{\rho^-,\rho^+}:\,(\rho^-,\rho^+)
\in\mathcal R',\,z\in\Z\} \\
 Bl_2 & =&  \{\nu\in\mathcal I_e:\,\nu\mbox{ is a }(0,2)\mbox{-shock measure}\}
\label{def_bl2} 
\end{eqnarray}
 (ii) The sets $\mathcal R$, $\mathcal R'$ and $Bl_2$ 
 enjoy the following properties: \newline\newline
 (a)    The set $Bl_2$ is stable by translations, 
and outside the case
\be\label{case_deg}
l_0=l_1=q=0,
\ee
it contains at most (up to translations) two elements.  
If $Bl_2$ contains at least one blocking measure, 
then it consists exactly  (up to translations)
of two blocking measures.   \newline\newline
 (b) The set $Bl_2$ is empty if 
\be\label{cond_nob2}
 q\gamma_0+p\gamma_1<0 
\ee  
 (c)   Outside the cases
\begin{eqnarray}\label{unless} 
&&p=q\mbox{ and } \gamma_0+\gamma_1=0,\\ 
\label{unless_2}
&&\gamma_0=\gamma_1=0, \\
\label{unless_3}
&&q=0=\gamma_0\gamma_1,
\end{eqnarray}
the set $\mathcal R$ contains at most  one element 
and $\mathcal R'$ at most two elements  (up to translations).   \newline\newline
 (d)   Outside \eqref{unless}--\eqref{unless_3}, 
the following holds.
Unless $q=0$ and $\gamma_0=\gamma_1>0$, 
the set $\mathcal R\cup\mathcal R'$ contains at most two elements
 (up to translations). 
If $q>0$ and $\gamma_0+\gamma_1\neq 0$, the set $\mathcal R'$ is empty. 
If $q>0$, $q\neq p$ and $\gamma_0+\gamma_1=0\neq\gamma_0\gamma_1$, 
the set $\mathcal R$ is empty. 
\end{theorem}
\begin{remark}\label{remark_case_b}
In view of \eqref{wlog}, condition \eqref{cond_nob2} implies $p>q$ 
and $\gamma_1<0<\gamma_0$.
\end{remark}
  Theorem \ref{thm:Characterization_of_invariant_measures} 
  yields the following information. 
  The decomposition  \eqref{charac_ext}--\eqref{def_bl2}
says  that  every element of $\mathcal I_e$  that 
 is a not a product Bernoulli measure 
is a shock measure of amplitude $1$ or $2$, and that 
for a given shock of amplitude $1$,  
an associated shock measure is unique up to translations.   
Outside the case \eqref{case_deg} (which will be further 
studied in the next theorem),
up to translations, we can have {\em at most} 
two shock measures of amplitude $2$.  This case is special 
because the kernel \eqref{eq:intensity_c} lacks 
standard irreducibility assumptions (see Definition \ref{def_irred}), 
so usual ordering properties must be weakened 
(see Definitions \ref{def_supinf} and \ref{def_bowtie}). 
The only possible shock of amplitude $2$  is $(0,2)$. 
The $(0,2)$-shock measures are 
analogues of \textit{blocking} or 
\textit{profile} measures in \cite{Bramson2002}. 
We shall see below that  when both drifts are positive, 
shocks of amplitude $2$ are blocking measures, 
and under additional assumptions, there are
{\em exactly} two of them modulo translations. 
Shock measures of amplitude $1$ can be divided 
into two classes with  a different meaning. The
 first one,  namely  $Bl_1$,  contains measures associated to 
shocks in $\mathcal B_1$.  The second one, namely $\mathcal R$, is associated 
to other shock measures of amplitude $1$. 
 There cannot exist measures in $\mathcal B_1$  outside cases $q=0$ 
 or $\gamma_0+\gamma_1=0$; 
they are then zero-flux measures (see Proposition \ref{prop:extrema}, {\em (o)} 
and {\em (ii)}). 
Among  elements of $Bl_1$ 
are \textit{partial} blocking measures:  we shall see below   
(in Theorem \ref{cor:entire}) 
that these  may only (and do indeed) 
arise if $q=0$. Under such measures, one lane 
carries a $(0,1)$-shock and the other is either empty or full. 
The set $\mathcal R$ is associated 
to other shock measures of amplitude $1$. 
We believe that $\mathcal R$ is empty and prove that 
it contains at most  one element  outside 
the case \eqref{unless}. 
 This conjecture comes from the belief 
that the variance of the shock is of order $t$ 
with a positive diffusion coefficient,
as follows from extrapolating the results of \cite{ff} 
for single-lane ASEP.
 This property is incompatible with a shock stationary state, 
 but suggests (as in \cite{fks} for single-lane ASEP) existence of 
 a stationary state for the process seen from a proper random location. 
In contrast,  based on this extrapolation, 
we expect the diffusion coefficient to vanish 
in the last case of Theorem \ref{thm:Characterization_of_invariant_measures}, 
\textit{(c)}; we have no clear conjecture whether 
$\mathcal R'$ is empty in this case. 
Under \eqref{unless}, 
the model is diffusive and nongradient, and 
we conjecture that the only invariant measures are Bernoulli. 
We leave  the above conjectures  for future 
investigation, as the methods involved to prove  
 them  are presumably quite different 
from those used here. \newline\newline
Next, we
provide more information on the sets $\mathcal R$, 
$\mathcal R'$, $Bl_1$ and $Bl_2$,
and obtain a full description of $\mathcal I_e$ 
 for a set of parameter values   including  
\eqref{case_deg}  and \eqref{unless_2}--\eqref{unless_3}.  
This is the content of Theorem  \ref{cor:entire} below.  
 Its statement  will be completed 
in  Section  \ref{sec:blocking} by
the explicit description of  the sets $Bl_1$ and $Bl_2$ 
referred to in the following statements.  
 Recall \eqref{wlog}.  We define the reduced parameters
\be\label{reduced_par}
r:=\frac{q}{p},\quad \mathfrak{a}:=\frac
{\gamma_0}{\gamma_0+\gamma_1}
\mbox{ if }\gamma_0+\gamma_1\neq 0
\ee
and set 
\be\label{def_r0}
r_0:=\frac{1-2\sqrt{-7+\sqrt{52}}}{1+2\sqrt{-7+\sqrt{52}}}=0,042\cdots
\ee
Due to \eqref{wlog}, we have $(\mathfrak{a},r)\in[0,1]\times[0,1]$. 
\begin{theorem}\label{cor:entire} 
(o) 
If $\gamma_0>0$ and $\gamma_1>0$, elements of $Bl_2$ are supported 
on the set ${\mathcal X}_2$. \newline\newline
(i)  Assume \eqref{unless_2}.
Then $\mathcal R=\mathcal R'=Bl_2=\emptyset$, hence
\be\label{charac_sym}
\mathcal I_e=\{\nu_\rho:\,\rho\in[0,2]\}
\ee
$\bullet$ Assume $q>0$. Then:  \newline\newline
(ii) Assume either: (a) $d_0/l_0=d_1/l_1>1$; or 
(b) $l_0=l_1=0$ and $d_0,d_1>0$. Then $Bl_2$ is nonempty 
and given by \eqref{def_bl2(ii)}. \newline\newline
(iii) 
There exists an open subset $\mathcal Z$ of 
$[0,1]\times[0,1]$, containing  $\{1/2\}\times (0,r_0)$, 
such that $\mathcal R=\mathcal R'=\emptyset$  
for every $(\mathfrak{a},r)\in\mathcal Z$.
In particular, if   $r\in(0,r_0)$,  
$d_1=\lambda d_0$ and $l_1=\lambda l_0$ with 
 $\lambda\in\R$  close enough to $1$, then \eqref{charac_ext} 
holds with $Bl_2$ as in (ii);  this yields a 
complete description of $\mathcal I_e$.  \newline\newline
$\bullet$ Assume now $q=0<p$.  Then a complete 
description of $\mathcal I_e$ can be obtained 
whenever $\gamma_0\neq\gamma_1$. More precisely: \newline\newline
(iv) (a) If $\gamma_0>0$ and $\gamma_1>0$, 
then $\mathcal R'=\{(0,1);(1,2)\}$; 
$\mathcal R$ is empty if $\gamma_0\neq\gamma_1$, 
or contained in $\{(3/2,1/2)\}$ 
if $\gamma_0=\gamma_1$. The set $Bl_1$ is given by 
\eqref {def_bl_3_2}. The set $Bl_2$ is empty unless $l_0=l_1=0$. 
(b) If  $l_0=l_1=0$, $Bl_2$ is given by \eqref{def_bl2(iv)(b)}. \newline\newline
(v) If $\gamma_1<0<\gamma_0$, then $\mathcal R'=\{(1,0),(1,2)\}$, 
$\mathcal R=Bl_2=\emptyset$. 
The set $Bl_1$ is given by \eqref{def_bl_5}. 
\newline\newline
(vi) If  $\gamma_0=0<\gamma_1$, then $\mathcal R'=\{(0,1)\}$,
$\mathcal R=Bl_2=\emptyset$. The set
$Bl_1$ is given by \eqref{def_bl_3_1}.
\end{theorem}
 In \textit{(ii)--(vi)} above, the measures in the sets $Bl_1$ 
and $Bl_2$ defined by \eqref{def_bl2(ii)}, \eqref {def_bl_3_2} 
and \eqref{def_bl2(iv)(b)}--\eqref{def_bl_3_1} are reversible. 
\begin{remark}\label{remark_sym}
In case (i), when $p=q$, the kernel defined by 
\eqref{eq:intensity_c} is symmetric. 
The result is then a particular case 
of the general picture (\cite{ligsym,liggett2012interacting})  
for symmetric exclusion processes, although 
our method of proof is different. However 
when in case (i) we have $p\neq q$, the two-lane  SEP  
 is not a symmetric exclusion process, and our result is new.
\end{remark}
\begin{remark}\label{remark_open}
For reader's convenience, we summarize here the questions 
left open by Theorems \ref{thm:Characterization_of_invariant_measures} 
and \ref{cor:entire}.
\begin{enumerate}
\item When $p=q$ and $\gamma_0+\gamma_1=0$, we can prove that outside the measures 
$\nu_\rho$, $\mathcal I_e$ may only contain shock measures, but we conjecture 
that there are no shock measures in $\mathcal I_e$. 
\item When $q>0$ and $\gamma_0+\gamma_1\neq 0$, we conjecture that there are 
no shock measures of amplitude $1$. We prove this when $q/p$ is small enough 
and $\gamma_0/(\gamma_0+\gamma_1)$ close enough to $1/2$.
\item When $q>0$, $p\neq q$  and $\gamma_0+\gamma_1=0$, we do not know if there are shock measures of amplitude $1$.
\item When $q=0$ and $\gamma_0=\gamma_1>0$,
we believe that $\mathcal R$ is empty but can only prove that it is contained in $\{(3/2,1/2)\}$
\item We conjecture that outside cases \eqref{unless}-- \eqref{unless_2}, if $q>0$ and  
\eqref{cond_nob2} fails, $Bl_2$ consists (up to horizontal translations) of two blocking measures. 
We can prove this  when $q>0$ and $d_0/l_0=d_1/l_1\in[1;+\infty]$. This conjecture is related 
to property (vii) of Proposition \ref{prop:extrema}.
\end{enumerate}
\end{remark}
 \subsubsection{Explicit blocking measures in Theorem  
 \ref{cor:entire}}\label{sec:blocking}
We  here complete the statement of Theorem \ref{cor:entire} 
by giving the explicit description of  $Bl_1$ and $Bl_2$ 
in each case. 
 For this, we need to recall blocking measures 
denoted hereafter by $\{\widehat{\mu}_n:\,n\in\Z\}$, 
which are reversible (\cite{Liggett1976}) for single-lane 
ASEP with jump rate $d$ to the right and $l$ to the left, 
cf. \eqref{eq:generator of the Exclusion}--\eqref{def_tasep}, 
where $d+l>0$ and $d-l>0$. These measures will be building 
blocks for  certain elements of  $\mathcal I_e$.  
For $l=0$,   $\widehat{\mu}_n$ is defined by 
\be\label{blocking_config}
\widehat{\mu}_n:=\delta_{\eta^*_n}
\quad\mbox{where}\quad\eta^*_n(x):={\bf 1}_{\{x> n\}} 
\ee
 where $\delta$ denotes the Dirac measure. In the sequel, 
 we shall also use notations \eqref{blocking_config} 
 by extension for $n=\pm\infty$. 
In \eqref{blocking_config},  $\eta^*_{-\infty}$ and 
$\eta^*_{+\infty}$ are respectively understood as 
the configuration with all $1$'s and the one with all $0$'s. 
When $l>0$,  $\widehat{\mu}_n$ is defined   as follows. First, set
\be\label{sol_rev_asep}
\rho_i^c:=\frac{c\left(\frac{d}{l}\right)^i}
{1+c\left(\frac{d}{l}\right)^i}
\ee
where $c>0$. The measure  $\mu_{\rho_.}=\mu_{\Z,\rho_.}$ 
(cf. Definition \eqref{alpha_x} and Remark \ref{rk_prod}) 
is supported on the set  $\mathcal X_1$ defined by \eqref{def_x0}. 
On this set,  
 a function $H_1$ can be defined by 
\be\label{def_H}
 H_1(\eta):= \sum_{x\leq 0}\eta(x)-\sum_{x>0}[1-\eta(x)]
\ee
 One can then define   the  probability  measure 
(which does not depend on the choice of $c$)
\be\label{cond_rev_meas}
\widehat{\mu}_n:={\mu}_{\rho^c_.}\left(
.\left|H_1(\eta)=n\right.
\right)
\ee
We can now give details for Theorem \ref{cor:entire}. \newline\newline 
\textbf{Case \textit{(ii)}.} We set
\be\label{def_bl2(ii)}
Bl_2:=\left\{ \breve{\nu}_{z}:z\in\mathbb{Z}\right\} 
\cup\left\{ \widehat{\nu}_{z}:z\in\mathbb{Z}\right\} 
\ee
where the measures $\breve{\nu}_{z}$ and $\widehat{\nu}_{z}$ 
are defined below distinguishing cases \textit{(ii), (a)} 
and \textit{(ii), (b)}:  \newline\newline
\textit{Case (ii), (a).} Let $\theta=d_0/l_0=d_1/l_1$. 
 Define 
\be\label{sol_rev_twolane}
\rho_{z,i}^c:=\frac{c\theta^z\left(\frac{p}{q}\right)^i}
{1+c\theta^z\left(\frac{p}{q}\right)^i},\quad (z,i)\in\Z\times W,  c>0  
\ee
 We consider the probability measure  $\mu_{\rho_.^c}$  
 on $\mathcal X$ under which the random variables 
 $\{\eta(z,i):\,x\in\Z,\,i\in W\}$ are independent,
and $\eta(z,i)$ is Bernoulli distributed with parameter $\rho_{z,i}^c$. 
\begin{remark}\label{remark_v}
The measures $\mu_{\rho_.^c}$ are analogues in this context 
of the 2-dimensional blocking measures constructed 
in \cite[Theorem 2]{ligd}, which are $v$-profile measures 
(for any $i\in\{0,1\}$) where 
$v=\left(\ln\frac{d_i}{l_i},\ln\frac{p}{q}\right)$.
\end{remark}
We define the following function on $\mathcal X_2$  (cf. \eqref{def_x2}): 
\be\label{def_H_2}
H_2(\eta):=
\sum_{ x\in V:\,  x(0)\leq 0}
\eta(x)-\sum_{ x\in V: \, x(0)>0}[1-\eta(x)]
\ee
Note that $\mathcal X_2$ is stable by the dynamics, 
and $H_2$ is a conserved quantity for the process on $\mathcal X_2$.
The measures involved in \eqref{def_bl2(ii)} are defined 
in the following lemma, proved 
 in Subsection \ref{subsec:proof_inv_3} along with Theorem \ref{cor:entire}. 
\begin{lemma}\label{def_blocking_h2}
The measure  $\mu_{\rho_.^c}$  is supported on $\mathcal X_2$, and
the measures defined below do not depend on $c>0$:
\begin{eqnarray}
\check{\nu}_n & := & \mu_{\rho^c_.}\left(
.\left|H_2(\eta)=2n\right.
\right),\quad 
\label{cond_rev_twolane}
\widehat{\nu}_n 
:= \mu_{\rho^c_.}\left(
.\left|H_2(\eta)=2n+1\right.
\right),\quad n\in\Z
\end{eqnarray}
These measures satisfy the relations
\begin{eqnarray}
\check{\nu}_n & = &  \tau_{n}\check{\nu}_0,\quad 
\label{shift_rev_twolane}
\widehat{\nu}_n 
=\tau_{n} \widehat{\nu}_0,\quad n\in\Z
\end{eqnarray}
\end{lemma}
\noindent\textit{Case (ii), (b).} 
  Let,  for $x\in V$,
 \begin{align*}
\breve{\eta}\left(x\right)  & =1_{\{x\left(0\right) > 0\}}\\
\widehat{\eta}^{0}\left(x\right) 
& =1_{\{x\left(0\right) > 0\}}
+1_{\{x=\left( 0,0\right)\}}
\\
\widehat{\eta}^{1}\left(x\right) 
& =1_{\{x\left(0\right) > 0\}}
+1_{\{x=\left( 0,1\right)\}}.
\end{align*} 
 We define the measures $\breve{\nu}_{0}$ and $\widehat{\nu}_{0}$
through
\be
\label{brevenu} \breve{\nu}_{0}  
 = \delta_{\breve{\eta}},\quad 
\widehat{\nu}_{0} 
 = \frac{q}{p+q}
\delta_{\widehat{\eta}^0}+\frac{p}{p+q}\delta_{\widehat{\eta}^1}
\ee
 We define also   
$\breve{\nu}_{z}=\tau_{ -z}\breve{\nu}_{0}$  
and $\widehat{\nu}_{z}=\tau_{ -z}\widehat{\nu}_{0}$
for every $z\in\mathbb{Z}$.\newline\newline
 \textbf{Cases \textit{(iv)--(vi)}.} 
 Using the blocking measures for single lane ASEP, 
 we define the following two-lane measures.  For $n\in\Z$, 
 we denote by $\nu^{\bot,+\infty,n}$  
and  $\nu^{\bot,n,-\infty}$  
the probability measures on $\mathcal X$ defined as follows.
Under $\nu^{\bot,+\infty,n}$, $\eta^0=\eta^*_{+\infty}$,  
see definition \eqref{blocking_config} (i.e. lane $0$ is empty), 
and $\eta^1\sim\widehat{\mu}_n$, 
where $\widehat{\mu}_n$ is given by \eqref{cond_rev_meas} 
with $l=l_1$ and $d=d_1$ if $l_1>0$, 
or by \eqref{blocking_config} if $l_1=0$  (where $\sim$ 
means equality in distribution).
Under $\nu^{\bot,n,-\infty}$, $\eta^1=\eta^*_{-\infty}$ 
(i.e. lane $1$ is full)  and $\eta^0\sim\widehat{\mu}_n$, 
where $\widehat{\mu}_n$ is given by \eqref{cond_rev_meas} 
with $l=l_0$ and $d=d_0$ if $l_0>0$, 
or by \eqref{blocking_config} if $l_0=0$.  In the following cases, 
whenever it is not empty, the set $Bl_1$ reduces to partial blocking measures; 
and whenever it is not empty, the set $Bl_2$ consists of 
blocking measures.   \newline\newline
\textbf{Case \textit{(iv)}}. \textit{(a)}  The  set  $Bl_1$ is given by 
\begin{eqnarray}\label{def_bl_3_2}
Bl_1 & := & \left\{\nu^{\bot,+\infty,n}:\,n\in\Z\right\}\cup 
\left\{\nu^{\bot,n,-\infty}:\,n\in\Z\right\}
\end{eqnarray}
\textit{(b).} 
Let $\mathbb{B}$ denote the set of  $(i,j)\in\Z^2$
 such that  $i\geq j$, and set 
 $\overline{\mathbb{B}}:=\mathbb{B}\cup\{(+\infty,n),(n,-\infty):\,n\in\Z\}$. 
For $(i,j)\in\overline{\mathbb{B}}$, let $\nu^{\bot,i,j}$
denote the Dirac measure 
supported on the configuration $\eta^{\bot,i,j}$ defined by
 (recalling  \eqref{blocking_config}) 
\be\label{def_etabot}
\eta^{\bot,i,j}(z,0)=\eta_i^*(z),\quad
\eta^{\bot,i,j}(z,1)=\eta^*_j(z)
\ee
for every $z\in\Z$.
The set $Bl_2$ is given by
\begin{eqnarray}
Bl_2 & := & \left\{
{\nu}^{\bot,i,j}:(i,j)\in\mathbb{B}
\right\}\label{def_bl2(iv)(b)}
\end{eqnarray}
\textbf{Case \textit{(v)}.} 
For $n\in\Z$, we denote by $\nu^{\bot,+\infty,n\leftarrow}$  
 the probability measure on $\mathcal X$ defined as follows.
 Recall 
the lane symmetry operator $\sigma$ defined by \eqref{sym_op}. 
Under $\nu^{\bot,+\infty,n\leftarrow}$, $\eta^0=\eta^*_{+\infty}$ 
 and $\sigma\eta^1\sim\widehat{\mu}_n$, where 
$\widehat{\mu}_n$ is given by  \eqref{cond_rev_meas} with $l=l_1$ and 
$d=d_1$ if $l_1>0$, or by \eqref{blocking_config} if $l_1=0$. 
The set $Bl_1$ is then given by 
\be\label{def_bl_5}
Bl_1  :=  \left\{\nu^{\bot,+\infty,n\leftarrow}:\,n\in\Z\right\}
\cup \left\{\nu^{\bot,n,-\infty}:\,n\in\Z\right\}
\ee
\textbf{Case \textit{(vi)}.} The set $Bl_1$ is given by
\be\label{def_bl_3_1}
Bl_1:=\left\{\nu^{\bot,+\infty,n}:\,n\in\Z\right\}
\ee
\subsection{Multilane SEP 
and rotational invariance}\label{subsec:rotation}
In this section,  we consider the general model  
defined by  \eqref{eq:generator of the Exclusion} 
with  \eqref{structure},
 \eqref{restrict_kernel}  and \eqref{restrict_2}.  
Without loss of generality, we may consider $W=\{0,\ldots,n-1\}$.  
We are interested in a generalization of the two-lane model with $p=q$
 (cf. \eqref{restrict_1}).  
To this end, we introduce the following   assumption. 
\begin{assumption}\label{assumption_ker} 
$W=\mathbb{T}_n$  is a torus, 
and $q(.,.)$ is an irreducible translation-invariant kernel, 
that is $q(i,j)=Q(j-i)$ for some function
$Q:\mathbb{T}_n\to[0,+\infty)$. 
\end{assumption}  
For $\rho\in[0,n]$, we denote by 
$\nu_{\rho}$ the product measure on $\mathcal X$ such that 
\be\label{def_inv_gen}
\nu_{\rho}\left\{
\eta(z,i)=1
\right\} =\frac{\rho}{n}
\ee
for every $(z,i)\in \Z\times W$.  
For Theorems \ref{thm:characterization_lemma}, 
 \ref{thm:Characterization_of_invariant_measures} 
and \ref{cor:entire}, 
the scheme of proof  laid out in  Sections \ref{sec:proofs} 
 to \ref{subsec:proofs_lemmas}   carries over to 
 the  multilane model. Here, since  $W=\mathbb{T}_n$, 
in addition to the shift operator $\tau$ along $\Z$ already considered, 
we can consider the translation operator $\tau'$ along $W$.
Following \cite[page 2309]{ligd}, we shall call a probability measure on 
$\mathcal X$ \textit{rotationally invariant} 
if it is invariant by $\tau'$. The 
\textit{open question 1. for the ladder process} raised in \cite{ligd} 
is whether,  when $d_i$ and $l_i$ are independent of $i$ 
(i.e. the horizontal dynamics is the same on each lane),   
\textit{all} invariant 
measures are rotationally invariant. 
We give a positive answer to this question
 in item  \textit{(iii)}  of the following theorem. 
\begin{theorem}\label{thm:Characterization_of_invariant_measures_gen}
Under Assumption \ref{assumption_ker}, the following hold: \newline\newline
  (o)    We have
$(\mathcal I\cap\mathcal S)_e=
\{\nu_\rho,\,\rho\in[0,n]\}$. \newline\newline  
  (i)    For $k=1,\ldots,n$, let 
$(\rho^-_k,\rho^+_k)=
\left(\frac{n-k}{2},\frac{n+k}{2}=n-\rho^-_k\right)$.
Then: (a)
\be\label{structure_multi}
{\mathcal I}_e=\{\nu_\rho:\,\rho\in[0,n]\}\cup 
\bigcup_{k=1}^{n}\mathcal I_k\ee
where $\mathcal I_k$ is a (possibly empty) set of 
$(\rho^-_k,\rho^+_k)$-shock measures of amplitude $k$, 
which contains at most (up to horizontal translations) $k$ measures. 
(b) If $\gamma_i>0$ for all $i$, $\mathcal I_n$ 
is supported on the set $\mathcal X_n$ defined 
by the right-hand side of \eqref{def_x2}. 
(c) If $d_i/l_i$ does not depend on $i$, 
$\mathcal I_n$ consists (up to horizontal translations) 
of $n$ explicit blocking measures
$\nu_i$ defined below for $i=0,\ldots,n-1$. \newline\newline
  (ii)   If  $\gamma_i:=d_i-l_i=0$ for all $i\in W$, then 
$\mathcal I_e=\{\nu_\rho:\,\rho\in[0,n]\}$. \newline\newline
  (iii)    If $d_i$ and $l_i$ do not depend on $i$,  
any invariant measure is rotationally invariant. 
\end{theorem}
 The blocking measures in \textit{  (i),   (c)} are defined as 
 in cases \textit{(ii), (a)} and \textit{(ii), (b)} 
  of Theorem \ref{cor:entire}: \newline\newline
{\em First case.}
 If $l_i>0$ for all $i$,  we define $\rho_{z,i}^c$ 
as in \eqref{sol_rev_twolane}, with $\theta=d_i/l_i$ 
and $p/q$ replaced by $1$.
For $i=0,\ldots,n-1$, we define the conditioned measures 
(independent of the choice of $c>0$ as in \eqref{cond_rev_twolane})
\begin{eqnarray}
{\nu}_{i} & := & \mu_{\rho^c_.}\left(
.\left|H_n(\eta)=i\right.
\right)
\label{cond_rev_multilane}
\end{eqnarray}
where $H_n(\eta)$ is defined as the right-hand side of \eqref{def_H_2}. \newline\newline
{\em Second case.} If $l_i=0<d_i$ for all $i$, 
we define the configurations 
\be\label{blocking_config_multi} 
\eta_A:={\bf 1}_{\{x(0)\geq 0\}}+{\bf 1}_{\{x(0)
=-1,\,x(1)\in A\}},\quad A\subset\{0,\ldots,n-1\}
\ee
Then $\nu_i$ is the law of a random configuration $\eta_A$, 
where $A$ is a uniformly chosen random subset of 
$\{0,\ldots,n-1\}$ such that $|A|=i$: 
\be\label{blocking_meas_multi}
\nu_i:=\left(\ba{l}n\\i\ea\right)^{-1}
\sum_{A\subset\{0,\ldots,n-1\}:\,|A|=i}\delta_{\eta_A}
\ee
 \section{Proof of \thmref{characterization_lemma}}
\label{sec:proofs}
 The proof of  \thmref{characterization_lemma} 
mainly adapts the scheme of \cite[Theorem 1.1]{Liggett1976} 
to our model. However  when $q=l_0=l_1=0$,  additional arguments 
are required because the kernel  \eqref{eq:intensity_c} does not satisfy 
usual irreducibility assumptions. \\
First,  in Subsection \ref{subsec:par},  
we show how to parametrize the set $\mathcal F$ in 
\eqref{def_f_always} by the global density  
over the two lanes  and establish 
invariance of the associated product measures  
given in \eqref{eq:charac}.  Next, we introduce 
coupling prerequisites  in Subsection \ref{subsec:couple},  
and complete the proof of characterization  
in Subsection \ref{subsec:complete}. 
\subsection{Parametrization and proof of invariance}\label{subsec:par}
 The following lemma will lead to the parametrization \eqref{sameset}. 
  Note that to be exhaustive we choose here to include case 2 in 
 Lemma \ref{lemma_phi} below although it could be omitted 
 thanks to the symmetry restriction \eqref{wlog}. 
\begin{lemma}\label{lemma_phi}
(i)  The mapping 
$\psi:\mathcal F\to [0,2];\, 
 (\rho_0,\rho_1)\mapsto\psi(\rho_0,\rho_1):=\rho_0+\rho_1$,
is a bijection.  \newline\newline
 (ii)   Its inverse is of the form
$\psi^{-1}(\rho)=(\widetilde{\rho}_0(\rho),\widetilde{\rho}_1(\rho))$,
where  $\widetilde{\rho}_1(\rho):=\rho-\widetilde{\rho}_0(\rho)$,  
and  the function $\rho\mapsto\widetilde{\rho}_0(\rho)$ 
is given by the following formulae:\newline\newline
\textit{Case 1a.} $pq\neq 0$, $p\neq q$. Then, for $r=q/p$,  cf. \eqref{reduced_par}, 
\begin{align}
\widetilde\rho_{0}\left(\rho\right) &:=
\frac{\rho}{2}+ \frac{r+1-\sqrt{(r+1)^{2}+\rho(r-1)^{2}(\rho-2)}}{2(r-1)},
\label{eq:inverse_density}
\end{align}
\textit{Case 1b.} $p=q\neq 0$. Then 
\be\label{tilderho_1b}
\widetilde{\rho}_0(\rho):=\frac{\rho}{2}
\ee
\textit{Case 2.} $p=0<q$. Then
\be\label{tilderho_2a}
\widetilde{\rho}_0(\rho):=\min(\rho,1)
\ee
\textit{Case 3.} $q=0<p$. Then
\be\label{tilderho_2b}
\widetilde{\rho}_0(\rho):=\max(\rho-1,0)
\ee
\end{lemma}
\begin{remark}\label{remark_lemma_phi}
In Lemma \ref{lemma_phi},
the formulae in  \textit{(ii)}  imply that for $pq>0$, $\widetilde{\rho}_i(\rho)$ 
strictly increases from $0$ to $1$ as $\rho$ increases from $0$ to $2$, 
and $\widetilde{\rho}_i\in C^1([0,2])$.
\end{remark}
Next we define
\begin{equation}\label{def_nurho}
\nu_{\rho}:=
 \nu^{
\widetilde{\rho}_0(\rho),\widetilde{\rho}_1(\rho)
}
\end{equation}
By \eqref{def_nurho} and \eqref{eq:two-rate Bernoulli measure}, we have
 (recalling definition \eqref{config_lane}),   for every $i\in\{0,1\}$, 
\be\label{exp_bernoulli}
\Exp_{\nu_{\rho}}[\eta^i(0)]=\widetilde{\rho}_i(\rho)
\ee
 which implies \eqref{eq:density at a point}. 
\begin{remark}\label{rem:increasing set of measures} 
It follows from \textit{(ii)} of Lemma \ref{lemma_phi} 
that the measure
$\nu_{\rho}$ is weakly continuous and stochastically 
nondecreasing with respect to $\rho$.
Namely, if $\rho < \rho'$ then $\nu_{\rho}\le\nu_{\rho'}$.
\end{remark}
\begin{proof}[Proof of Lemma \ref{lemma_phi}]
We have to prove that,  for $\rho\in[0,2]$,  
the equation $\rho_0+\rho_1=\rho$ 
has a unique solution $(\rho_0,\rho_1)\in\mathcal F$; 
 then we define  $\rho_i=\widetilde{\rho}_i(\rho)$ for
$i\in\{0,1\}$.
 For $s>0$, we define
a mapping $\phi_s$ from $[0,1]$ to $[0,1]$ by 
\be\label{def_phi}
\phi_s(\rho_0):=
\frac{s\rho_0}{1-\rho_0+s\rho_0},\quad\forall\rho_0\in[0,1]
\ee 
One can then distinguish the following cases for $\mathcal F$: \\ \\
\textit{Case 1.} $pq\neq 0$.  Then
\be\label{def_f_gen}
\mathcal F:=\left\{
(\rho_0,\rho_1)\in[0,1]^2:\,\rho_1= \phi_{p/q}(\rho_0) 
\right\} 
\ee
\textit{Case 2.} $p=0<q$.  Then 
\be\label{def_f_p0}
\mathcal F:=\left(
[0,1]\times\{0\}
\right)\cup
\left(
\{1\}\times[0,1]
\right)
\ee
\textit{Case 3.} $q=0<p$.   Then 
\be\label{def_f_q0}
\mathcal F:=\left(
[0,1]\times\{1\}
\right)\cup
\left(
\{0\}\times[0,1]
\right)
\ee
Equalities \eqref{tilderho_2a}--\eqref{tilderho_2b} follow from 
\eqref{def_f_p0}--\eqref{def_f_q0}. 
For  \eqref{eq:inverse_density}--\eqref{tilderho_1b}, using 
\eqref{def_f_gen}, we equivalently show that,  for $\rho\in[0,2]$,  
the equation 
\be\label{eq_eq}
\rho_0+\phi_{p/q}(\rho_0)=\rho
\ee
 has a unique solution $\rho_0=:\widetilde{\rho_0}(\rho)\in[0,1]$ and 
deduce $\widetilde{\rho}_1(\rho)$.  If $p=q>0$, \eqref{def_f_gen} 
with $s=1$ yields $\phi_1(\rho_0)=\rho_0$, whence \eqref{tilderho_1b}.
If $p\neq q$, $p>0$ and $q>0$,  
 \eqref{def_f_gen}  and \eqref{eq_eq} yield 
a quadratic equation for $\rho_0$, and \eqref{eq:inverse_density}
is  its  unique solution in $[0,1]$.
\end{proof}
 The following lemma shows that the measures in  
 \eqref{eq:charac} of
 \thmref{characterization_lemma} are indeed 
 extremal translation invariant
 and invariant probability measures.  
 Stationarity can be derived from \cite[Theorem 1]{ligd}, 
 but we give an independent proof
based on prior knowledge of invariance 
along horizontal and vertical layers.  
\begin{lemma}
\label{lem:invariant_measures_for_the_asymetric_ladder}
Let $(\rho_0,\rho_1)\in\mathcal F$. Then 
$\nu^{\rho_0,\rho_1}\in\left(\mathcal{I}\cap\mathcal{S}\right)_e$. 
\end{lemma}
\begin{proof}[Proof of Lemma \ref{lem:invariant_measures_for_the_asymetric_ladder}]
Let $f$   be a cylinder function
on $\mathcal{X}$. 
Note that  
the generator \eqref{eq:generator of the Exclusion} 
has the following structure, 
\be\label{structure_gen}
L=\sum_{i\in  W}L_h^i+\sum_{z\in\Z}L_v^z
\ee 
where,  for $i\in W$ and $z\in\Z$, 
\begin{align*}
L_{h}^if\left(\eta\right) & =\sum_{z\in\mathbb{Z}}\sum_{z'\in\{z+1,z-1\}}
p\left((z,i),(z',i)\right)
\eta\left((z,i)\right)\left(1-\eta\left((z',i)\right)\right)\\
&  \qquad\qquad\qquad\qquad\qquad
\times \left(f\left(\eta^{(z,i),(z',i)}\right)-f\left(\eta\right)\right) \\
L_{v}^zf\left(\eta\right) & =\sum_{i,j\in W}p\left((z,i),(z,j)\right)
\eta\left((z,i)\right)\left(1-\eta\left((z,j)\right)\right)
\left(f\left(\eta^{(z,i),(z,j)}\right)-f\left(\eta\right)\right)
\end{align*}
 In other words, $L_{h}^i$,  acting only on $\eta^i$, 
and being translation invariant along the $\Z$ direction, 
describes the evolution of the process on $\mathbb{L}_{i}$, 
 which is the one of a (single-lane)  SEP;  while 
$L_{v}^z$, acting only on $\{z\}\times W$,  
 describes the motion of particles along $\{z\}\times W$, 
 that is,  the displacements
from one lane to another  at a fixed spatial location $z$. \newline
The statement $\nu^{\rho_0,\rho_1}\in\mathcal S$ holds because 
$\nu^{\rho_0,\rho_1}$ is a product Bernoulli measure whose  parameters
are  uniform in the $\Z$-direction.
Considering \eqref{structure_gen}, to prove that  $\nu^{\rho_0,\rho_1}$ 
belongs to $\mathcal{I}$,  it is enough to show that,  
for $i\in W$ and $z\in\Z$,  
\begin{equation}\label{eq:enoughtoshow_0}
\int L_h^{i}f\left(\eta\right)d\nu^{\rho_0,\rho_1}\left(\eta\right) =0
\end{equation}
and
\begin{equation}\label{eq:enoughtoshow}
\int L_{v}^zf\left(\eta\right)d\nu^{\rho_0,\rho_1}\left(\eta\right)=0.
\end{equation}
Let us write,  for a fixed $i\in W$,   
\be\label{decomp_product}\eta=(\eta^i,\eta'),
\quad \nu^{\rho_0,\rho_1}(d\eta)=\nu^i(d\eta^i)\otimes \nu'(d\eta')\ee
where $\eta'$ denotes the restriction of $\eta$ to lanes other than $i$. 
Note that $\nu^i$ is invariant for $L_h^i$ because $L_h^i$ 
is the generator 
of a  single-lane  SEP  on $\mathbb{L}_i$
and $\nu^i$ is a homogeneous product Bernoulli measure.
Since $L_h^i$ acts only on $\eta^i$, we have
\be\label{only_acts}
\int L_h^{i}f\left(\eta\right)d\nu^{\rho_0,\rho_1}\left(\eta\right)
  =  \int\left(
\int L_h^i [f(.,\eta')](\eta^i,\eta')d\nu^i(\eta^i)
\right)
d\nu'(\eta')=0
\ee
%
%
This establishes \eqref{eq:enoughtoshow_0}. 
\begin{remark}\label{rem:only_acts} 
 In \eqref{only_acts}, the notation $L_h^i [f(.,\eta')](\eta^i,\eta')$
 means that we apply the generator $L_h^i$  to the function 
 $\eta^i\mapsto f(\eta^i,\eta')$ for constant $\eta'$.  Though $L_h^i$ 
 acts only on $\eta^i$, the resulting function depends both on $\eta^i$ 
 and $\eta'$ (in the same way as a partial derivative operator acts only 
 on one variable but yields a function depending on all variables: here 
 $\eta^i$ plays a role similar to the variable with respect to which one differentiates). 
\end{remark}
We can similarly write,  for a fixed $z\in\Z$,  
\be\label{decomp_product_2}
\eta=(^z\eta,\eta''),\quad \nu_\rho(d\eta)=\,
^z\nu(d\,{^z\eta})\otimes\nu''(d\eta'')
\ee
 where $^z\eta$ is the restriction of $\eta$ to $\{z\}\times W$, 
 and $\eta''$ its restriction to the complement of $\{z\}\times W$.
So,  to prove \eqref{eq:enoughtoshow}, it is enough to prove that
 $^z\nu$ is invariant for $L_v^z$.
The latter is the generator of a simple exclusion process 
on $\{z\}\times W$. 
The invariance of $^z\nu$ follows from  Proposition \ref{prop:lig}  
applied to $S=\{z\}\times W$, 
$\pi((z,0),(z,1))=p$, $\pi((z,1),(z,0))=q$,
and definition \eqref{def_f_always} of $\mathcal F$. 
 Finally, since  $\nu^{\rho_0,\rho_1}\in\mathcal S$ 
is a homogeneous product measure,
it is spatially ergodic, that is extremal in $\mathcal S$, 
and thus also in $\mathcal I\cap\mathcal S$. 
\end{proof}
\subsection{Coupling, attractiveness 
and discrepancies} \label{subsec:couple}
 Let us first recall these properties for a general SEP; 
 we refer to 
\cite[Chapter VIII, Section 2]{liggett2012interacting} 
for details.  \newline\newline
 \textit{Coupling.} We recall the so-called
\textit{Harris graphical representation} (\cite{Harris72}). 
Suppose $\left(\Omega,\mathcal{F},\mathbb{P}\right)$
is a probability space that supports a family 
of independent Poisson processes
$\mathcal{N}=\left\{ \mathcal{N}_{\left(x,y\right)}:
\left(x,y\right)\in V\times V\right\} $
where $\mathcal{N}_{(x,y)}$ 
has intensity $p\left(x,y\right)$.  For a given $\omega \in \Omega$,
we let the  process evolve  according to the following rule:
if there is a particle at site $x\in V$ at time $t^{-}$ where 
 $t\in\mathcal{N}_{\left(x,y\right)}$, 
it will attempt to jump to site $y$. The attempt is suppressed if
at time $t^{-}$ site $y$ is occupied. \newline 
The graphical construction allows to couple the evolutions from
different initial configurations through 
\textit{basic coupling}, that is, by
using the same Poisson processes for them.  
In particular, if $(\eta_t)_{t\geq 0}$ and $(\xi_t)_{t\geq 0}$ 
are  two processes coupled in this way,  $(\eta_t,\xi_t)_{t\ge 0}$ 
is a Markov process on $\mathcal X\times\mathcal X$ 
whose generator $\widetilde{L}$ is given by 
\begin{eqnarray}
\widetilde{L}f(\eta,\xi) & := & 
\sum_{x,y\in V}p(x,y)[\eta(x)(1-\eta(y))\wedge\xi(x)(1-\xi(y))]\left[
f(\eta^{x,y},\xi^{x,y})-f(\eta,\xi)
\right]\nonumber
\\
& + & \sum_{x,y\in V}p(x,y)[\eta(x)(1-\eta(y))-\xi(x)(1-\xi(y))]^+\left[
f(\eta^{x,y},\xi)-f(\eta,\xi)
\right]\nonumber\\
& + & \sum_{x,y\in V}p(x,y)[\eta(x)(1-\eta(y))-\xi(x)(1-\xi(y))]^-\left[
f(\eta,\xi^{x,y})-f(\eta,\xi)
\right]\nonumber\\
& & 
\label{coupled_gen}
\end{eqnarray}
We shall denote by $(\widetilde{S}_t)_{t\geq 0}$ 
the semigroup generated by $\widetilde{L}$,
 by $\widetilde{\mathcal{I}}$  the set of invariant  probability  
measures for $\widetilde{L}$,
 by $\widetilde{\mathcal{S}}$  the set of probability measures on 
$\mathcal X\times\mathcal X$ that are
invariant with respect to translations along $\Z$,  and by
$\widetilde{\mathbb{E}}_{\widetilde\mu}$ 
the expectation for the coupled process 
with initial distribution a probability measure  $\widetilde\mu$ on 
$\mathcal X\times\mathcal X$.  \newline\newline
\textit{Attractiveness.} There is a natural 
partial order on $\mathcal{X}$, 
namely, for $\eta,\xi\in\mathcal{X}$,
\be\label{eq:orderconf}
\eta\leq\xi\quad\mbox{ if and only if }
\quad\forall x\in V,\, \eta\left(x\right)\leq\xi\left(x\right)
\ee 
We shall write $\eta<\xi$ if $\eta\leq\xi$ and $\eta\neq\xi$.\\
 If $\eta\leq\xi$ or $\eta\geq\xi$, we say that $\eta$ and $\xi$ 
are \textit{ordered} configurations.  \newline\newline
The order \eqref{eq:orderconf} endows an order 
on the set  $\mathcal{M}_{1}$  
 in the following way. A function $f$ on $\mathcal{X}$
is said to be \textit{increasing} if and only if $\eta\leq\xi$ implies 
$f\left(\eta\right)\leq f\left(\xi\right)$.
For two probability measures $\mu_{0},\mu_{1}$ on $\mathcal{X}$,
we write $\mu_{0}\leq\mu_{1}$ if and only if for every increasing
function $f$ on $\mathcal{X}$ we have 
$\int fd\mu_{0}\left(\eta\right)\leq\int fd\mu_{1}\left(\eta\right)$.
 We shall write $\mu_1<\mu_2$ if 
 $\mu_1\leq\mu_2$ and $\mu_1\neq\mu_2$.  
 We say $\mu_1$ and $\mu_2$ are \textit{ordered} 
 if $\mu_1\leq\mu_2$ or $\mu_2\leq\mu_1$. 
 In particular, $\mu_1\leq\mu_2$ 
 if there exists a measure $\widetilde{\mu}(d\eta,d\xi)$ 
 with marginals $\mu_1(d\eta)$ and $\mu_2(d\xi)$ 
(that is a {\em coupling} of $\mu_1$ and $\mu_2$) 
supported on 
$\{(\eta,\xi)\in\mathcal X\times\mathcal X:\,\eta \leq\xi\}$; 
such a coupling is called an \textit{ordered coupling}.  \newline\newline
The basic coupling shows that
the simple exclusion process is \textit{attractive}, 
that is, the partial order
\eqref{eq:orderconf} is conserved by the dynamics. 
In other words, 
\be\label{eq:attra}
\forall \eta_0,\xi_0\in\mathcal{X}, \, 
\eta_0\le\xi_0 \Rightarrow \forall t\ge 0,\,
\eta_t\le\xi_t \,\mbox{ a.s.}
\ee 
 This implies, for two probability measures $\mu,\nu$ on $\mathcal X$,
\be\label{attractive}
\mu\leq\nu\Rightarrow \mu S_t\leq\nu S_t
\ee
\textit{Discrepancies.}  If 
$(\eta,\xi)\in\mathcal X\times\mathcal X$, we say that at  
$x\in V$  there is 
an $\eta$ \textit{discrepancy} if $\eta(x)>\xi(x)$, 
a $\xi$ discrepancy 
if $\eta(x)<\xi(x)$, a coupled particle  
if $\eta(x)=\xi(x)=1$, 
a hole if $\eta(x)=\xi(x)=0$.  An $\eta$ and a $\xi$ 
discrepancy are called \textit{opposite discrepancies}.
 The evolution of the coupled process can be 
formulated as follows. At a time 
$t\in\mathcal N_{(x,y)}$,  a discrepancy  
or a coupled particle at $x$ exchanges with a hole at $y$; 
a coupled particle at $x$ 
exchanges with a discrepancy at $y$; if there is a pair of 
opposite  discrepancies  at $x$ and $y$, 
they are replaced by a hole at $x$ 
and a coupled particle at $y$.  
We call this a \textit{coalescence}. This shows that 
no new discrepancy can ever be created.  \newline\newline
Given an initial  tagged discrepancy, we may follow its motion over time.
We state in this context a classical 
\textit{finite propagation} property for discrepancies.
Single-lane versions of this statement can be found e.g. 
in  \cite[Lemma 3.1]{Bramson2002}  or 
\cite[Lemma 3.1, Lemma 3.2]{Bahadoran2002}.
Proofs are similar for the two-lane model.  
\begin{proposition}\label{prop:prop} 
There exist constants $\speed,C,C'>0$ such that the following holds.
Assume $(\eta_t)_{t\geq 0}$ and $(\xi_t)_{t\geq 0}$ 
are two coupled two-lane  SEP's with at least one discrepancy at time $0$.
Let $X_t=(X_t(0),X_t(1))\in \Z\times W$ denote 
the position of a tagged discrepancy 
at time  $t$. Then: \\
(i) Outside probability $e^{-Ct}$, it holds that 
$|X_t(0)-X_0(0)|\leq (1+\sigma)t$.  \newline
(ii) Similarly, if we assume  
$\eta_0(z,i)=\xi_0(z,i)$ for all $z\in[a,b]$ and $i\in\{0,1\}$, 
where $a,b\in\Z$ and $a<b$,
then outside probability $e^{-C't}$,  $\eta_t(z,i)=\xi_t(z,i)$ 
for all $z\in[a+\speed t,b-\speed t]$ and $i\in\{0,1\}$.
\end{proposition}
 \subsection{Irreducibility and discrepancies} \label{subsec:irred}
 As for general SEP, a crucial tool 
 to prove \thmref{characterization_lemma}
is an irreducibility property. We thus begin 
with the following definitions and properties. \newline 
For $x,y\in V$ such that $x\neq y$, and  $n\in\N$, 
 we write $x\stackrel{n}{\rightarrow}_p y$ if there exists  
 a path $(x=x_0,\ldots,x_{n-1}=y)$ of length $n$ such that 
 $p(x_k,x_{k+1})>0$ for $k=0,\ldots,n-1$. We write $x\rightarrow_p y$ 
 if there exists $n\in\N$ such that
$x\stackrel{n}{\rightarrow}_p y$. We omit mention of $p$
whenever there is no ambiguity on the kernel. 
We say $x$ and $y$ are $p$-connected if  $x\rightarrow_p y$
or $y\rightarrow_p x$.  We say two configurations $\eta,\xi$ in 
$\mathcal X$ are \textit{$p$-ordered} 
if there exists no $(x,y)\in V\times V$ 
such that $x$ and $y$ are $p$-connected and $(\eta,\xi)$ 
has  opposite discrepancies  at $x$ and $y$.  
\begin{definition}\label{def_irred}
 The kernel $p(.,.)$ is \textit{weakly irreducible} if,    
for every $(x,y)\in V\times V$ such that $x\neq y$, 
$x$ and $y$ are $p$-connected.
\end{definition}
The above notion is weaker than the \textit{usual} 
irreducibility property, 
for which a stronger notion of $p$-connection is required, namely 
$x\rightarrow_p y$ \textit{and} $y\rightarrow_p x$. 
 For instance, the kernel \eqref{def_tasep} is irreducible 
 if and only if $dl>0$; 
if $dl=0$, it is weakly irreducible but not irreducible. 
For our  two-lane and multilane  models, we need the following lemma.
\begin{lemma}\label{lemma_wi}
%
%
(i) The two-lane kernel $p(.,.)$ given by 
\eqref{eq:intensity_c} is weakly irreducible   
except when $q=0$ and both lanes are totally asymmetric 
in the same direction, that is 
\be\label{non_wi_case}d_0l_0+d_1l_1=0<d_0d_1+l_0l_1\ee
(ii) The multilane kernel $p(.,.)$ given by \eqref{restrict_kernel}
is weakly irreducible under assumption \ref{assumption_ker}.
\end{lemma}
\begin{proof}[Proof of Lemma \ref{lemma_wi}]
\mbox{}\newline\newline
\textit{Proof of (i).} Let $x,y\in\Z$ such that $x\neq y$. 
We need to go either 
from $(x,0)$ to $(y,1)$, or  from $(y,1)$ to $(x,0)$, 
 with the kernel $p(.,.)$.  \\
 \textit{(a)} Assume first $q>0$. Since the kernel 
 \eqref{def_tasep} is weakly irreducible, the horizontal kernel 
 on lane $0$
can either go from $(x,0)$ to $(y,0)$ or from $(y,0)$ to $(x,0)$. 
In the former case, since $p>0$, we go  from $(y,0)$ to $(y,1)$  
with the vertical kernel. 
In the latter,  since $q>0$,  we can go 
from $(y,1)$ to $(y,0)$ vertically and 
then from $(y,0)$ to $(x,0)$ horizontally. \\
 \textit{(b)} Assume now $q=0$. 
If the two lanes are totally asymmetric in the same direction, 
say e.g. $l_0=l_1=0<d_0d_1$, and $x>y$, we can neither 
go from $(x,0)$ to $(y,1)$ 
(because $l_0+l_1=0$), nor from $(y,1)$ to $(x,0)$ 
(because $q=0$); otherwise, 
we have either $d_0l_1>0$ or  $d_1l_0>0$,  say for instance the former. 
Then  we can go from $(x,0)$ to $(y,1)$ via $(y,0)$ if $x<y$, 
 or  via $(x,1)$ if $x>y$. \newline\newline
\textit{Proof of (ii).} Let $x,y\in\Z$ such that $x\neq y$ 
and $i,j\in W$ 
such that $i\neq j$. Since the vertical kernel $q(.,.)$ is irreducible, 
the same argument as in case \textit{(a)} of \textit{(i)} shows that 
we can either go from $(x,i)$ to $(y,j)$ or from $(y,j)$ to $(x,i)$.
\end{proof}
 The next lemma gives a characterization of $p$-ordered configurations. 
This requires the following definition.  Without loss of generality 
we assume in the sequel that $d_0d_1>0$ and $l_0l_1\geq 0$. 
We leave the reader symmetrically formulate  Definition \ref{def_supinf}
 and Lemma \ref{lemma:irred}  in the case $d_0d_1\geq 0$ and $l_0l_1>0$. 
\begin{definition}\label{def_supinf}
For $(\eta,\xi)\in\mathcal X\times\mathcal X$, we write $\eta\supinf\xi$ 
if and only if there exist $x,y\in\Z$ such that $x<y$ 
and the following hold: 
(a)  there are  opposite discrepancies  at $(x,1)$ and $(y,0)$;
(b) $\eta^0\leq\xi^0$ and $\eta^1\geq \xi^1$ 
if the discrepancy at $(x,1)$ 
is an $\eta$ discrepancy; or
$\eta^0\geq\xi^0$ and $\eta^1\leq \xi^1$ 
if the discrepancy at $(x,1)$ is a  $\xi$ discrepancy;
(c) There is no discrepancy at $(z,1)$ if $z>x$, 
nor any discrepancy at $(z,0)$ if $z<y$.\\
 We define
\be\label{def:Esupinf}
E_{\supinf}:=
\{(\eta,\xi)\in{\mathcal X}\times{\mathcal X}:\,\eta\supinf\xi\}
\ee 
\end{definition}
\begin{lemma}\label{lemma:irred}
For the kernel $p(.,.)$ in \eqref{eq:intensity_c},  under \eqref{wlog}, 
 we have the following:  \newline
(i)  Unless $q=0$ and $l_0=l_1=0$,  two configurations 
$\eta$ and $\xi$ are $p$-ordered if and only if 
they are ordered, i.e. $\eta\leq\xi$ or $\xi\leq\eta$.\newline
(ii) If   $q=0$  and $l_0=l_1=0$,  two configurations 
$\eta$ and $\xi$ are $p$-ordered 
if and only if  either  they are ordered, or  $\eta\supinf\xi$. 
\end{lemma}
\begin{proof}[Proof of Lemma \ref{lemma:irred}]
 Let $\eta$ and $\xi$ be two  $p$-ordered configurations. 
 Note that  two configurations are ordered  (see \eqref{eq:orderconf})  
if and only if they have no pair of opposite discrepancies.
 If $pq\neq 0$ or $l_0+l_1>0$,  because of \eqref{wlog}, 
any two distinct  
points  of  $V$  are $p$-connected, hence $\eta$ and $\xi$ are ordered. 
Assume $q=0<p$. First we note that for all $x,y\in \Z$, $(x,0)$ 
and $(y,0)$ are $p$-connected, 
and so are $(x,1)$ and $(y,1)$. Thus $\eta^i$ and $\xi^i$ are ordered. 
If the ordering is the same, then $\eta$ and $\xi$ 
are ordered. Otherwise, 
there exists a pair of opposite discrepancies, one at  $(\mathbf x,1)$ 
and one at $(\mathbf y,0)$ for $\mathbf x,\mathbf y\in\Z$. 
We must have $\mathbf x<\mathbf y$, otherwise $(\mathbf x,1)$ 
and $(\mathbf y,0)$ are $p$-connected. 
The ordering on each lane is imposed by the nature 
of the discrepancies at 
$(\mathbf x,1)$ and $(\mathbf y,0)$.
Assume for instance that there is an $\eta$ discrepancy 
at $(\mathbf x,1)$ 
and a $\xi$ discrepancy at $(\mathbf y,0)$. Then $\eta^0\leq \xi^0$ and 
$\eta^1\geq \xi^1$. For every $z>\mathbf y$, since $\eta^1\geq\xi^1$, 
we have $\eta^1(z)=\xi^1(z)$ or an $\eta$ discrepancy at $(z,1)$. 
The latter is ruled out because
$(\mathbf y,0)$ and $(z,1)$ are $p$-connected. Similarly there can be 
no discrepancy at $(z,0)$ if $z<\mathbf x$. 
We can then  redefine  $x$ as the location of the rightmost 
$\eta$ discrepancy 
on lane $1$, and $y$ denotes the location of the leftmost 
$\xi$ discrepancy on lane $0$. 
\end{proof}
 \subsection{Proof of characterization} \label{subsec:complete}
 The next two results will enable us to deal with discrepancies in 
the proof of \thmref{characterization_lemma}. 
\begin{lemma}\label{lemma:irred2}
Let $\widetilde{\nu}\in(\widetilde{\mathcal I}\cap\widetilde{\mathcal S})$. 
If   $l_0=l_1=q=0$,  then $\widetilde{\nu}(E_{\supinf})=0$. 
\end{lemma}
\begin{proof}[Proof of Lemma \ref{lemma:irred2}]
We  define the following random variables 
taking values in $\Z\cup\{\pm\infty\}$:
\begin{eqnarray}
X=X(\eta,\xi) & := & \sup\{x\in\Z:\,\eta^1(x)\neq\xi^1(x)\}\label{rightmost}\\
Y=Y(\eta,\xi) & := & \inf\{x\in\Z:\,\eta^0(x)\neq\xi^0(x)\}\label{leftmost}
\end{eqnarray}
with the convention $\sup\emptyset=-\infty=-\inf\emptyset$. 
That is, $X$ is the location (if it exists) 
of the rightmost discrepancy on lane $1$.
Indeed on  $E_{\supinf}$,  
we have $X(\eta,\xi)\in\Z$ by Lemma \ref{lemma:irred}. Hence 
\be\label{Esupinf-middle}
\widetilde{\nu}(E_{\supinf} 
)\leq\sum_{k\in\Z}\widetilde{\nu}(X=k)\leq 1
\ee
 Since $\widetilde{\nu}\in\widetilde{\mathcal S}$, 
 $\widetilde{\nu}(X=k)$ 
does not depend on $k\in\Z$. This quantity must vanish 
by the second inequality in \eqref{Esupinf-middle}, 
hence the result follows from the first one.
\end{proof}
 For $m,n\in\Z\cup\{\pm\infty\}$, where $m\leq n$, let
\be\label{def:Dmn}
D_{m,n}(\eta,\xi):=
\sum_{x\in V:\,m\leq x(0)\leq n}|\eta(x)-\xi(x)|
\ee
denote the number of discrepancies in the space interval $[m,n]\cap\Z$. 
We simply write $D(\eta,\xi)$ when $(m,n)=(-\infty,+\infty)$.
\begin{proposition}\label{prop_kilroy}
 Let  $\widetilde{\lambda}\in\widetilde{\mathcal I}$. 
Assume either 
$\widetilde{\lambda}\in\widetilde{\mathcal S}$, or 
\be\label{assumption_finite_disc}
\int_{\mathcal X\times\mathcal X}D(\eta,\xi)\widetilde{\lambda}(d\eta,d\xi)<+\infty
\ee
Then, for every  $(x,y)\in V\times V$ 
such that $x$ and $y$ are $p$-connected,
\be
\label{nodis}
\widetilde{\lambda}\left(E_{x,y}\right)=0
\ee
where
\be\label{def_dk}
E_{x,y}:=\left\{
(\eta,\xi)\in \mathcal X\times\mathcal X:\, 
\mbox{there are  opposite discrepancies  at }x\mbox{ and }y
\right\}
\ee
\end{proposition}
 An equivalent formulation of Proposition \ref{prop_kilroy} is 
\be\label{p_ordered_conf}
\widetilde{\lambda}\left\{
(\eta,\xi)\in\mathcal X\times\mathcal X:\,\eta
\mbox{ and }\xi\mbox{ are }p\mbox{-ordered}
\right\}=1
\ee
In \cite[Theorem 1.1]{Liggett1976} it is proved that 
 if $\widetilde{\lambda}$ 
  is a translation invariant and invariant probability measure 
  for a one-dimensional  translation invariant
 SEP (coupled via basic coupling),  
then  \eqref{nodis} holds whenever $x$ and $y$ are $p$-connected. 
The argument carries over to our setting by using 
 only  translation invariance in the $\Z$ direction. 
For the sake of completeness, details  of the proof 
of Proposition \ref{prop_kilroy} 
are given in Appendix \ref{app:nodis}. \newline\newline
 We are now in a position to  complete the proof of \thmref{characterization_lemma}.  
\begin{proof}[Proof of \thmref{characterization_lemma}]
Let $\mu\in(\mathcal{I}\cap\mathcal{S})_{e}$  and $\rho\in[0,2]$. 
Since $\nu_\rho\in(\mathcal{I}\cap\mathcal{S})_{e}$ 
(cf. Lemma \ref{lem:invariant_measures_for_the_asymetric_ladder}),  
using \cite[Proposition 2.14 in Chapter VIII]{liggett2012interacting}, 
we obtain  a measure $\widetilde{\lambda}$ on 
$\mathcal{X}\times\mathcal{X}$,
which belongs to 
$(\widetilde{\mathcal{I}}\cap\widetilde{\mathcal{S}})_{e}$ 
and whose marginals are $\mu$ and
$\nu_{\rho}$.      The events 
\begin{eqnarray}\label{def:E-}
E_-&:=&\{(\eta,\xi)\in\mathcal X\times\mathcal X:\,
\eta\leq\xi\}\quad\mbox{and}\\ \label{def:E+}
E_+&:=&\{(\eta,\xi)\in\mathcal X\times\mathcal X:\,\eta\geq\xi\}
\end{eqnarray}
 are invariant with respect to spatial translations, 
 and (by attractiveness) they
are conserved by 
the coupled dynamics. Since
  $\widetilde{\lambda}\in 
  (\widetilde{\mathcal{I}}\cap\widetilde{\mathcal{S}})_{e}$,  
$E_+$ and $E_-$ both have $\widetilde{\lambda}$-probability $0$ or $1$. 
The main step is to prove 
that 
\be\label{ordered_coupling}
 \widetilde{\lambda}(E_+\cup E_-)= \widetilde{\lambda}\left\{
(\eta,\xi)\in\mathcal X\times\mathcal X:\,\eta\leq\xi
\mbox{ or }\xi\leq\eta
\right\}=1
\ee
implying that one of the events $E_+$ and $E_-$ has probability $1$.
It follows that for every $0<\rho<2$
we either have $\mu\leq\nu_{\rho}$ or $\mu\geq\nu_{\rho}$.
By Remark \ref{rem:increasing set of measures} we conclude that there
exists some $r\in[0,2]$ such that $\mu=\nu_{r}$.\newline\newline
 We now turn to the proof of \eqref{ordered_coupling}.
 Outside the case $q=0=l_0=l_1<p$, the kernel $p(.,.)$ 
in \eqref{eq:intensity_c} is weakly irreducible; 
thus \eqref{ordered_coupling} 
follows from \eqref{p_ordered_conf} and \textit{(i)} 
of Lemma \ref{lemma:irred}.
Now assume $q=0=l_0=l_1<p$.  
By \textit{(ii)} of  Lemma \ref{lemma:irred},  we obtain
 \be\label{notquite}
\widetilde{\lambda}(E_-\cup E_+\cup E_{\supinf})=1
\ee  
and the conclusion follows from Lemma \ref{lemma:irred2}.
\end{proof}
\section{Proofs of  
Theorems \ref{thm:Characterization_of_invariant_measures}, 
\ref{cor:entire} 
and \ref{thm:Characterization_of_invariant_measures_gen}}
\label{sec:proofs-2}
The proofs of Theorems \ref{thm:Characterization_of_invariant_measures} 
and \ref{cor:entire} are developed respectively 
in Subsections \ref{subsec:proof_inv} and
 \ref{subsec:proof_inv_3}. 
  They are decomposed into  six steps,  summarized in 
 Subsection \ref{subsec:ideas}. These 
 intermediate results are  all 
established in Section \ref{subsec:proofs_lemmas}, 
except Proposition \ref{prop:real_r_finite}, established 
in Subsection \ref{subsec:proof_inv_2}. Indeed this proposition 
is necessary for 
Theorem \ref{thm:Characterization_of_invariant_measures}, 
but its proof  introduces material 
 (namely current and flux function)  also used for 
Theorem \ref{cor:entire}. 
Finally, Theorem \ref{thm:Characterization_of_invariant_measures_gen}
 is proved in Subsection \ref{subsec:proof_gen}. 
\subsection{Main ideas  to prove Theorems \ref{thm:Characterization_of_invariant_measures}, 
\ref{cor:entire}}\label{subsec:ideas}
 We summarize the scheme of proof in six steps 
described here informally,  whose precise statements 
are given in the next subsection. \\ \\
{\bf Step one: shifting an  invariant measure.} 
Let $\mu\in\mathcal I_e$.  We want to 
compare the measure $\mu$ with its shift $\tau\mu$.  
This will follow from construction of a coupling
of these two measures satisfying some ordering or pseudo-ordering relation. 
Two  very different subcases must be considered: \\ \\
{\em Weakly irreducible case.} This is when either  $q>0$, 
or  $q=0$ but \eqref{non_wi_case} does not hold. We then prove  
that  $\mu\leq\tau\mu$ or $\tau\mu\leq\mu$ (stochastic order).  
This follows from construction of a coupling  
$\overline{\lambda}(d\eta,d\xi)$ of $\mu(d\eta)$ 
and $\tau\mu(d\xi)$ under which $\eta\leq\xi$ or 
$\xi\leq\eta$ a.s. This construction, performed in Proposition \ref{prop:step1}, 
is an adaptation to our model of \cite[Proposition 3.2]{Bramson2002}. \\ \\
{\em Non-weakly irreducible case.} When $q=0$ and \eqref{non_wi_case} 
holds, as in the proof of Theorem \ref{thm:characterization_lemma}, 
the above arguments do not lead to $\eta\leq\xi$ or $\xi\leq\eta$, 
but only to $\eta><\xi$. Unlike in Theorem \ref{thm:characterization_lemma}, 
we cannot use translation invariance to eventually obtain $\eta\leq\xi$ 
or $\xi\leq\eta$. We  introduce an intermediate relation denoted by 
$\eta\bowtie\xi$, that is a strenghtening of $\eta\supinf\xi$, 
see Definition \ref{def_bowtie}, and obtain a coupling under which 
$\eta\bowtie\xi$. This is also contained in Proposition \ref{prop:step1}. 
 But unlike in \cite{Bramson2002}, we cannot next prove that 
$\mu$ or $\tau\mu$ are ordered. A different type of argument is required 
to conclude in this case that $\mu$ can only be a blocking or 
a partial blocking measure. This is the object of 
Proposition \ref{prop:bowtie}. \\ \\
 Steps two to five below apply to the weakly irreducible case, 
whereas the conclusion of the non-weakly irreducible case in contained 
in Step six.  \\ \\
{\bf Step two: getting a ``mean'' shock.} 
It is shown in Proposition \ref{prop:step1} that the total number 
of discrepancies $D(\eta,\xi)$ (see \eqref{def:Dmn}) is  
 a.s. a finite  constant 
$k$ under $\widetilde{\lambda}$. If $k=0$, then $\tau\mu=\mu$, 
and we are back to Theorem \ref{thm:characterization_lemma}. 
Otherwise, along the proof of Proposition \ref{prop:step1}, 
we show that the expectation of $D(\eta,\xi)$ under 
$\widetilde{\lambda}$ yields (since we have an ordered coupling of 
$\mu$ with its shift) a telescoping sum equal to the (positive) 
difference of mean densities at 
$\pm\infty$, $\rho^\pm=\lim_{x\to\pm\infty}\mu[\eta(x)]$. 
We call this a ``mean'' shock. Since $D(\eta,\xi)$ is an integer and this 
difference cannot exceed $2$,  $k=\rho^+-\rho^-\in\{1,2\}$. 
\newline\newline
{\bf Step three: mean shock implies shock.}
Since $\mu$ and $\tau\mu$ are ordered, the limits 
$\mu_\pm:=\lim_{x\to\pm\infty}\tau_x\mu$ exist, and an averaging 
argument shows that $\mu_\pm\in(\mathcal I\cap\mathcal S)$. 
This is Corollary \ref{cor:step1}. 
At this stage a crucial step appears, that is {\em not} 
required for single-lane ASEP because the latter model has the 
simplifying feature that densities are restricted to $1$ 
(see Remark \ref{remark:diff} for more details on this). 
The problem is to show that $\mu_\pm\in (\mathcal I\cap\mathcal S)_{e}$, 
implying that $\mu_\pm=\mu_{\rho^\pm}$,
hence that $\mu$ is a {$(\rho^-,\rho^+)$-{\em shock measure}}. 
This is done in Proposition \ref{prop:step2}. Thus we know that 
if $\mu\in\mathcal I_e\setminus(\mathcal I\cap\mathcal S)$, then $\mu$ is
a  shock measure of amplitude $|\rho^+-\rho^-|\in\{1,2\}$. 
If $|\rho^+-\rho^-|=2$ we have a $(0,2)$ or a $(2,0)$ shock that are 
analogous to profile measures in \cite{Bramson2002}. The choice 
\eqref{wlog} implies that only $(0,2)$ is possible, see Lemma \ref{lemma:nob2}.
If $|\rho^+-\rho^-|=1$, we need to restrict possible shocks 
$(\rho^-,\rho^+)$. \newline\newline 
{\bf Step four: restricting possible shocks.} 
The relevant object is the (microscopic and macroscopic) 
{\em flux function} of our model, introduced in 
\eqref{def_microflux_ladder}--\eqref{eq:G-from-rho_0}. In 
Proposition \ref{cor:step2}, we show that a 
$(\rho^-,\rho^+)$-flux function cannot exist unless 
$(\rho^-,\rho^+)$ is an entropy shock for the macroscopic flux 
function $G$, see Definition \ref{def_set} and 
Remark \ref{rem:entropy} below.  
In Proposition \ref{prop:extrema} and Lemmas 
\ref{lemma:r_finite}--\ref{lemma:nob2},
explicit computations on the macroscopic flux 
$\rho\in[0,2]\mapsto G(\rho)$ allow us to disqualify most shocks 
and prove  (in Proposition \ref{prop:real_r_finite}) 
 statement \textit{(ii)} 
of  Theorem \ref{thm:Characterization_of_invariant_measures}. 
These computations further show that in a certain parameter range 
(see statement \textit{(iii)} of Theorem \ref{cor:entire}), 
no entropy shock, hence no shock measure  of amplitude $1$ exists.  
Condition \eqref{cond_nob2} in \textit{(b)} 
of Theorem \ref{thm:Characterization_of_invariant_measures}, 
which excludes blocking measures, expresses the fact that 
the graph of the flux function $\rho\mapsto G(\rho)$ crosses 
the $\rho$-axis.  \newline\newline
Special situations are \eqref{unless}--\eqref{unless_2}. In these cases 
the function $G(\rho)$ is identically $0$ and does not help to eliminate 
shocks. In the   latter case  we show (statement \textit{(i)} 
of Theorem \ref{cor:entire}) that the system is of diffusive gradient type, 
i.e. the microscopic flux is a gradient, which leads to non-existence 
of shocks. In the former, as mentioned in the comments following 
Theorem \ref{thm:Characterization_of_invariant_measures}, the model 
is presumably diffusive but non-gradient, and specific techniques would 
be required. \newline\newline
\textbf{Step five: uniqueness of a $(\rho^-,\rho^+)$-shock measure.}
We next show in Proposition \ref{prop:translate} that if
$|\rho^+-\rho^-|=k\in\{1,2\}$, there are (up to shifts) at most 
$k$ $(\rho^-,\rho^+)$-shock measures in $\mathcal I_e$, except 
for $k=2$ and $q=0$. Recall indeed from Subsection 
\ref{sec:blocking} that in  the statement of  Theorem \ref{cor:entire} we may have 
a family of two (up to shifts) $(0,2)$-shock measures when $q>0$, 
and infinitely many when $q=0$.
To prove Proposition \ref{prop:translate}, 
 a key step is showing that two 
$(\rho^-,\rho^+)$-shock measures $\mu$ and $\nu$ are comparable 
(Proposition \ref{lemma:translate}). Relying on this, we can extend to arbitrary 
shocks of any amplitude an argument of \cite{Bramson2002} for 
ASEP $(0,1)$-shock measures, whose idea is to squeeze $\nu$ between 
successive translates of $\mu$. Note that the prior comparability 
step is not necessary in the single-lane case. \\ \\
{\bf Step six: the case $q=0$.} 
In this case, the flux function $G(\rho)$ is very explicit, cf. \eqref{flux_degenerate}  in Example \ref{ex:periodic}.   This allows 
more precise shock selection in Step four: in particular 
$\mathcal R=\emptyset$ if $\gamma_0\neq\gamma_1$. \newline\newline
Next, thanks to the condition $q=0$, one can compare each lane 
with an ASEP and use convergence and characterization results 
for single-lane ASEP (\cite[Theorem 1.4]{Liggett1976}, 
\cite[Chapter VIII]{liggett2012interacting}, \cite[Theorem 1]{bm}). \newline\newline
In the non-weakly irreducible case \eqref{non_wi_case}, 
starting from the partial conclusion $\eta\bowtie\xi$ of Step one, 
Proposition \ref{prop:bowtie} further concludes  that the invariant 
measure $\mu$ must be a blocking or a partial blocking measure 
as in Theorem \ref{cor:entire}, \textit{(iv) (b)}. \newline\newline
In the weakly irreducible case, that is cases \textit{(iv)--(vi)} 
of Theorem \ref{cor:entire}, statement \textit{(iii)} of 
Proposition \ref{cor:step2} shows that a shock measure of amplitude 1 
with a profile outside $\mathcal R$ must belong to 
\eqref{def_bl_3_2}, \eqref{def_bl_5} or \eqref{def_bl_3_1}. \newline\newline
Finally, to show that $Bl_2$ is empty outside case \eqref{non_wi_case}, 
assuming that one lane carries a blocking measure, we exhibit (see
 \eqref{def_h_lane_1}) a Lyapunov functional on one lane 
that has a positive probability of decreasing unless the other lane is empty. 
%
\begin{remark}\label{remark:diff}
 As mentioned in the introduction, 
since for the single-lane ASEP the maximal density is $1$, for a 
mean shock of amplitude $1$, we must have $\{\rho^+,\rho^-\}=\{0,1\}$; 
this automatically implies that $\mu$ is asymptotic at $\pm\infty$ 
to the corresponding (deterministic) Bernoulli measures, i.e., 
$\mu$ is a $(0,1)$ or a $(1,0)$-shock measure.
Further analysis shows that it cannot be a $(1,0)$-shock measure.
Thus for single-lane ASEP, 
$\mathcal I_e$ contains only profile measures, and there is 
no need for Steps 3 and 4, namely, proving that $\mu$ 
is a shock measure and analyzing possible shocks. 
\end{remark}
 \subsection{Proof of 
Theorem \ref{thm:Characterization_of_invariant_measures}}
\label{subsec:proof_inv}
We have to distinguish   the case  \eqref{case_deg}, 
where the kernel $p(.,.)$ in \eqref{eq:intensity_c} 
is not weakly irreducible,  cf. Lemma \ref{lemma_wi}. 
In this case, we introduce the following definition. 
\begin{definition}\label{def_bowtie}
For $(\eta,\xi)\in\mathcal X\times\mathcal X$, we write 
$\eta\bowtie\xi$  if and only if the following hold: 
(i) $\eta\supinf\xi$ (cf. Definition \ref{def_supinf}); 
(ii) the number of locations 
$z\in\Z^+$ on lane $1$ that are not occupied by a  coupled particle is  
finite; (iii)  the number of locations 
$z\in\Z^-$  on lane $0$ that are not  occupied by a hole is finite. \\
 We define
\be\label{def:Ebowtie}
E_{\bowtie}:=\{(\eta,\xi)\in{\mathcal X}\times{\mathcal X}:
\,\eta\bowtie\xi\}
\ee 
\end{definition}
 Following the steps described in Subsection \ref{subsec:ideas}
(that we recall below),
 the main results  for the proof 
 of \thmref{Characterization_of_invariant_measures} 
are  Propositions 
\ref{prop:step1}--\ref{prop:bowtie} 
and Corollary \ref{cor:step1}, 
stated below.  Among these, Proposition \ref{prop:real_r_finite} 
is proved in Subsection \ref{subsec:proof_inv_2}, and other results  
in  Subsection \ref{subsec:inter_1}.   \newline\newline
\noindent \textit{Step one.}  Let $\mu\in\mathcal I_e$. 
We prove the following proposition.
\begin{proposition}\label{prop:step1}
 (i) 
There exists a measure $\widetilde{\lambda}(d\eta,d\xi)$ 
on $\mathcal X\times\mathcal X$
with marginals $\mu(d\eta)$ and $\tau_1\mu(d\xi)$, satisfying one of
\eqref{etaless}--\eqref{equal} below  (if  $q>0$), or one of 
\eqref{etaless}--\eqref{case_bowtie} below (if  $l_0=l_1=q=0<p$, 
that is \eqref{case_deg}, the 
  non-weakly irreducible case):  
 \begin{eqnarray}\label{etaless}
\widetilde{\lambda}\,(E_1)&=&1\quad\hbox{where}\quad 
E_1:=\left((\eta,\xi)\in \mathcal X\times\mathcal X:\,
\eta<\xi\right)  \\
\label{etamore}
\widetilde{\lambda}\,(E_2)&=&1\quad\hbox{where}\quad
E_2:=\left((\eta,\xi)\in \mathcal X\times\mathcal X:\,\xi<\eta\right) \\
\label{equal}
\widetilde{\lambda}\,(E_3)&=&1\quad\hbox{where}\quad
E_3:=\left((\eta,\xi)\in \mathcal X\times\mathcal X:\,\eta=\xi\right)  \\
\label{case_bowtie}
\widetilde{\lambda}\,(E_{\bowtie})&=&1
\end{eqnarray} 
(ii)  For any measure $\widetilde{\lambda}(d\eta,d\xi)$ with marginals 
$\mu(d\eta)$ and $\tau_1\mu(d\xi)$ satisfying \eqref{etaless}  
or \eqref{etamore},  
there exists $k\in\{1,2\}$ such that  
(cf.  definition of $D(\eta,\xi)$ below \eqref{def:Dmn}) 
\be\label{number_disc}
\widetilde{\lambda}\left(
(\eta,\xi)\in \mathcal X\times\mathcal X:\,D(\eta,\xi)=k
\right)=1
\ee 
\end{proposition}
 \noindent{\em Step two.} 
Proposition \ref{prop:step1} has the following consequences.
\begin{corollary}\label{cor:step1}
%
%
(i) In cases \eqref{etaless}--\eqref{equal}, 
the family $(\tau_n\mu)_{n\in\Z}$ is stochastically monotone.\newline\newline
(ii) If  a probability measure $\widehat \mu$  on $\mathcal X$ 
is such that  $\widehat\mu\in\mathcal I$ and 
$(\tau_n\widehat\mu)_{n\in\Z}$ 
is stochastically monotone, then there exist probability measures 
$\gamma^-(d\rho)$ and $\gamma^+(d\rho)$ on $[0,2]$ 
such that the limits 
\be\label{limits_mu_2}
\widehat\mu_\pm:=\lim_{n\to\pm\infty}\tau_n\widehat\mu
=\int_{[0,2]}\nu_\rho\gamma^\pm(d\rho)
\ee
 hold in the sense of weak convergence.
\end{corollary}
 \noindent{\em Step three.} 
We  show that the measures 
$\gamma^\pm$ of Corollary \ref{cor:step1} are 
Dirac measures. 
\begin{proposition}\label{prop:step2}
  In cases \eqref{etaless}--\eqref{etamore}, 
there exists  $(\rho^-,\rho^+)\in[0,2]^2\setminus\mathcal D$  
such that
(i)  
$\gamma^\pm=\delta_{\rho^\pm}$,  thus $\mu$ is a 
$(\rho^-,\rho^+)$-shock measure, 
cf. \eqref{limits_mu}; 
(ii)  $|\rho^+-\rho^-|=k$, where $k$ is defined in 
(ii) of Proposition \ref{prop:step1}. 
\end{proposition}
\noindent \textit{Step four. } 
We first introduce the sets $\mathcal R$ and $\mathcal R'$ 
involved in \thmref{Characterization_of_invariant_measures}.
\begin{definition}\label{def_set_r}
We denote by $\mathcal R$ the set of 
$(\rho^-,\rho^+)\in[0,2]^2\setminus(\mathcal D\cup\mathcal B)$ 
such that $\mathcal I_e$ contains at least one 
$(\rho^-,\rho^+)$-shock measure, 
and by $\mathcal R'$ the set of $(\rho^-,\rho^+)\in\mathcal B_1$ 
such that 
$\mathcal I_e$ contains at least one $(\rho^-,\rho^+)$-shock measure.
\end{definition}
 In Subsection \ref{subsec:proof_inv_2} below, 
we prove the following proposition, 
  after  introducing the \textit{macroscopic flux function} of our model.
   Recall that \eqref{case_deg} corresponds to the 
  non-weakly ireducible case. 
\begin{proposition}\label{prop:real_r_finite}
 Outside   \eqref{unless}--\eqref{unless_2},  
 the following holds:
(i)   in cases \eqref{etaless}--\eqref{etamore} 
with $k=2$, 
$\mu$ is a $(0,2)$-shock measure; 
(ii)   Statements (ii),   \textit{(b), (c)} and \textit{(d)}  
of Theorem \ref{thm:Characterization_of_invariant_measures} 
hold. 
 (iii) Statement (i) of Proposition \ref{prop:translate} 
 below 
still holds if we have \eqref{case_deg} and $\nu\in Bl_1$. 
\end{proposition}
 \noindent\textit{Step five. } 
In Proposition \ref{prop:translate} below, we study 
the relation between extremal invariant measures 
that are $(\rho^-,\rho^+)$-shock measures
for a common pair $(\rho^-,\rho^+)$. 
 The proof of Proposition \ref{prop:translate} 
 requires the following variant of Proposition \ref{prop:step1}. 
\begin{proposition}\label{lemma:translate}
Let $(\rho^-,\rho^+)\in[0,2]^2\setminus\mathcal D$, 
and assume $\nu,\,\nu'\in\mathcal I_e$ 
are two $(\rho^-,\rho^+)$-shock measures. 
Then:  (i) there exists a coupling of $\nu$ and $\nu'$ 
that satisfies 
one of the properties \eqref{etaless}--\eqref{case_bowtie},  
property \eqref{case_bowtie} being possible 
only under assumption \eqref{case_deg}; 
(ii) in case \eqref{case_bowtie}, $\nu$ and $\nu'$ lie in  $Bl_1\cup Bl_2$; 
(iii) in cases \eqref{etaless}--\eqref{etamore}, 
\eqref{number_disc} holds for some 
$k\in(\N\setminus\{0\})\cup\{+\infty\}$. 
\end{proposition}
\begin{proposition}\label{prop:translate}
Let $\nu,\nu'\in \mathcal I_e$ be two $(\rho^-,\rho^+)$-shock measures. 
(i) Assume $|\rho^+-\rho^-|=1$, and we do not  
simultaneously have \eqref{case_deg} and $\nu\in Bl_1$. 
Then $\nu'$ is  a translate of $\nu$, i.e. there exists 
$n\in\Z$ such that $\nu'=\tau_n\nu$.
(ii) Assume $|\rho^+-\rho^-|=2$, $\nu'$ is not a translate of $\nu$, 
and we do not have \eqref{case_deg}. Then every 
$(\rho^-,\rho^+)$-shock measure is either a translate of $\nu$, 
or a translate of $\nu'$.
\end{proposition}
\noindent\textit{Step six. } 
 We  conclude in case 
\eqref{case_bowtie} of Proposition \ref{prop:step1}. 
 This step is pursued in Subsection \ref{subsec:proof_inv_3}, 
in the proof of statements \textit{(iv)--(vi)} of 
Theorem \ref{cor:entire}. 
\begin{proposition}\label{prop:bowtie}
In case \eqref{case_bowtie}, 
we have   $\mu\in Bl_1\cup Bl_2$. 
\end{proposition}
 \noindent \textit{Final step.}  
We   assemble the previous steps to conclude 
the proof of \thmref{Characterization_of_invariant_measures}. 
\begin{proof}[Proof of \thmref{Characterization_of_invariant_measures}]
  First, for $\mu\in\mathcal I_e$,  we consider 
  the different possibilities in Proposition \ref{prop:step1}.
In case \eqref{equal}, 
we have $\mu\in(\mathcal I\cap\mathcal S)_e$;
 by Theorem \ref{thm:characterization_lemma}, 
 $\mu=\nu_\rho$ for some $\rho\in[0,2]$. 
In case \eqref{case_bowtie}  
(which may only occur under \eqref{case_deg}),  
Proposition \ref{prop:bowtie}  implies $\mu\in Bl_1\cup Bl_2$, 
with $Bl_1$ and $Bl_2$ given by \eqref{def_bl_3_2}--\eqref{def_bl2(iv)(b)}. 
In cases \eqref{etaless}--\eqref{etamore} 
with $k=2$  in \eqref{number_disc},  Proposition 
 \ref{prop:step2}   and \textit{(i)} of 
 Proposition \ref{prop:real_r_finite}  lead to $\mu\in Bl_2$.
 In cases \eqref{etaless}--\eqref{etamore} 
 with $k=1$ in \eqref{number_disc},
by Proposition  \ref{prop:step2}, $\mu$ is a shock measure 
of amplitude $1$. \newline\newline
Next, to obtain \eqref{charac_ext}, we consider the structure 
modulo translations of shock measures. Cardinality bounds 
for $\mathcal R$ and $\mathcal R\cup\mathcal R'$ are given 
by Proposition \ref{prop:real_r_finite}. 
By \textit{(i)} of Proposition \ref{prop:translate}  
and \textit{(iii)} of Proposition \ref{prop:real_r_finite}, 
for every 
 $(\rho^-,\rho^+)\in\mathcal R\cup\mathcal R'$,  the set of 
$(\rho^-,\rho^+)$-shock measures in $\mathcal I_e$ 
consists of translates of a single measure.
 The set $Bl_2$ is stable  by translation 
because the generator  
\eqref{eq:generator of the Exclusion}  with transition kernel  
\eqref{eq:intensity_c} is translation invariant. 
 This concludes the proof of  \textit{(i)}.  
Statements  \textit{(ii), (b), (c)} and \textit{(d)} 
are contained in 
statement \textit{(ii)} of Proposition \ref{prop:real_r_finite}. 
 We conclude with the proof of \textit{(ii), (a)}.  By \textit{(ii)} 
 of Proposition \ref{prop:translate}, 
outside \eqref{case_deg}, $Bl_2$ consists of at most 
(up to translations) two measures. \newline\newline
 We now prove that if $Bl_2$ is nonempty, outside case \eqref{case_deg}, 
 it consists of exactly (up to translations) two measures. We already know by 
{\it (ii), (a)} of Theorem \ref{thm:Characterization_of_invariant_measures} 
that $Bl_2$ has at most two elements. Thus we must show that if $Bl_2$ 
contains some element
$\nu$, it contains another one $\nu'$ that is not a shift of $\nu$. 
Since for $\eta\in\mathcal X_2$ and $x\in\Z$, we have
\be\label{shift_H}
H_2(\tau_x\eta)=H(\eta)-2x
\ee
for $H_2$ defined by \eqref{def_H_2}, 
without loss of generality, we may assume that $\nu$ is supported 
on $\{\eta\in\mathcal X_2:\,H_2(\eta)=0\}$ or $\{\eta\in\mathcal X_2:\,H_2(\eta)=1\}$.
The proof being similar in both cases, we assume the former.
Let 
\be\label{def_leftmost}
X_0(\eta):=\inf\{x\in\Z:\,\eta(x,0)+\eta(x,1)=1\}
\ee
denote the position of the leftmost $\eta$-particle, that is  
finite on $\mathcal X_2$ and thus under $\nu$. 
At time $0$ we consider an initial random configuration $\eta\sim\nu$ 
and define a random configuration $\xi$
by adding to $\eta$ a (so-called second-class) particle at 
$Y_0(\eta,\xi):=(X_0(\eta)-1,0)$. 
We consider the coupled process $(\eta_t,\xi_t)$ starting 
from the random initial configuration $(\eta,\xi)$. We denote t
he law of $\xi_t$ by $\nu'_t$  and 
consider
\[
M'_t:=\frac{1}{t}\int_0^t\nu'_s\, ds
\]
The family $(M'_t)_{t>0}$ is tight because it is supported 
on the compact space $\mathcal X$. The following proposition is proved 
in Appendix \ref{app:tight}, and with \eqref{shift_H}, 
yields the desired conclusion.
\end{proof}
\begin{proposition}\label{prop:tight}
Any subsequential limit 
$M'(d\xi)$ of the family $(M'_t)_{t>0}$ is an element of $Bl_2$ 
supported on the set
\be\label{def_H_2_1} \mathcal X_{2,1}:=\{\eta\in\mathcal X_2:\,H_2(\eta)=1\}\ee
\end{proposition}
%
%
\subsection{Proof of  Proposition \ref{prop:real_r_finite} }
\label{subsec:proof_inv_2}
We begin by defining the flux function, which will 
also play an important role in the proof of Theorem \ref{cor:entire},
and state some of its properties.
  \subsubsection{Microscopic current and macroscopic flux}\label{subsec:current-flux} 
 We first define the microscopic current by 
\be\label{def_microflux_ladder}
j(\eta):=\sum_{x(0)\leq 0,y(0)>0}p(x,y)\eta(x)(1-\eta(y))
-\sum_{x(0)\leq 0,y(0)>0}p(y,x)\eta(y)(1-\eta(x))
\ee
 for $\eta\in\mathcal{X}$.  With the kernel 
 defined by \eqref{eq:intensity_c}, this yields 
\begin{eqnarray}
j(\eta) & = &  \sum_{i=0}^1
\left\{d_i\eta^i(0)[1-\eta^i(1)]-l_i\eta^i(1)[1-\eta^i(0)]\right\} \\
 & = & 
=\sum_{i=0}^1\left\{\gamma_i\eta^i(0)[1-\eta^i(1)]
+l_i[\eta^i(0)-\eta^i(1)]\right\}\label{def_microflux_ladder_2} 
\end{eqnarray}
The macroscopic flux is then given by,  for $\rho\in[0,2]$, 
\be\label{def_macroflux_ladder}
G\left(\rho\right)  :=  \int j(\eta)d\nu_{\rho}\left(\eta\right).
\ee
 Using  \eqref{def_nurho} and 
 (\ref{eq:two-rate Bernoulli measure}), this yields
\begin{equation}\label{eq:G-from-rho_0}
G\left(\rho\right)
= \gamma_0 G_{0}\left[\widetilde{\rho}_{0}\left(\rho\right)\right]
+ \gamma_1 G_{0}\left[\widetilde{\rho}_{1}\left(\rho\right)\right]
\end{equation}
where  $G_0$ is the flux function of the single-lane TASEP, given by
\begin{equation}\label{flux_tasep}
G_{0}\left(\alpha\right)=\alpha\left(1-\alpha\right)
\quad  \forall \alpha\in[0,1]. 
\end{equation} 
 In the following two special cases, 
 the function $G$  has a simple expression. 
\begin{example}\label{ex:periodic}
Assume  $q=0<p$.  Then, by \eqref{eq:G-from-rho_0} 
and \eqref{tilderho_2b}, 
\be\label{flux_degenerate}
G(\rho)=\left\{
\ba{lll}
\dsp  \gamma_1\rho(1-\rho) & \mbox{if} & \rho\in[0,1]\\ 
\dsp \gamma_0(\rho-1)(2-\rho) & \mbox{if} & \rho\in(1,2]
\ea
\right.
\ee
In particular, when  $\gamma_0=\gamma_1$,  
the flux is a function of period $1$ 
whose restriction to $[0,1]$ is the TASEP flux. 
It exhibits a change of convexity at $\rho=1$, 
where it is also non differentiable. 
Note that the latter property is not seen in usual  
single-lane models with product invariant measures.
\end{example}
\begin{example}\label{ex:sym}
Assume $p=q>0$. Then, by \eqref{eq:G-from-rho_0} 
and \eqref{tilderho_1b},
\be\label{flux_sym}
 G(\rho)=\frac{\gamma_0+\gamma_1}{4}\rho(2-\rho)
\ee
Here, the flux has the same shape as the single-lane TASEP flux 
(from which it is obtained by a scale change in the horizontal 
and vertical directions). It is in particular strictly concave.
\end{example}
 Useful properties of $G$ are gathered in the following proposition. 
 Except for statement \textit{(iii)} below, these properties 
do not rely on \eqref{wlog}. 
\begin{proposition}\label{prop:extrema}\mbox{}\newline\newline
(o) $G(0)=G(2)=0$.  The function $G$ is identically $0$ 
in cases  \eqref{unless}-- \eqref{unless_2}. \\ \\  
(i)  Outside cases 
 \eqref{unless}, \eqref{unless_2} and \eqref{unless_3}, 
$G$ has at least one and at most three local extrema. \\ \\
(ii)  (a)  $G(1)=0$  if and only if $q=0$ or $\gamma_0+\gamma_1=0$; 
 (b)  if $q> 0$, $G$ is continuously differentiable on $[0,2]$, 
and $G'(1)=0$ if and only if  $\gamma_0=\gamma_1$  or $p=q$. \\ \\
(iii)  Under \eqref{wlog},  
$G'(2)\leq 0$.
Besides, 
$G'(2)< 0$  
holds unless we have \eqref{unless}, or \eqref{unless_2}, 
or $q=\gamma_0=0$. \\ \\
%
 (iv) The function $G$ depends only on the parameters 
 $\gamma_0,\gamma_1$  and $r$ defined in \eqref{reduced_par}. 
 Denoting $G=G_{\gamma_0,\gamma_1,r}$, it holds that 
  for every $\rho\in[0,2]$, 
\be\label{symmetry_g}
G_{\gamma_0,\gamma_1,r}(2-\rho)  
=  G_{\gamma_1,\gamma_0,r}(\rho)=G_{\gamma_0,\gamma_1,r^{-1}}(\rho)
\ee
where the last equality holds when $r>0$.
If $\gamma_0+\gamma_1\neq 0$, 
it holds that  for every $\rho\in[0,2]$, 
\be\label{hom_g}
G_{\gamma_0,\gamma_1,r}(\rho)
=(\gamma_0+\gamma_1) G_{\mathfrak{a},1-\mathfrak{a},r}(\rho) 
\quad\mbox{with\quad}  \mathfrak{a}\mbox{ defined in  }\eqref{reduced_par}. 
\ee
 (v) Assume $\gamma_0=\gamma_1\neq 0$, that is  $\mathfrak{a}=1/2$. 
Then: (a) $G'(1/2)>0$; (b) for   $r\in(0,r_0)\cup(1/r_0,+\infty)$,  
with $r_0$ given by \eqref{def_r0}, we have $G(1/2)>G(1)$. \\ \\
 (vi) 
If $q\neq 0$ and $\gamma_0+\gamma_1\neq 0$, 
the equation $G(\rho+1)-G(\rho)=0$ has a unique solution in $[0,1]$. 
If  $q\neq 0$,  $p\neq q$ and  $\gamma_0+\gamma_1=0\neq\gamma_0\gamma_1$,  
the solutions of this equation in $[0,1]$ are $\rho=0$ and $\rho=1$. \\ \\
 (vii) Outside cases \eqref{unless}--\eqref{unless_2}, 
the equation $G(\rho)=0$ has at least one solution in $(0,2)$ 
if and only if condition \eqref{cond_nob2} holds. In this case, 
the solution is  unique and $G$ changes sign around this solution.
\end{proposition}
\begin{remark}\label{rk_r0}
Statement (v) of Proposition \ref{prop:extrema} is involved 
in the proof of statement (iii) of Theorem \ref{cor:entire}. 
In the latter the condition on $r$ is $r\in(0;r_0)$ instead of 
$r\in(0,r_0)\cup(1/r_0,+\infty)$, because the theorem is stated 
under \eqref{wlog}, hence $r<1$.
\end{remark}
 \subsubsection{Proof of Proposition \ref{prop:real_r_finite}}
\label{subsec:proof-prop37} 
 The scheme of proof of Proposition \ref{prop:real_r_finite}
is the following. 
We introduce in Definition \ref{def_set}  
a set denoted by $\mathcal R_0$,
which depends only on the flux function. 
Lemma \ref{lemma:r_finite} (which will be proved 
using Proposition \ref{prop:extrema})  
says that $\mathcal R_0$ contains at most three elements, 
in most cases no more than one, and sometimes none. 
 Proposition \ref{cor:step2}  provides information
  on possible stationary shock measures, implying  that 
  $\mathcal R$ is contained in $\mathcal R_0$; 
  part of this proposition will be useful 
  for the proof of Theorem \ref{cor:entire}.  
  Lemma \ref{lemma:nob2} sets restrictions on possible 
  shock measures of amplitude $2$.   The proof of Proposition 
\ref{prop:real_r_finite} is concluded using 
 Lemma \ref{lemma:r_finite}, Lemma \ref{lemma:nob2}  
 and Proposition \ref{cor:step2}; 
these are proved in Subsection \ref{subsec:inter_2}. 
\begin{definition}\label{def_set}
Let $\mathcal R_0$ denote the set of pairs 
$(\rho^-,\rho^+)\in [0,2]^2\setminus\mathcal D$ 
satisfying the following conditions: 
(i) $|\rho^+-\rho^-|=1$; 
  (ii) $G(\rho^+)=G(\rho^-)
  =\min_{\rho\in[\rho^-,\rho^+]}G(\rho)$ if $\rho^-<\rho^+$; or
$G(\rho^+)=G(\rho^-)=\max_{\rho\in[\rho^+,\rho^-]}G(\rho)$ 
if $\rho^->\rho^+$,  where $G$ is defined by 
 \eqref{def_macroflux_ladder}--\eqref{eq:G-from-rho_0}.   
\end{definition}
\begin{remark}\label{rem:entropy}
Condition (ii) in Definition \ref{def_set} 
means that $(\rho^-,\rho^+)$ is 
\textit{an entropy shock} for the scalar conservation law 
with flux function $G$, 
that is the expected hydrodynamic equation  of our model 
for the {\em total} density (that is the sum of densities 
over all lanes), see \cite{abbs2}.  Thus $\mathcal R_0$ 
is exactly the set of entropy shocks of amplitude $1$. 
\end{remark}
\begin{lemma}\label{lemma:r_finite}
 Outside \eqref{unless}--\eqref{unless_3}, the set 
 $\mathcal R_0$ contains at most three elements. More precisely:  \newline
(i)   If $q>0$ and $\gamma_0+\gamma_1\neq 0$,  $\mathcal R_0$ 
contains   one  element,  and $\mathcal B_1\cap\mathcal R_0=\emptyset$.  \\
 (ii)  If  $q>0$, $p\neq q$ and 
 $\gamma_0+\gamma_1=0\neq\gamma_0\gamma_1$,  or if $q=0$ and 
 $\gamma_0\neq\gamma_1$,  $\mathcal R_0$ 
contains  two elements, and $\mathcal R_0\subset\mathcal B_1$.  \\
 (iii) Assume  $\mathfrak{a}=1/2$,   and recall $r_0$ defined by \eqref{def_r0}.  
 Then $\mathcal R_0=\{(1/2,3/2)\}$ if and only if  
$r\in[r_0,1]$,  $\mathcal R_0=\emptyset$ if and only if  
\be\label{cond_r0_0}
r\in(0,r_0),
\ee
and $\mathcal R_0=\{(3/2,1/2);(0,1);(1,2)\}$ if and only if $r=0$. \newline
(iv)  There exists an open subset $\mathcal Z$ of $[0,1]^2$, 
containing  $\{1/2\}\times(0,r_0)$,  such that $\mathcal R_0=\emptyset$ 
for  $(\mathfrak{a},r)\in\mathcal Z$. 
\end{lemma}
\begin{proposition}\label{cor:step2}
 (i) Assume that a measure $\nu\in\mathcal I$ 
 is a $(\rho^-,\rho^+)$-shock measure. 
 Then $(\rho^-,\rho^+)$ satisfies 
 condition (ii) of Definition \ref{def_set}. \\
(ii)  Assume that in Proposition \ref{prop:step1} 
we have \eqref{etaless} 
or \eqref{etamore}, and $k=1$. 
Then the pair $(\rho^-,\rho^+)$ in Proposition 
\ref{prop:step2} satisfies 
$(\rho^-,\rho^+)\in\mathcal R_0$. \\
 (iii) Under the assumptions of (ii),  suppose in addition that 
 $(\rho^-,\rho^+)\in \mathcal B_1$; then either 
 $\gamma_0+\gamma_1=0$, or $q=0$.
If  $q=0$,  
we are in one of the cases (iv), resp. (v), (vi) 
of  Theorem  \ref{cor:entire}, and  $\nu$ lies in the set  given by
\eqref{def_bl_3_2}, resp. \eqref{def_bl_5}, \eqref{def_bl_3_1}. 
\end{proposition}
\begin{lemma}\label{lemma:nob2}
(i) If $\mu\in\mathcal I$ is a $(\rho^-,\rho^+)$-shock measure 
of amplitude $2$, then $(\rho^-,\rho^+)=(0,2)$;
(ii) Under condition \eqref{cond_nob2}, no such measure exists.
\end{lemma}
%
%
\begin{proof}[Proof of Proposition \ref{prop:real_r_finite}]
\mbox{}\newline\newline
\textit{Proof of (i).}   This follows from 
\textit{(i)} of Lemma \ref{lemma:nob2}.  \newline\newline
\textit{Proof of (ii).}   Statement \textit{(ii), (b)}  
of  Theorem \ref{thm:Characterization_of_invariant_measures}
 follows from \textit{(ii)} of Lemma \ref{lemma:nob2}. 
We turn to statements \textit{(ii), (c)} and \textit{(d)} 
of  Theorem \ref{thm:Characterization_of_invariant_measures}. 
By Definition \ref{def_set_r} 
and \textit{(ii)} of Proposition \ref{prop:step2}, 
$\mathcal R$ and $\mathcal R'$ contain only shocks 
of amplitude $1$ associated 
with stationary shock measures.
By Proposition \ref{cor:step2},  \textit{(i)}, 
and Remark \ref{rem:entropy},
any shock associated with a stationary shock measure 
is an entropy shock.  
Thus by Definitions \ref{def_set_r}, \ref{def_set} 
and Remark \ref{rem:entropy},  
we have $\mathcal R\cup\mathcal R'\subset \mathcal R_0$, 
$\mathcal R\subset\mathcal R_0\setminus\mathcal B_1$ 
and $\mathcal R'\subset\mathcal R_0\cap\mathcal B_1$. 
The results then follow from \textit{(i)} and 
\textit{(ii)} of Lemma \ref{lemma:r_finite} if $q>0$.  
If $q=0$, $\gamma_0\gamma_1\neq 0$ and $\gamma_0\neq\gamma_1$, 
\eqref{flux_degenerate} and Definition \ref{def_set} show that 
$\mathcal R_0$ contains at most two points and
$\mathcal R_0\subset\mathcal B_1$, thus $\mathcal R=\emptyset$ 
and $\mathcal R'\subset\mathcal R_0$.  \newline\newline
 {\em Proof of (iii).} In this case, by \textit{(iii)} of 
 Proposition \ref{cor:step2}, $Bl_1$ is contained in the 
 right-hand side of \eqref{def_bl_3_2}. Each of the two sets 
 on this right-hand side consists of translates of a single measure; 
 the first set contains only $(1,2)$-shock measures and 
 the second one only $(0,1)$-shock measures.
\end{proof}
\subsection{Proof of Theorem \ref{cor:entire}}
\label{subsec:proof_inv_3}
 We start with the 
\begin{proof}[Proof of Lemma \ref{def_blocking_h2}] 
We prove the first equality \eqref{shift_rev_twolane}, 
the proof of the second one being similar.  Let  
$\mathcal X_n:=\{\eta\in\mathcal X:\,H_2(\eta)=2n\}$,  
and $\xi^n$ denote  the element of $\mathcal X_n$  defined by 
\[
\xi^n(z,i)={\bf 1}_{\{z>-n\}};\quad z\in\Z,\,i\in W
\] 
so that $\xi^n=\tau_{n}\xi^0$,  with $\xi^0\in\mathcal X_0$. 
For $\eta,\xi\in\mathcal X_n$, let $\mathcal A$, resp. 
 $\mathcal A'$, 
denote the  set of $x=(z,i)\in V$, where $z\in\Z$ 
and $i\in W$,   for which $\eta(x)=1-\xi(x)=0$, 
resp. $\eta(x)=1-\xi(x)=1$. 
Then  $|\mathcal A|=|\mathcal A'|<+\infty$,  and 
\begin{eqnarray*}
\frac{\nu^{\rho^c_.}(\eta)}{\nu^{\rho^c_.}(\xi)}
 & = &  \prod_{ x\in\mathcal A'}\frac{\rho_x}{1-\rho_x}
 \prod_{x\in\mathcal A}\frac{1-\rho_x}{\rho_x}  \nonumber\\
& = & \left(\frac{p}{q}\right)^{\sum_{z\in\Z,\,i\in W}
i[\eta(z,i)-\xi(z,i)]}\theta^{\sum_{z\in\Z,\,i\in W}z[\eta(z,i)-\xi(z,i)]} 
=:  r(\eta,\xi)
\end{eqnarray*}
The second equality  above  follows from  
$\rho_{z,i}/(1-\rho_{z,i})=c(p/q)^i\theta^z$ 
 for $(z,i)\in V$.  
We apply this to $\eta\in\mathcal X_n$ and $\xi^n$:
\[
\check{\nu}_n(\eta)  =  \frac{\nu^{\rho^c_.}(\eta)}
{\nu^{\rho^c}(\mathcal X_n)} \\
= \frac{r(\eta,\xi^n)}{Z^n} 
\]
where
\[
Z^n:=\sum_{\xi\in\mathcal X_n}r(\xi,\xi^n)
\]
 Thus $\check{\nu}_n$ does not depend on $c$. 
Note that  if $\eta\in\mathcal X_0$, we have 
$\tau_n\eta\in\mathcal X_n$, and 
$r(\tau_n\eta,\xi^n)=r(\eta,\xi^0)$. This implies that 
$Z^n$ does not depend on $n$ and that
\[
\check{\nu}_n(\tau_n\eta)=\check{\nu}_0(\eta)
\]
\end{proof}
 We will need the following lemma, 
 proved in Section \ref{subsec:inter_2}. 
\begin{lemma}\label{lemma_blocking}
Assume $\nu^1,\nu^2\in\mathcal I$ are supported on $\mathcal X_2$, 
and $H_2(\eta)$ defined by \eqref{def_H_2} has the same constant value 
under $\nu^1$ and $\nu^2$. Then, unless $l_0=l_1=q=0<p$,
 that is \eqref{case_deg}, the 
  non-weakly irreducible case, 
it holds that $\nu^1=\nu^2$.
\end{lemma}
\begin{proof}[Proof of Theorem  \ref{cor:entire}]
 We prove here the results stated in  Theorem  \ref{cor:entire}
as well as their complements given in Subsection \ref{sec:blocking}. 
\mbox{}\newline\newline
\textit{ $\bullet$ Preliminaries on $Bl_2$:  Proof of (o).}
Since $\mu\in\mathcal I$, we have  (recalling \eqref{config_total}) 
 \[
 \int L\sum_{ x\in\Z:\, 
 m\leq x\leq n}\overline{\eta}\left(x\right)d\mu(\eta)=0
 \]
 where 
\be\label{difference_currents}
L\overline{\eta}(x)=\tau_{x-1}j(\eta)-\tau_x j(\eta)
\ee
with $j$ defined by  \eqref{def_microflux_ladder_2}.  Hence,
%
 for arbitrary $n,m\in\mathbb{Z}$, we conclude that 
 the quantity  $\mu[\tau_x j(\eta)]$
 is independent of $x$. 
Since $\mu$ is a $(0,2)$-shock measure  (see \eqref{limits_mu}), 
we have 
\[
\lim_{n\to+\infty}\mu[\eta(n,i)]=1,\quad
\lim_{n\to-\infty}\mu[\eta(n,i)]=0
\]
Since $0\leq \eta(x,i)(1-\eta(y,i))\leq\min[\eta(x,i),1-\eta(y,i)]$ 
for  $x,y\in \Z$,  this implies
\be\label{zero_current}
\lim_{n\to+\infty}\mu[\eta(x,i)(1-\eta(x+1,i))]
=\lim_{n\to+\infty}\mu[\eta(x+1,i)(1-\eta(x,i))]=0
\ee
for $i\in\{0,1\}$. Thus  $\mu[\tau_x j(\eta)]=0$,  
which can be written
\be\label{zero_current_2}
\mu\left\{
\sum_{i=0}^1 \gamma_i\eta^i(x)[1-\eta^i(x+1)]
\right\}=\mu\left\{\sum_{i=0}^1 l_i[\eta^i(x)-\eta^i(x+1)]\right\}
\ee
Summing \eqref{zero_current_2} over $x\in\Z$ 
and using \eqref{zero_current}, we obtain 
\be\label{zero_current_3}
\sum_{i=0}^1\gamma_i\mu\left\{
\sum_{x\in\Z}\eta^i(x)[1-\eta^i(x+1)]
\right\}<+\infty
\ee
For each $i\in\{0,1\}$,  $\gamma_i>0$, hence the series inside braces 
in \eqref{zero_current_3} converges $\mu$-almost surely. 
Thus, $\mu$-almost surely,
$\eta^i(x)[1-\eta^i(x+1)]\to 0$ as $x\to\pm\infty$ implying  
$\eta^i(x)[1-\eta^i(x+1)]=0$ 
for $|x|$ large enough, and  $\eta\in{\mathcal X}_2$.\newline\newline
\textit{ $\bullet$ Symmetric case on each lane, $\gamma_0=\gamma_1=0$:  
Proof of (i).} Let $\varphi\in C^0_K(\R)$, 
 that is, a continuous function with compact support.  
 We consider the function 
$F_N:\mathcal X\to\R$  defined by
\[
F_N(\eta):=N\sum_{x\in \Z}\varphi
\left(\frac{x}{N}\right)\overline{\eta}(x)
\]
Since $\gamma_0=\gamma_1=0$, i.e. $l_0=d_0$ and $l_1=d_1$, 
the microscopic current \eqref{def_microflux_ladder_2} writes
\be\label{current_diff}
j(\eta)=\sum_{i=0}^1 d_i[\eta^i(0)-\eta^i(1)]
\ee
Using \eqref{difference_currents}, \eqref{current_diff}  
and two summations by parts, we obtain 
\be\label{ipp}
LF_N(\eta)=N^{-1}\sum_{i=0}^1 
d_i\sum_{x\in\Z}\varphi''\left(\frac{x}{N}\right)\eta^i(x)+o_N(1)
\ee 
where $o_N(1)$ is a quantity bounded in modulus 
by a deterministic sequence vanishing as $N\to+\infty$. 
By Theorem \ref{thm:Characterization_of_invariant_measures}, 
$\mu$ satisfies \eqref{limits_mu}, where either 
$\rho^+=\rho^-$ and $\mu$ is a product measure 
given by Theorem \ref{thm:characterization_lemma}, or $\rho^+\neq\rho^-$
and $\mu$ is a $(\rho^-,\rho^+)$-shock measure. 
We show that we are in the first situation. 
Indeed, \eqref{limits_mu} implies
\be\label{limits_ends}
\lim_{x\to\pm\infty}
\int_{\mathcal X}\eta^i(x)d\mu(\eta)=\widetilde{\rho}_i(\rho^\pm)
\ee
Thus, taking the expectation of \eqref{ipp} 
and using stationarity of $\mu$, we have
\be\label{conclude_diff}
0=\int_{\mathcal X} LF_N(\eta)d\mu(\eta)= \varphi''(0)  
\sum_{i=0}^1 d_i\left[
\widetilde{\rho}_i(\rho^-)-\widetilde{\rho}_i(\rho^+)
\right]+\varepsilon_N
\ee
where $\varepsilon_N\to 0$ as $N\to+\infty$. Since $\widetilde{\rho}_i$ 
is increasing and $\varphi$ arbitrary, it follows that 
$\rho^+=\rho^-$. \newline\newline
  \textit{$\bullet$ Case $q>0$: Proof of (ii), (iii).}  \\
\textit{Proof of (ii), (a).} The product measure 
 $\mu_{\rho^c_.}$ is reversible because
 $\rho^c_.$ satisfies the reversibility equations 
\eqref{eq:detailed balance_alpha} with respect to the 
kernel \eqref{eq:intensity_c}. 
The result follows since the measures in \eqref{cond_rev_twolane} 
are defined by conditioning the reversible measure $\mu_{\rho_.}$ on the 
conserved quantity $H_2$. 
Assume now $\check{\nu}_n=(1-\alpha)\nu^1+\alpha\nu^2$ 
with $\nu^1,\nu^2\in\mathcal I$.
 Since $\check{\nu}_n$ is supported  on $\mathcal X_2$ 
and $H_2$ has constant value 
$2n$ under $\check{\nu}_n$, the same holds for $\nu^1$ and $\nu^2$. Thus 
$\nu^1=\nu^2$ by Lemma \ref{lemma_blocking}, implying that 
$\check{\nu}_n$ is extremal. 
The same argument applies to $\widehat{\nu}_n$.  \newline\newline
 \textit{Proof of (ii), (b).} The measure $\check{\nu}_0$ 
 is a product measure of the form  \eqref{alpha_x}, with $\rho_.$ given by
\be\label{rho_case_2_1}
\rho_{x,i}= \check{\rho}_{x,i}:={\bf 1}_{\{x > 0\}},
\quad (x,i)\in\Z\times W,
\ee
 On the other hand, let $\widehat{\mu}_0$ denote 
 the product measure \eqref{alpha_x} where 
\begin{eqnarray}\nonumber
\rho_{x,i} & = & \widehat{\rho}_{x,i}:={\bf 1}_{\{x > 0\}} 
%
+ \rho_0{\bf 1}_{\{( 0,0)\}}(x,i)
 +\rho_1{\bf 1}_{\{( 0,1)\}}(x,i),
 \quad (x,i)\in\Z\times W\label{rho_case_2_2}
\end{eqnarray}
with 
\be\label{rho_case_2_2_bis}
\rho_0:=\frac{c}{1+c},\,\quad
\rho_1:=\frac{\frac{p}{q}c}{1+\frac{p}{q}c};\quad c>0
\ee
 The functions defined by \eqref{rho_case_2_1} and 
\eqref{rho_case_2_2}--\eqref{rho_case_2_2_bis}  are solutions of 
\eqref{eq:detailed balance_alpha}. Thus, $\widehat{\mu}_0$  and  
$\check{\nu}_0$ are reversible. 
Under $\widehat{\mu}_0$, we have a.s. that
\be\label{frozen}\eta(x,i)=1\mbox{ for }x> 0,
\quad\eta(x,i)=0\mbox{ for }x< 0;\quad i\in\{0,1\}\ee
which does not evolve in time.
Hence  under  $\widehat{\mu}_0$,  $\eta(0,0)+\eta(0,1)$ 
is conserved by the evolution,
and conditioning $\widehat{\mu}_0$ on 
  $\{\eta(0,0)+\eta(0,1)=1\}$ 
yields a reversible measure satisfying \eqref{frozen},
under which the vertical layer $\{ 0\}\times \{0,1\}$  
contains a single  particle 
located at $i\in\{0,1\}$ with probability $p_i$ given by
\begin{eqnarray*}
p_0 & = & \frac{\rho_0(1-\rho_1)}{\rho_0(1-\rho_1)
+\rho_1(1-\rho_0)}=\frac{q}{p+q}\\
p_1 & = & \frac{\rho_1(1-\rho_0)}{\rho_0(1-\rho_1)
+\rho_1(1-\rho_0)}=\frac{p}{p+q}
\end{eqnarray*}
This measure is exactly $\widehat{\nu}_0$.  
Note that the process starting 
with \eqref{frozen} and a single particle on  $\{ 0\}\times\{0,1\}$ 
reduces to the two state Markov process 
followed by this single  particle 
jumping between lanes $0$ and $1$, and $\widehat{\nu}_0$ 
reduces to the unique 
 invariant measure of this process (which is reversible). \newline\newline
For the measures $\check{\nu}_0$ and $\widehat{\nu}_0$,  
the proof of extremality 
in \textit{(ii), (a)} also applies here. Finally, 
by Theorem \ref{thm:Characterization_of_invariant_measures}, 
\textit{(ii), (a)}, 
the above measures are  (modulo horizontal translations)  
 the only elements of $Bl_2$.  \newline\newline
 \textit{Proof of (iii).} 
That $\mathcal R=\mathcal R'=\emptyset$ 
follows from Definition \ref{def_set_r}, \textit{(ii)} 
of Proposition \ref{cor:step2} and \textit{(ii)--(iv)} 
of Lemma \ref{lemma:r_finite}. When $d_1=\lambda d_0$ 
and $l_1=\lambda l_0$ with $\lambda$ close to $1$, 
then  $\mathfrak{a}$  is close to $1/2$, and \eqref{charac_sym} follows 
from \eqref{charac_ext} and \textit{(ii)} of Theorem 
\ref{cor:entire} proven above.  \newline\newline
\textit{ $\bullet$ Case $q=0$:  Proof of (iv)--(vi)}  
(end of step six from Subsection \ref{subsec:ideas}). 
We first treat $Bl_1$ with an argument common to 
the three situations. 
Indeed in (\textit{iv)}, resp. \textit{(v)}, \textit{(vi)}, 
by statement \textit{(iii)} of Proposition \ref{cor:step2}, 
any element of $Bl_1$  must belong to the set \eqref{def_bl_3_2}, 
resp. \eqref{def_bl_5}, \eqref{def_bl_3_1}.
Conversely, elements of these sets are extremal invariant 
probability measures in each case  and they are reversible.
Reversibility follows from the fact that the process starting 
from these measures is isomorphic to an ASEP on the  lane 
that is not full or empty, and its restriction to this lane is 
a blocking measure for this ASEP, known to be reversible. 
We detail the extremality  argument for $\nu^{\bot,+\infty,n}$ 
in case \textit{(iv)}, all others are similar. 
Assume $\nu^{\bot,+\infty,n}=(1-\alpha)\nu^1+\alpha\nu^2$, 
with $\nu^1,\nu^2\in\mathcal I$ and 
$\alpha\in(0,1)$. Since lane $0$ is empty under 
$\nu^{\bot,+\infty,n}$, the same holds 
under $\nu^1$ and $\nu^2$. Thus under the three measures, 
lane $1$ evolves as an autonomous 
SEP with jump rate $d_1$ to the right and $l_1$ to the left, i.e., 
\eqref{def_tasep}  with transition kernel  
\eqref{eq:generator of the Exclusion} with $(d,l)=(d_1,l_1)$. 
The marginal of each measure on lane $1$ is then 
an invariant measure for this  SEP. 
Since the marginal of  $\nu^{\bot,+\infty,n}$ 
is an extremal (blocking) invariant measure 
for the SEP on lane $1$, we have 
$\nu^1=\nu^2=\nu^{\bot,+\infty,n}$. Since lane $0$ 
remains empty under the evolution, $\nu^{\bot,+\infty,n}$ 
is indeed an invariant measure for the two-lane SEP.\newline\newline
We next treat $\mathcal R$ and $Bl_2$ with arguments 
specific to each situation. \newline\newline
\textit{Proof of (iv), (a).} 
This corresponds to  \eqref{flux_degenerate},  
that is, example \ref{ex:periodic}. 
 Note first that, since $\gamma_0>0$ and $\gamma_1>0$,
\be\label{shocks_in_B}
(\rho^-,\rho^+)=(0,1);\quad (\rho^-,\rho^+)=(1,2)
\ee
satisfy the conditions of Definition \ref{def_set}. These two  
 shocks belong to $\mathcal B_1$. \newline\newline
 If $\gamma_0=\gamma_1$, we have $G(\rho+1)=G(\rho)$ 
 for every $\rho\in[0,1]$; thus  $(\rho^-,\rho^+)=(\rho,\rho+1)$ 
 and $(\rho^-,\rho^+)=(\rho+1,\rho)$ 
satisfy $G(\rho^+)=G(\rho^-)$ for every $\rho\in[0,1]$. 
Among such shocks 
different from the ones in \eqref{shocks_in_B},  
only $(\rho^-,\rho^+)=(3/2,1/2)$ satisfies
the variational equality in condition \textit{(ii)} 
of Definition \ref{def_set}. 
Thus $\mathcal R_0\subset \{(3/2,1/2);(0,1);(1,2);(0,2)\}$, 
and  $\mathcal R\subset\{(3/2,1/2)\}$.  \newline\newline
 If  $\gamma_0\neq \gamma_1$, then for every $\rho\in(0,1)$, 
 we have $G(\rho+1)\neq G(\rho)$. 
 Thus $\mathcal R_0$ contains no other shock 
 than those in \eqref{shocks_in_B}.
 Hence $\mathcal R=\emptyset$. \newline\newline
We finally prove  by contradiction that $Bl_2$ is empty 
unless $l_0=l_1=0$. Let  $\mu\in Bl_2$. The function  
\be\label{def_h_lane_1}
 H_{\mathbb{L}_1}(\eta):=\sum_{z\in\Z:\,  z\leq 0}\eta(z,1)
 -  \sum_{z\in\Z:\,  z>0}[1-\eta(z,1)]
\ee
is well-defined since $\mu$ is supported on $\mathcal X_2$. It 
is constant along horizontal jumps  and is increased  
by vertical jumps from lane $0$ to lane $1$.   
 Let $(\eta_t)_{t\geq 0}$ denote the stationary process such that $\eta_0\sim\mu$. 
We claim and prove below that if $l_0>0$,  
there is a positive probability that by time $1$, 
the leftmost particle initially on lane $0$ has jumped to lane $1$. 
This implies
\be\label{contra_h1}
\Exp_\mu\left[H_{\mathbb{L}_1}(\eta_1)-H_{\mathbb{L}_1}(\eta_0)\right]>0
\ee
which contradicts stationarity.  
Similarly,  if $l_1>0$,  there is a positive probability 
that by time $1$, the leftmost particle on lane $1$ 
has jumped to lane $0$, 
which implies the reverse strict inequality in \eqref{contra_h1}.
 \begin{remark}\label{rem:not-confused} 
In \eqref{contra_h1}, $\eta_0$ and $\eta_1$ denote 
the two-lane process configurations at respective times $0$ and $1$. 
These should not be confused with $\eta^0$ and $\eta^1$, 
cf. \eqref{config_lane}, which denote the restriction of 
a two-lane configuration $\eta$ to lane $0$ and $1$ respectively. 
\end{remark} 
We now  prove the claim when $l_0>0$ 
(the proof in the case $l_1>0$ is similar). 
In the sequel, on each lane, we call \textit{active} 
those particles initially on the left 
of the rightmost hole and the next particle 
to the right of this hole (we also call active 
those sites where active particles are initially sitting). 
For $x,y\in V$, 
we say a Poisson process $\mathcal N_{(x,y)}$ 
of the Harris construction is 
\textit{attached} to  some site $z\in V$ if $z\in\{x,y\}$. 
We condition $\mu$ on the number and positions 
of active  particles on each lane. 
Denote respectively by $x_0$, $y_0$, $x_1$ the initial positions 
of the leftmost particle 
on lane $0$, the next particle on its right, 
and the leftmost particle on lane $1$. 
We couple our two-lane SEP with a random walk on lane $0$ 
starting from $(x_0,0)$, 
that jumps to the right and left with respective rates 
$d_0,l_0$ and is reflected at $(y_0,0)$. 
The random walk is defined from the Harris system as follows: 
if its current position 
is $(x,0)\in V$, at the first point of a Poisson process 
$\mathcal N_{(x,0),(x+\varepsilon,0)}$ 
where $\varepsilon\in\{-1,1\}$, it jumps to $x+\varepsilon$, 
except if $x=y_0-1$ and $\varepsilon=1$.
Let $x'_0:=\min(x_0,x_1-1,y_0-2)$, and  $E_0$ 
denote the event that the random walk hits 
$(x'_0,0)$ for the first time before time $1/2$ 
and stays there at least until time $1$
 (if $x'_0=x_0$, $E_0$ corresponds to the return time to $(x_0,0)$).  
This event has positive probability and depends only 
on the Poisson processes 
$\mathcal N_{(x,x+1)}$  and $\mathcal N_{(x+1,x)}$   
for $x'_0-1\leq x\leq y_0-2$.
Let $T_0$ denote  the first time  among all 
the following Poisson processes: \\
(a) $\mathcal N_{(y_0,0),(y_0-1,0)}$; \\
(b) $\mathcal N_{(x,0),(x,1)}$ for $x'_0< x\leq y_0-1$; and \\
(c) the Poisson processes attached to active sites on lane $1$; $T_0$ 
is an exponential random variable. Consider the event
\begin{eqnarray*}
E'_0 & := & E_0\cap\{T_0>1\} 
\\
&& \cap  
 \{\mathcal N_{(x'_0,0),(x'_0,1)}
\mbox{ has at least one point in the time interval }[1/2,1] \}
\end{eqnarray*}
On $E_0$, in the two-lane SEP starting from the conditioned measure,  
the particle initially at $(x_0,0)$  coincides 
with the random walk until  it
reaches $(x'_0,0)$; then its next motion is  a jump 
from there to $(x'_0,1)$ 
before time $1$; all this occurs before any particle 
initially on lane $1$ has moved  
and before the particle initially at $(y_0,0)$ has moved. 
The  three  events 
of which $E'_0$ is the intersection are independent, because
they depend on disjoint sets of Poisson processes. 
It follows that $E'_0$ has positive probability when starting 
from the conditioned measure, irrespective of the conditioning. \newline\newline
\textit{Proof of (iv), (b).} Assume $l_0=l_1=0$. 
The measure $\nu^{\bot,i,j}$ is shown to be reversible 
as in \textit{(ii)} above, since it is a product measure 
of the form \eqref{alpha_x} with 
\[
\rho_{x,0}={\bf 1}_{\{x>i\}},\quad \rho_{x,1}={\bf 1}_{\{x>j\}}
\] 
for $x\in\Z$,
and the above function $\rho_.$ satisfies the reversibility conditions 
\eqref{eq:detailed balance_alpha}.
The measure $\nu^{\bot,i,j}$ is a Dirac measure, hence it is extremal 
in the set of probability measures on $\mathcal X$, 
thus \textit{a fortiori} extremal in $\mathcal I$.  \newline\newline 
Now we prove that any element  $\mu$ of $Bl_2$ 
is of the form $\nu^{\bot,i,j}$ 
 for $(i,j)\in\mathbb{B}$.  Indeed, since $\mu$ is supported on 
$\mathcal X_2$, the random variable
\[
n_0:=\inf\{z\in\Z:\, \eta(z,i)=1,\quad\forall z\geq n_0,\,i\in\{0,1\}\}
\]
is $\mu$-a.s. finite, as well as the number of particles 
to the left of $n_0$. 
Since jumps are totally asymmetric both horizontally 
and vertically, conditioned 
on this number and on $n_0$, the process lives on a finite space, and 
its irreducible classes are singletons containing states 
$\{\eta^{\bot,i,j}\}$  that can be reached from the initial state 
(indeed, states $\eta^{\bot,i,j}$ 
are the only ones from which no transition is possible, 
and no return is possible from a state not belonging to this class).
 This implies that $\mu$ is a mixture of the invariant measures 
 $\nu^{\bot,i,j}$, and by extremality, it must be one of them. \newline\newline 
\textit{Proof of (v).} 
The flux function $G$ has the form \eqref{flux_degenerate}, 
 that is, example \ref{ex:periodic}.  Since $\gamma_1<0<\gamma_0$, 
 the only pairs $(\rho^-,\rho^+)$  
 satisfying the requirements of  Definition \ref{def_set}  are 
$(\rho^-,\rho^+)=(1,0)$ and $(\rho^-,\rho^+)=(1,2)$. 
These shocks belong to 
$\mathcal B_1$, hence $\mathcal R=\emptyset$. We next prove 
that $Bl_2$ is empty. 
Indeed by statement \textit{(i)} of Proposition \ref{cor:step2}, 
a $(\rho^-,\rho^+)$-shock measure must satisfy condition \textit{(ii)} 
of Definition \ref{def_set}. This is not the case for 
$(\rho^-,\rho^+)=(0,2)$ or $(\rho^-,\rho^+)=(2,0)$,  
because by \eqref{flux_degenerate}, $G(0)=G(2)$ 
and $\gamma_1<0<\gamma_0$ implies that $G$ is negative on $(0,1)$ 
and positive on $(1,2)$. \newline\newline
\textit{Proof of (vi).} The flux function $G$ 
has the form \eqref{flux_degenerate}, 
 that is, example \ref{ex:periodic}. 
The proof that $\mathcal R=\emptyset$ is similar 
to the case $\gamma_0\neq\gamma_1$ 
in  \textit{(iv), (a)} (the fact that $\gamma_0=0$ 
being irrelevant there).  \newline\newline
The proof that $Bl_2$ is empty  can be reduced as follows 
to the same argument as 
in  \textit{(iv), (a)} above. First, note that
$\gamma_0=0$ implies $l_0>0$. Unlike in \textit{(iv), (a)}, 
since $\gamma_0=0$, we cannot use \textit{(o)}, i.e. we
do not know whether $\mu$ is supported on $\mathcal X_2$.
Nevertheless, since $\gamma_1>0$, partially repeating the 
proof of \textit{(o)} shows that $\mu$ is supported on the set 
of configurations $\eta$ for which $H_{\mathbb{L}_1}(\eta)$, 
see \eqref{def_h_lane_1}, is finite. 
 Let now $\eta'_0$ be the random configuration 
obtained by removing from $\eta_0\sim\mu$ all particles 
at sites $(z,0)$ such that $z<0$, and $(\eta'_t)_{t\geq 0}$ 
be the process starting from $\eta'_0$. Then 
\be\label{h_with_etaprime}
 H_{\mathbb{L}_1}(\eta'_0)=H_{\mathbb{L}_1}(\eta_0),\quad 
 H_{\mathbb{L}_1}(\eta'_1)\leq H_{\mathbb{L}_1}(\eta_1)\ee
where the inequality follows from attractiveness.
We can then repeat the argument in \textit{(iv), (a)} 
for $(\eta'_t)_{t\geq 0}$ to infer \eqref{contra_h1} 
for the latter process. Using \eqref{h_with_etaprime} 
we obtain \eqref{contra_h1} for $(\eta'_t)_{t\geq 0}$. 
\end{proof}
\subsection{Proof of 
Theorem 
\ref{thm:Characterization_of_invariant_measures_gen}}\label{subsec:proof_gen}
\textit{Proof of   (o).  } The proof of 
Theorem \ref{thm:characterization_lemma} 
requires only minor changes. First we can repeat the proof 
of Lemma \ref{lem:invariant_measures_for_the_asymetric_ladder}. 
The only difference 
is that on a vertical layer $\{z\}\times W$, we now use 
the fact that $L^z_v$ 
is the generator of a translation-invariant SEP on a torus, 
and $\nu_\rho$ is a homogeneous product Bernoulli measure, 
which is thus invariant for $L^z_v$.
The rest of the proof 
is exactly similar to that of Theorem 
\ref{thm:characterization_lemma}. Note that here 
by \textit{(ii)} of Lemma \ref{lemma_wi}, 
a $p$-ordered pair of configurations is ordered, 
so we do not need analogues of Lemmas \ref{lemma:irred} 
and \ref{lemma:irred2} in this context.  \newline\newline
\textit{Proof of   (i).  }  \textit{(a)}.  
We can repeat the following steps of 
the proof of Theorem \ref{thm:Characterization_of_invariant_measures}:
Proposition \ref{prop:step1} (leading to \eqref{etaless}--\eqref{equal} 
with $k\in\{1,\ldots,n\}$ instead of 
$k\in\{1,2\}$ in \eqref{number_disc}, 
\eqref{case_bowtie} is irrelevant here because $p(.,.)$ 
is weakly irreducible), 
Corollary \ref{cor:step1}, Proposition \ref{prop:step2} 
and Proposition \ref{prop:translate}. 
This yields that an extremal invariant measure 
that is not invariant by horizontal translations 
is a shock measure whose amplitude 
lies in $[1,n]\cap\Z$.  Similarly to Proposition 
\ref{prop:translate}, we can prove that there are 
at most (up to translations) $k$ shock measures of amplitude $k$.  
To further characterize possible shocks 
$(\rho^-,\rho^+)$, we  consider the macroscopic flux 
function $G$  defined by \eqref{def_microflux_ladder} and 
\eqref{def_macroflux_ladder} in this setup. 
By definition \eqref{def_inv_gen} of $\nu_\rho$, 
\[
G(\rho)=\left(\sum_{i=0}^{n-1}\gamma_i\right)
\frac{\rho}{n}\left(1-\frac{\rho}{n}\right)
\]
We can then repeat the proof of statement \textit{(i)} 
of Proposition \ref{cor:step2}. 
Since the above function $G$ is strictly concave and 
symmetric around $\rho=n/2$, 
shocks satisfying condition \textit{(ii)} of 
Definition \ref{def_set} are those 
specified in the theorem.\newline 
\textit{(b)}. The proof is similar to 
Theorem \ref{cor:entire}, \textit{(o)}.\newline
\textit{(c)}
Stationarity of the product measure $\nu^{\rho^c_.}$ 
is proved as in Lemma 
\ref{lem:invariant_measures_for_the_asymetric_ladder},
observing that on vertical layers we have a periodic SEP 
for which a homogeneous product measure is invariant. 
Stationarity and extremality of the conditioned measure 
are proved as in Theorem \ref{cor:entire}, \textit{(ii)}. \newline\newline
\textit{Proof of   (ii).  } We can repeat with minor modifications the proof 
of statement \textit{(i)} 
of Theorem \ref{cor:entire}. The microscopic current 
is as in \eqref{current_diff}. 
Summation there and in \eqref {limits_ends}--\eqref{conclude_diff} 
is now over 
$i\in W:=\{0,\ldots,n-1\}$. In the latter two displays,
$\widetilde{\rho}_i(\rho)$ is replaced by $\rho/n$.  \newline\newline
\textit{Proof of   (iii).  } Let $\mu\in\mathcal I_e$. 
If $\mu\in\mathcal S$, 
the conclusion follows from \textit{(o)}.
Otherwise, since the generator 
\eqref{eq:generator of the Exclusion} is invariant 
by $\tau'$, $\mu':=\tau'\mu\in{\mathcal I}_e$. 
By \textit{(1)},  $\mu$ and $\mu'$  are 
shock measures, and $\mu'=\tau'\mu$ implies 
that they are $(\rho^-,\rho^+)$-shock 
measures for the same pair $(\rho^-,\rho^+)$.  
The proof of Proposition 
\ref{lemma:translate} carries over to the multilane model 
(notice indeed that under condition \textit{(iii')}, 
the global kernel $p(.,.)$ is weakly irreducible; thus
when repeating the part of the proof of 
Proposition \ref{prop:step1} that is used to derive Proposition 
\ref{lemma:translate}, we always obtain 
\eqref{eq:blocking_measures-4}, and do not need 
an analogue of step three). 
Hence, we have either $\mu\leq\mu'$ or $\mu'\leq\mu$. 
Since $\tau'$ is a periodic shift, this implies $\mu'=\mu$. 
 \section{Proofs of 
intermediate results from Subsections 
\ref{subsec:proof_inv}--\ref{subsec:proof_inv_3}}
\label{subsec:proofs_lemmas} 
\subsection{Proofs of intermediate results 
from Subsection \ref{subsec:proof_inv}}\label{subsec:inter_1}
\begin{proof}[Proof of Proposition \ref{prop:step1}]
\mbox{}\newline\newline
\textit{Proof of (i).}  Cases \eqref{etaless}--\eqref{equal} 
are an adaptation of 
\cite[Proposition 3.2]{Bramson2002},  the main ingredients 
of which we recall
in steps one and two below, whereas 
an additional argument (step three below) is required for \eqref{case_bowtie}.  
 Let us fix $T>0$. \newline\newline
\textit{Step one.} 
 Let $\widetilde{\lambda}_0$  denote the distribution 
 on $\mathcal{X}\times\mathcal{X}$
of the coupled configuration  $(\eta_0,\xi_0)$, 
where $\eta_0\sim\mu$ and $\xi_0=\tau\eta_0$. 
We denote by $(\eta_t,\xi_t)_{t\ge 0}$  the coupled process 
starting from $(\eta_0,\xi_0)$. 
Define 
\begin{equation}
\mathcal{R}_{T}=\left\{ x\in V:-T\leq x\left(0\right)
\leq T\right\} \label{eq:blocking_measures-5}
\end{equation}
 Let $\mathcal{N}_{T}$ be the number of discrepancies of  
 $(\eta_t,\xi_t)_{t\ge 0}$ 
 that visit $\mathcal{R}_{T}$ at any time in $[\sqrt{T},T]$,  
 $\mathcal N_T^{in}$ 
 the number of these starting from $[-(1+\sigma)T,(1+\sigma)T]$
(where $\sigma$ is the constant in Proposition \ref{prop:prop}), and 
$\mathcal N_T^{out}$ the number of these starting outside this interval.
Adapting the proof in \cite[Proposition 2.5]{Bramson2002} 
to our model  yields 
 \begin{equation}
 \mathbb{E}^{\widetilde{\lambda}_{0}}
 \left(\mathcal{N}_{T}\right)=o\left(T\right)
 \qquad  \hbox{when}\quad
  T\rightarrow\infty.\label{eq:blocking_measures-2}
 \end{equation}
 The proof of \cite[Proposition 2.5]{Bramson2002} 
 used only the following properties 
 of single-lane SEP, which hold also for our two-lane model. \\
 (a) The finite propagation property (Proposition \ref{prop:prop}) 
 is used to show 
\be\label{discrepancies_out}
\mathbb{E}^{\widetilde{\lambda}_{0}}
\left(\mathcal{N}_{T}^{out}\right)=o\left(T\right)
\ee
(b)  the invariance of the generator with respect to 
horizontal translations, 
and \\
(c) the characterization theorem (here  
Theorem  \ref{thm:characterization_lemma}) 
for stationary measures invariant with respect to 
such translations: these are used to show
\be\label{discrepancies_in}
\mathbb{E}^{\widetilde{\lambda}_{0}}
\left(\mathcal{N}_{T}^{in}\right)=o\left(T\right)
\ee
 \textit{Step two.}
  For $x,y\in\Z$, 
  let $\mathcal N_T^{x,y}$ denote 
  the number of discrepancies that visit either $x$ or $y$ 
  and disappear during
the time interval $[\sqrt{T},T]$. Recall the definition 
\eqref{def_dk} of $E_{x,y}$, and define
\[
e_{x,y}:=\inf_{(\eta,\xi)\in E_{x,y}}\Prob_{(\eta,\xi)}\left(
\mbox{one of the discrepancies at $x$ and $y$ has coalesced by time $1$}
\right)
\]
where $\Prob_{(\eta,\xi)}$ denotes the law of the coupled process 
starting from $(\eta,\xi)$. The same argument as in 
\cite[Lemma 3.1]{Bramson2002} shows that 
\be\label{prob_connect}
e_{x,y}>0\mbox{ if }x\rightarrow_p y\mbox{ or }y\rightarrow_p x
\ee 
Let
 \be\label{time_average_measure}
  	\widetilde{\lambda}^{T}=\frac{1}{T-\sqrt{T}}
  	\int_{\sqrt{T}}^{T}\widetilde{\lambda}_0\widetilde{S}_{t}dt
 \ee 
 and let $\widetilde{\lambda}
 =\lim_{i\rightarrow\infty}\widetilde{\lambda}^{T_{i}}$ 
 be a subsequential weak   limit. Then 
 \be\label{eq:tildelambdainvt}
 \widetilde{\lambda}\in\widetilde{\mathcal I}
 \ee  
 Since $\mu\in\mathcal I_e$ 
 and the two-lane  SEP   is translation-invariant 
 in the $\Z$-direction, we have 
 $\tau\mu\in\mathcal I_e$. 
Since $\widetilde{\lambda}_0$ has marginals 
$\mu\in\mathcal I$ 
and $\tau\mu\in\mathcal I$, $\widetilde{\lambda}$ 
has marginals $\mu$ and $\tau\mu$. 
   As in \cite[Proposition 3.2]{Bramson2002}, \eqref{eq:blocking_measures-2} 
   and the strong Markov property yield respectively 
   the following equality and inequality:
\be\label{bystepone}
0=\liminf_{T\to+\infty}\frac{1}{T}\Exp_{\widetilde{\lambda}_0}
\left(\mathcal N_T^{x,y}\right)\geq e_{x,y}\widetilde{\lambda}(E_{x,y})
\ee
Combining \eqref{prob_connect} and \eqref{bystepone}, we obtain
 \begin{equation}
\label{eq:blocking_measures-4}
\widetilde{\lambda}\left\{(\eta,\xi)
\mbox{ is }p-\mbox{ordered}\right\}=1
 \end{equation}
 that is, \eqref{p_ordered_conf}. 
In the case  $q>0$,  by \textit{(i)} of Lemma \ref{lemma:irred},  
\eqref{eq:blocking_measures-4} implies \eqref{ordered_coupling}. 
When  $q=0$,  we only arrive at \eqref{notquite}.\newline\newline
%
 \textit{Step three.}   
Assuming  $q=0$,  we prove below that 
\be\label{wefirstprove}
\widetilde{\lambda}(E_{\supinf}\setminus E_{\bowtie})=0
\ee  
This together with \eqref{notquite} implies
\be\label{implies}
\widetilde{\lambda}(
E_0\cup E_1\cup E_2\cup E_{\bowtie}
)=1
\ee
Moreover, each of the events  in \eqref{implies} is invariant 
under the coupled dynamics. 
Then using the fact that $\mu$ and $\tau\mu$ lie in $\mathcal I_e$, 
we can conclude as in \cite[Proposition 3.2]{Bramson2002} 
that $\widetilde{\lambda}$ actually satisfies one of the  conditions 
\eqref{etaless}--\eqref{case_bowtie}. \newline\newline
We now prove the claim \eqref{wefirstprove}. 
Recall the random variables $X,Y$ 
defined by \eqref{rightmost}--\eqref{leftmost}.  Then, by 
conditions \textit{(ii)--(iii)} of Definition \ref{def_bowtie}, 
\be\label{decomp_bowtie}
E_{\supinf}\setminus
E_{\bowtie}\,\subset\bigcup_{x,y\in\Z:\,x<y}E'_{\bowtie,x,y}
\cup\bigcup_{x,y\in\Z:\,x<y}F'_{\bowtie,x,y}
\ee
where, for $x<y$, 
\begin{eqnarray*}
E'_{\bowtie,x,y} & := & E_{\supinf}\cap
\{X=x,\,Y=y\}
\\
 &&
  \cap \, \{
\mbox{There are at least }y-x\mbox{ holes on lane }1
\mbox{ to the right of }x
\}\\ 
F'_{\bowtie,x,y} & := &  E_{\supinf}\cap
\{X=x,\,Y=y
\}
\\
 &&  
 \cap \, \{
\mbox{There are at least }y-x\mbox{ coupled particles 
on lane }0\mbox{ to the
}\\ & & \phantom{\{}\mbox{
left of }y
\}
\end{eqnarray*}
 We claim that 
$\widetilde{\lambda}(E'_{\bowtie,x,y})
=\widetilde{\lambda}(F'_{\bowtie,x,y})=0$ 
 which, in view of \eqref{decomp_bowtie}, implies  
 \eqref{wefirstprove}. 
On  $E'_{\bowtie,x,y}$,  there is a possible sequence 
of moves with positive probability  
that brings the discrepancy from $(x,1)$ to $(y,1)$  
that is $p$-connected 
to  $(y,0)$. Indeed one can construct an event 
on the Harris system prescribing 
that on the time interval $[0,1]$, the corresponding Harris clocks
will ring in the desired order while no other clock rings. Hence, 
by stationarity, $\widetilde{\lambda}(E'_{\bowtie,x,y})>0$ implies
$\widetilde{\lambda}(E_{(y,0),(y,1)})>0$, in contradiction with 
\eqref{eq:blocking_measures-4}.  Similarly on  $F'_{\bowtie,x,y}$, 
there is a possible sequence of moves with positive probability  
that brings the discrepancy from $(y,0)$ to $(x,0)$  
that is $p$-connected to  $(x,1)$. \newline\newline
\noindent \textit{Proof of (ii).}  
Since the coupled configurations $\eta$ and $\xi$ are a.s. ordered under 
$\widetilde{\lambda}$, all discrepancies (if any) are of the same type
(that is  $\eta$ or $\xi$ discrepancies), so no coalescence occurs.
Thus,  recalling the definition of $D(\eta,\xi)$ 
from \eqref{def:Dmn},  the sets $A_k:=\{D(\eta,\xi)=k\}$,
 with $k\in\mathbb{N}\cup\left\{ +\infty\right\}$,  
 are invariant under the dynamics. Hence, 
\be\label{decomp_2}
\widetilde{\lambda}=\sum_{k\in\N\cup\{+\infty\}:\,
\widetilde{\lambda}(A_k)>0}\widetilde{\lambda}(A_k)
\widetilde{\lambda}_k
\ee
where $\widetilde{\lambda}_k:=\widetilde{\lambda}(.|A_k)\in\widetilde{\mathcal I}$. 
 Since $\mu$ and $\tau\mu$ are extremal elements of $\mathcal I$,  
 for each $k$ such that $\widetilde{\lambda}(A_k)>0$, 
$\widetilde{\lambda}_k$ has marginals $\mu$ and $\tau\mu$.
Assume for instance that $\widetilde{\lambda}$ 
(and thus $\widetilde{\lambda}_k$) 
satisfies \eqref{etaless}. Then
\begin{eqnarray}
\nonumber
{\widetilde{\lambda}_k} 
 [D_{m,n}(\eta,\xi)]  
& = & {\widetilde{\lambda}_k}\left\{
\sum_{x\in V:\,m\leq x(0)\leq n} [\xi(x)-\eta(x)]\right\}
\\ & = & 
 \mu[\overline{\eta}(n+1)]-\mu[\overline{\eta}(m)]\in [0,2]
\label{whereweused}
\end{eqnarray}
Letting $m\to-\infty$ and  $n\to+\infty$,  by monotone convergence, 
 and because $\widetilde{\lambda}_k$ is supported on $A_k$,  we obtain 
\be\label{eq:supportedon}
k={\widetilde{\lambda}_k}[D(\eta,\xi)]\in\{0,1,2\}
\ee 
 Notice that the right-hand side of \eqref{whereweused}, 
 and thus also its limit, 
 depends only on $\mu$. Hence $k$ depends only on $\mu$.  
 This shows that $\widetilde{\lambda}=\widetilde{\lambda}_k$ 
 for a unique $k\in\{0,1,2\}$. 
  Since we are in case \eqref{etaless}, $k=0$ 
  would yield a contradiction.
 Thus $k\in\{1,2\}$. \\
 Dealing with the case \eqref{etamore} is similar. 
\end{proof}
\begin{proof}[Proof of Corollary \ref{cor:step1}]
\textit{(i)} The marginals of $\widetilde{\lambda}$ are 
$\mu$ and $\tau \mu$, 
thus $\mu\leq\tau\mu$ in case \eqref{etaless}, or $\tau\mu\leq\mu$ 
in case \eqref{etamore}, or $\tau\mu=\mu$ in case \eqref{equal}.\newline\newline
\textit{(ii)}  Given the assumptions,  
the limits \eqref{limits_mu_2} exist and satisfy 
 $\tau\widehat\mu_\pm=\widehat\mu_\pm$, that
is $\widehat\mu_\pm\in\mathcal S$. Besides, 
we have $\widehat\mu_\pm\in\mathcal I$. 
Indeed if $f$ is a local function on $\mathcal X$,
\begin{eqnarray*}
\int_{\mathcal X}Lf(\eta)d\widehat\mu_\pm(\eta) 
& = & \lim_{n\to\pm\infty}\int_{\mathcal X}
Lf(\eta)d(\tau_n\widehat\mu)(\eta)
\\
& = & 
\lim_{n\to\pm\infty}
\int_{\mathcal X}L[\tau_n f](\eta)d\widehat\mu(\eta)=0
\end{eqnarray*}
where we used that $L$ commutes with the shift 
and $\widehat\mu\in\mathcal I$. 
The last equality in \eqref{limits_mu_2} follows 
from \thmref{characterization_lemma} 
and \eqref{sameset}.
\end{proof}
\begin{proof}[Proof of Proposition \ref{prop:bowtie}]
 Recall from Proposition \ref{prop:step1} that \eqref{case_bowtie} 
 may only occur when  $q=l_0=l_1=0$. 
Hence the dynamics of horizontal jumps on each lane is a TASEP, 
and these TASEP's interact through vertical jumps from lane $0$ 
to lane $1$. 
Let  $(\eta_0,\xi_0)=(\eta,\xi)\sim\widetilde{\lambda}$, 
where $\widetilde{\lambda}$  
is the measure in Proposition \ref{prop:step1}.
We couple the process $\eta_.$ through basic coupling with 
a process $\zeta_.$ 
such that for every $i\in\{0,1\}$, $\zeta^i_t$ is 
a TASEP  on lane $i$ 
starting from configuration $\eta^i_0=\eta^i$ 
 (with jumping rates  $l_i=0<d_i$).  Then one has, 
for every $t\geq 0$, 
\be\label{compare_tasep}
\eta^1_t\geq\zeta^1_t,\quad \eta^0_t\leq\zeta^0_t
\ee
Indeed,  to derive the first inequality  in \eqref{compare_tasep}, 
note that 
at certain random times belonging to one of the Poisson processes 
$\mathcal N_{(z,0),(z,1)}$, a new particle may   appear (following 
a jump of a particle from lane $0$)  at site $z\in\Z$ in $\eta^1_.$ 
that  does not appear  in $\zeta^1_.$. On the other hand, 
between such  times, 
both processes evolve as coupled TASEP's
on lane $1$, whose order is preserved by attractiveness 
property \eqref{attractive}. 
A similar argument holds for the second inequality  
in \eqref{compare_tasep}.  \newline\newline
For $t>0$, define the empirical measures
\be\label{empirical_measures}
M_t^i:=\frac{1}{t}\int_0^t \mu_s^i ds,
\quad N_t^i:=\frac{1}{t}\int_0^t \nu_s^i ds
\ee
where $\mu_t^i$ denotes the law of $\eta_t^i$ and $\nu_t^i$ 
that of $\zeta_t^i$. 
Since  $\mu\in\mathcal I$,   $\mu_t=\mu$ does not depend on $t$, hence 
$M_t^i=:\mu^i$ does not depend on $t$ and is the marginal 
of $\mu$ on lane $i$. \newline\newline
Let $t_n\uparrow+\infty$ be a subsequence along which
$N^i_{t_n}\to\nu^i_\infty$, where $\nu^i_\infty$ is an 
invariant measure for TASEP. 
Since $\eta\bowtie\xi$,   there is a random variable 
$N\in\Z\cup\{-\infty\}$ such that
 $\zeta^i_0(x)=\eta^i_0(x)=1$ for $i\in\{0,1\}$ and $x\in\Z$ 
 with $x\geq N$. By TASEP dynamics,
this remains true at time $t$ for $\zeta^i_t$ with the same  $N$. 
Thus if $\zeta$ is a random configuration with distribution 
$\nu^i_\infty$, we  a.s. have 
\be\label{shock_for_zeta}\zeta(x)=1,\quad\forall  x\geq N'\ee
where $N'$ has the same law as $N$. As   $\nu^i_\infty$ 
is invariant for TASEP by \cite[Theorem 1.4]{Liggett1976}, 
it is a mixture of Bernoulli and blocking measures. 
But by \eqref{shock_for_zeta},
the only possible Bernoulli measure is the one with density $1$. 
Thus there exists a random variable  
$N_i\in \Z\cup\{-\infty\}$ such that  
the random configuration  $\zeta^i_\infty:=\eta^*_{N_i}$ 
 has distribution $\nu^i_\infty$. \newline\newline
By \eqref{compare_tasep}, $\mu^1\geq\nu^1_\infty$  and $\mu^0\leq\nu^0_\infty$.  
It follows from the above that there exist random variables $M_0$ 
and $M_1$ with values in $\Z\cup\{-\infty\}$ such that 
$\eta^1_0\geq \eta^*_{M_1}$ and $\eta^0_0\leq \eta^*_{M_0}$ a.s. 
Since $\eta\bowtie\xi$, the dynamics of $(\eta^1_t)_{t\ge 0}$ 
can only possibly create 
a finite number of particles (from lane $0$) to  the left of $M_1$  
and move 
these particles to the right until they pile up and get blocked. 
The same argument 
applies to holes in $\eta^0$,  since the dynamics of holes 
is a two-lane TASEP 
with jumps to the left and from lane $1$ to lane $0$.
Thus  there exist  random variables  
$-\infty\leq M'_1\leq M_1<+\infty$ and 
$-\infty<M_0\leq M'_0\leq +\infty$ such that
\be\label{convergence_blocking}
\lim_{t\to+\infty}\eta^1_t=\eta^{*}_{M'_1},\quad 
\lim_{t\to+\infty}\eta^0_t=\eta^{*}_{M'_0}
\ee
 Since particles can jump from lane $0$ to lane $1$ 
 but not the other way,  the dynamics imposes 
\be\label{imposes}
M'_0\geq M'_1
\ee
The limits in \eqref{convergence_blocking}  imply that 
$\eta_t$ converges in law 
to the distribution of the random configuration
$\eta_\infty$ defined by $\eta_\infty^i=\eta^*_{M'_i}$  
for $i\in\{0,1\}$, 
that is (cf. \eqref{def_etabot}) $\eta_\infty=\eta^{\bot,M'_0,M'_1}$. 
By stationarity, $\mu$ is the distribution of 
this configuration;  hence, recalling the definition 
of $\overline{\mathbb{B}}$  above \eqref{def_etabot},
\be\label{mixture_bot}
\mu=\int_{\overline{\mathbb{B}}}
\nu^{\bot,i,j}dm(i,j)
\ee
where $m(di,dj)$ denotes the law of $(M'_0,M'_1)$.
This  with \eqref{def_bl_3_2}--\eqref{def_bl2(iv)(b)} 
implies that $\mu$ is a mixture of the measures in   $Bl_1\cup Bl_2$. 
\end{proof}
\begin{proof}[Proof of Proposition \ref{prop:step2}]
\mbox{}\newline\newline
 \textit{Proof of (i).}  Without loss of generality, 
 we may assume $\mu\leq\tau\mu$.
Since $\mu\leq\mu_+$  (where $\mu_+$ is defined as 
in Corollary \ref{cor:step1}) 
and  $\mu\neq\mu_+$  (because we are not in case \eqref{equal}), 
 by \cite[Proposition 2.14 in Chapter VIII]{liggett2012interacting}
there exists a coupling measure $\widetilde{\mu}(d\eta,d\xi)$ 
with marginals $\mu(d\eta)$ and $\mu_+(d\xi)$, such that 
\be\label{order_tilde}
\widetilde{\mu}\left(
(\eta,\xi)\in\mathcal X\times\mathcal X:\, \eta<\xi
\right)=1
\ee
and which is  invariant for the coupled process.\par
\medskip\noindent 
 For $n,m\in\Z$ such that $m\leq n$, and $\xi\in\mathcal X$, we set 
\[
 M_{m,n}(\xi):=\frac{1}{n-m+1}\sum_{x\in V:\,m\leq x(0)\leq n}\xi(x)
\]
and simply write $M_n(\xi)$ when $m=1$.
Because $\mu_+$ is a mixture of Bernoulli measures, 
by the ergodic theorem, the limit
\be\label{ergo_xi}
M(\xi):=\lim_{n\to+\infty}M_n(\xi)=\lim_{n\to+\infty}M_{-n,n}(\xi)
\ee
exists $\widetilde{\mu}$-a.s. The distribution of 
 $M(\xi)$  is exactly $\gamma^+$.
Besides, 
 $M(\xi)$  is a conserved quantity for the dynamics 
 of the  stationary  coupled 
process $(\eta_t,\xi_t)_{t\geq 0}$ starting from 
$\widetilde{\mu}(d\eta_0,d\xi_0)$. 
Indeed, by the finite propagation property 
(Proposition \ref{prop:prop}),
\begin{eqnarray*}
\frac{2n+1-4\lfloor\sigma t\rfloor}{2n+1}M_{-n+2\lfloor 
\sigma t\rfloor,n-2\lfloor \sigma t\rfloor}(\xi_0) 
& \leq & M_{-n,n}(\xi_t)
\\
& \leq & 
 \frac{2n+1+4\lfloor\sigma t\rfloor}{2n+1}
M_{-n-2\lfloor \sigma t\rfloor,n+2\lfloor \sigma t\rfloor}(\xi_0)
\end{eqnarray*}
with probability greater than $1-e^{-Cn}$. 
Letting $n\to+\infty$ yields
\be\label{conserved}
M(\xi_t)=M(\xi_0)
\ee
It follows that for every $\rho$ in the support of   $\gamma^+$, 
the conditioned measure  $\widetilde{\mu}(\rho)$ defined by 
\be\label{conditioned_measure}
 \widetilde{\mu}(\rho)(d\eta,d\xi):=\widetilde{\mu}\left(
(d\eta,d\xi)| M(\xi)=\rho
\right)
\ee
is invariant for the coupled process. 
 Indeed, 
for every bounded  function $f$ on $\mathcal X\times\mathcal X$ 
and every bounded measurable function  $g$ on $[0,2]$,
\begin{eqnarray*}
\int_{[0,2]} <\widetilde{\mu}_\rho, \widetilde{S}_t f>
g(\rho)\gamma^+(d\rho)  
& = & \widetilde{\Exp}_{\widetilde{\mu}} 
\left[f(\eta_t,\xi_t)g\left( M(\xi_0)\right)\right] \\
& = & \widetilde{\Exp}_{\widetilde{\mu}} 
\left[f(\eta_t,\xi_t)g\left( M(\xi_t)\right)\right]\\
& = & \widetilde{\Exp}_{\widetilde{\mu}} 
\left[f(\eta_0,\xi_0)g\left( M(\xi_0)\right)\right] \\
& = & \int_{[0,2]} <\widetilde{\mu}(\rho), f>
g(\rho)d\gamma^+(\rho)
\end{eqnarray*}
In  the above display,  the first and last  equality 
follow from  definition 
\eqref{conditioned_measure}, 
the second one  from \eqref{conserved}, 
and the third one from stationarity.\par\medskip\noindent 
Hence, the $\eta$-marginal of  $\widetilde{\mu}(\rho)$, 
that is ${\mu}(\rho)$ defined by 
%
%
${\mu}(\rho)(d\eta):=\widetilde{\mu}\left(
d\eta| M(\xi)=\rho
\right)$
%
%
is invariant for $L$. Since
\be\label{decomp_measure}
\mu=\int_{[0,2]}\mu(\rho)  d\gamma^+(\rho), 
\ee
by extremality of $\mu$, we must have $\mu(\rho)=\mu$ 
for $\gamma^+$-a.e. $\rho\in[0,2]$. 
This means that
that under $\widetilde{\mu}(d\eta,d\xi)$, $\eta$ 
is independent of  $M(\xi)$.  \newline\newline
 {\em Notational remark.} In the above argument, we emphasize 
 the use of notations $\widetilde{\mu}(\rho)$ and $\mu(\rho)$, 
 but {\em not} $\mu_\rho$, to avoid any confusion 
 with the unrelated product measures $\mu_\rho$ defined 
 in \eqref{alpha_x} and Remark \ref{rk_prod}.   \newline\newline
Now we consider $A,B,A',B'\in\R$ such that $A$ 
lies in the support of $\gamma^+$ and $B<B'<A'<A$. 
Let $f,g$ be nondecreasing continuous functions on $\mathcal X$ 
supported respectively on $[A',+\infty)$ and $(-\infty,B']$,
taking constant value $1$ respectively 
on $[A,+\infty)$ and $(-\infty,B]$.
By \eqref{order_tilde},  \eqref{ergo_xi}, 
and independence  of $M_n(\eta)$ and $M(\xi)$, 
the following holds under $\widetilde{\mu}$:
\begin{eqnarray*}
0=
\widetilde{\Exp}_{\widetilde{\mu}}
\left[
f( M_n(\eta))g( M_n(\xi))
\right] & = & 
\widetilde{\Exp}_{\widetilde{\mu}}\left[
%
%
f( M_n(\eta)) g(M(\xi))
\right]+\varepsilon_n
\\
& = & 
 \widetilde{\Exp}_{\widetilde{\mu}}\left[
f( M_n(\eta))
\right]
\widetilde{\Exp}_{\widetilde{\mu}}\left[
 g(M(\xi))
\right] + \varepsilon_n
\end{eqnarray*}
for some sequence $\varepsilon_n\to 0$. It follows that
\[
\lim_{n\to+\infty}{\widetilde{\mu}}\left(
 M_n(\eta)>A\right){\widetilde{\mu}}\left(
 M(\xi)<B\right)=0
\]
Choosing $B$ strictly larger than the infimum 
of the support of $\gamma^+$ yields
\[
\lim_{n\to+\infty}{\widetilde{\mu}}\left(
 M_n(\eta)>A\right)=0
\]
It follows that
\be\label{exp_bareta}
\limsup_{n\to+\infty}\widetilde{\Exp}_{\widetilde{\mu}}
[ M_n(\eta)]\leq A
\ee
Set
\[
\overline{\mu}_n:=\frac{1}{n}\sum_{x=1}^n\tau_x\mu
\]
so that  \eqref{exp_bareta} also writes
\be\label{exp_bareta_2}
\limsup_{n\to+\infty}\int_{\mathcal X} \eta(0)
d\overline{\mu}_n(\eta)\leq A
\ee
On the other hand, by Proposition \ref{prop:step1}, 
$\overline{\mu}_n\to\mu_+$, thus
\[
\int_{\mathcal X}\eta(0)d\mu_+(\eta)=\int_{[0,2]}\rho 
d\gamma^+(\rho)\leq A
\]
for every $A$ in the support of $\gamma^+$. 
Hence $\gamma^+=\delta_{\rho^+}$ for some 
$\rho^+\in[0,2]$.\newline\newline
 \textit{Proof of  (ii).}  
Assume for instance $\rho^-<\rho^+$, the other case being similar.
The equality \eqref{whereweused}  yields 
(recall that $\widetilde{\lambda}=\widetilde{\lambda}_k$  
for $k\in\{0,1,2\}$, cf.
\eqref{eq:supportedon}) 
\[
\widetilde{\lambda}[D(\eta,\xi)]=k
=\lim_{n\to+\infty}\mu[\eta(n)]-\lim_{m\to-\infty}\mu[\eta(m)]
=\rho^+-\rho^-
\]
\end{proof}
\begin{proof}[Proof of Proposition \ref{lemma:translate}]
 The proof of \textit{(ii)} is similar to that of 
 Proposition \ref{prop:bowtie}. We prove \textit{(i)} 
 and \textit{(iii)} below. \newline\newline
 \textit{Proof of (i), step one.} 
We show that if  $\widetilde{\lambda}_0\in\widetilde{\mathcal I}_e$ 
is a coupling 
of $\nu$ and $\nu'$ (that exists by 
\cite[Proposition 2.14 in Chapter VIII]{liggett2012interacting}), 
 then
\be\label{couple_shocks}
\int_{\mathcal X\times\mathcal X}\left(
\sum_{x\in\Z,\,i\in W,\:\,|x|\leq T}|\eta^i(x)-\xi^i(x)|
\right)d\widetilde{\lambda}_0(\eta,\xi)=o(T),\quad\mbox{as }T\to+\infty
\ee
Let  (recall definition \eqref{def:Dmn})  
\[
\widetilde{\lambda}_T^\pm:=\frac{1}{|[-T,T]\cap\Z^\pm|}
\sum_{x\in\Z^\pm:\,|x|\leq T}\tau_x\widetilde{\lambda}_0,
\]
\[
A_l(\eta,\xi) 
:= \frac{1}{2(2l+1)}
\sum_{y\in\Z,\,i\in W:\,|y|\leq l}|\eta^i(y)-\xi^i(y)|
 = \frac{1}{2(2l+1)} D_{-l,l}(\eta,\xi) 
\]
\[
B_l(\eta,\xi) 
:= \left|
\frac{1}{2(2l+1)}\sum_{y\in\Z:\,\,i\in W:\,|y|\leq l}\eta^i(y)
-\frac{1}{2(2l+1)}\sum_{y\in\Z:\,i\in W:\,|y|\leq l}\xi^i(y)
\right|
\]
Every subsequential weak limit $\widetilde{\lambda}^\pm_\infty$ 
of the family $(\widetilde{\lambda}^\pm_T)_{T\geq 0}$  (which is tight 
as it lives on a compact space) lies in 
$\widetilde{\mathcal I}\cap\widetilde{S}$. 
Thus by Proposition \ref{prop_kilroy} and Lemma \ref{lemma:irred2},
it is supported on 
 $E_-\cup E_+$ (see \eqref{def:E-}--\eqref{def:E+}),  where $A_l=B_l$. 
The desired conclusion \eqref{couple_shocks} is equivalent to having, 
for any subsequential limit $\widetilde{\lambda}_\infty$,
\begin{eqnarray}\nonumber
0 & = & \lim_{l\to+\infty}\lim_{T\to+\infty}
\int_{ \mathcal X\times \mathcal X }\left(
\frac{1}{T}\sum_{x\in\Z^\pm,\,|x|\leq T}\tau_x A_l(\eta,\xi)
\right)d\widetilde{\lambda}_0(\eta,\xi)  \\
\label{eq:150}
& = & \lim_{l\to+\infty}
\int_{\mathcal X\times\mathcal X}B_l(\eta,\xi)
d\widetilde{\lambda}^\pm_\infty(\eta,\xi) 
\end{eqnarray}
By definition \eqref{limits_mu} of shock measures,
$\widetilde{\lambda}^\pm_\infty$ has marginals $\nu_{\rho^\pm}$. 
It follows that under $\widetilde{\lambda}^\pm_\infty$, 
the spatial averages 
in $B_l(\eta,\xi)$ both converge in probability and 
(being bounded by $2$) 
in $L^1$ to $\rho^\pm$, thus implying the limits 
 in \eqref{eq:150}.  
\newline\newline
 \textit{Proof of (i), step two.} 
We now adapt the proof of Proposition \ref{prop:step1}, 
defining $\mathcal N_T$, 
$\mathcal N_T^{in}$ and $\mathcal N_T^{out}$ as we did there, 
and replacing 
the initial distribution $\widetilde{\lambda}_0$ 
defined there by the one 
considered in the first step of the current proof. In the first step of 
the proof of Proposition \ref{prop:step1}, 
we similarly derive \eqref{discrepancies_out} 
from Proposition \ref{prop:prop}, whereas we can now 
obtain \eqref{discrepancies_in} 
as a consequence of \eqref{couple_shocks}. 
Steps two and three are unchanged 
and yield \eqref{eq:blocking_measures-4},
where the measure $\widetilde{\lambda}$ now coincides with 
$\widetilde{\lambda}_0$, 
because the latter is invariant.  Hence,  
we obtain \eqref{ordered_coupling} if $q>0$,
 or \eqref{implies} if $q=0$.   By extremality, 
 this implies that  $\widetilde{\lambda}_0$ 
 satisfies one of \eqref{etaless}--\eqref{case_bowtie}. \newline\newline
 {\em Proof of (iii).} The proof is similar to that of 
 Proposition \ref{prop:step1}, statement \textit{(ii)}. 
 The only differences lie in the following points,
assuming for instance that the conclusion of \textit{(i)} 
is \eqref{etaless}. First, the second line of \eqref{whereweused} 
is now
\be\label{whereweused_new}
\sum_{x\in V:\,m\leq x(0)\leq n}[\nu'(\xi(x))-\nu(\xi(x))]\in[0,+\infty]
\ee
which depends only on $\nu,\nu'$. Next, in \eqref{eq:supportedon}, 
$k$ can be a priori any  value in $\N\cup\{+\infty\}$ instead of 
only $0,1,2$. 
\end{proof}
\begin{proof}[Proof of Proposition \ref{prop:translate}]
 For two {\em ordered} probability measures $\gamma,\gamma'$ on 
 $\mathcal X$, let
\be\label{def_dist}
\Delta(\gamma,\gamma'):=\sum_{x\in V}|\gamma(\eta(x))
-\gamma'(\eta(x))|\in[0,+\infty]
\ee
Note  that $\Delta(\gamma,\gamma')$ satisfies 
the three  following properties: 
\be\label{dist_equal}
\Delta(\gamma,\gamma')=0\mbox{ if and only if }\gamma=\gamma'
\ee
 If $\widetilde{\gamma}$ is an ordered coupling 
of $\gamma$ and $\gamma'$, we have
\begin{eqnarray}
\Delta(\gamma,\gamma') & = & \int_{\mathcal X\times\mathcal X}
D(\eta,\xi)d\widetilde{\gamma}(\eta,\xi) 
\label{prop_dist_1}
\end{eqnarray}
 If a probability measure  $\gamma''$ 
on $\mathcal X$ is such that 
 $\gamma\leq\gamma'\leq\gamma''
\mbox{ or }\gamma''\leq\gamma'\leq\gamma$, then
\begin{eqnarray}
%
\Delta(\gamma,\gamma'') & = & \Delta(\gamma,\gamma')+\Delta(\gamma',\gamma'') 
\label{dist_triangle_1}
\end{eqnarray}
 {\em Proof of (i).} Without loss of generality, 
 we assume $\rho^-<\rho^+$.  For $n\in\Z$,  let us denote 
 $\nu_n:=\tau_n\nu$. 
We can apply Proposition \ref{prop:step1} to $\nu$ and 
rule out the case \eqref{case_bowtie} by assumption and 
Proposition \ref{prop:bowtie}. Thus by \textit{(i)} of 
Corollary \ref{cor:step1},
$\nu_n\leq\nu_{n+1}$ for all $n\in\Z$. We can also exclude 
the case  $k=2$ by \textit{(ii)} of Proposition \ref{prop:step2} 
because $|\rho^+-\rho^-|=1$; and the case \eqref{equal} 
because $\rho^-\neq\rho^+$.
Thus by  \eqref{prop_dist_1},  \textit{(ii)} of Proposition 
\ref{prop:step1} and \textit{(ii)} of Proposition \ref{prop:step2},
\be\label{nu_and_nextnu}
\Delta(\nu_{n-1},\nu_{n})=1,\quad\forall n\in\Z
\ee
By  \textit{(i)} of  Proposition \ref{lemma:translate},  for $n\in\Z$,
there exists a coupling $d\widetilde{\nu}_n(\eta,\xi)$
of $d\nu_n(\eta)$ and $d\nu'(\xi)$ that satisfies one of the 
properties \eqref{etaless}--\eqref{case_bowtie} 
of Proposition \ref{prop:step1}.
 By assumption and \textit{(ii)} of Proposition 
 \ref{lemma:translate}, we can rule out \eqref{case_bowtie}. 
 Thus $\nu_n$ and $\nu'$ are ordered. Besides, 
 \textit{(iii)} of Proposition \ref{lemma:translate}
and \eqref{prop_dist_1} imply 
\be\label{nu_and_nuprime} \Delta(\nu_n,\nu')\in\N,  
\quad\forall n\in\Z\ee
Let $S:=\{n\in\Z:\,\nu'\leq\nu_n\}$. We claim that $S$ is 
non-empty and bounded from below. Indeed if  $S$ were empty,  
 since $\nu$ a $(\rho^-,\rho^+)$-shock measure 
 (cf. definition \eqref{limits_mu}),  $\nu'\geq\nu_n$ 
 and $n\to+\infty$ would imply
$\nu'\geq\nu_{\rho^+}$; if $S$ were not bounded from below,  
$n\to-\infty$ along a subsequence where $\nu'\leq\nu_n$ 
would imply $\nu'\leq\nu_{\rho^-}$. Both conclusions would 
contradict  $\nu'$  being a $(\rho^-,\rho^+)$-shock measure. 
We set $n_0:=\min(S)$, thus
\be\label{sandwich_nu}\nu_{n_0-1}<\nu'\leq\nu_{n_0}\ee
By \eqref{sandwich_nu} and \eqref{dist_triangle_1},
\be\label{app_triangle} 
\Delta(\nu_{n_0-1},\nu_{n_0})=\Delta(\nu_{n_0-1},\nu')
+\Delta(\nu',\nu_{n_0})\ee
By \eqref{sandwich_nu}, \eqref{dist_equal} and 
\eqref{nu_and_nuprime}, the first term on the 
right-hand side of \eqref{app_triangle} is a nonzero integer; 
thus by \eqref{nu_and_nextnu}  for $n=n_0$,   
the second term is zero, and the conclusion follows 
from \eqref{dist_equal}.\newline\newline
{\em Proof of (ii).} We can consider $n_0$ and the couplings of 
$\nu_{n_0}$ with $\nu'$ and $\nu_{n_0-1}$ with  $\nu_{n_0}$ 
as in \textit{(i)}.  Let $\nu''$ be a $(\rho^-,\rho^+)$-shock 
measure. 
For  $n\in\Z$,  we can also apply \textit{(i)} of 
Proposition \ref{lemma:translate} to   $\nu''_{n}:=\tau_n\nu''$  
and $\nu$, and rule out case \eqref{case_bowtie},  
since by assumption we exclude \eqref{case_deg};  
thus these measures are ordered. 
The same holds for $\nu''_n$ and $\nu'$. 
Similarly to $n_0$, we can then define $n_1\in\Z$ such that 
\be\label{sandwich_n1}
\nu''_{n_1-1}<\nu_{n_0-1}\leq\nu''_{n_1}
\ee
Property \eqref{nu_and_nuprime} holds, 
but instead of \eqref{nu_and_nextnu}, 
 \textit{(iii)} of Proposition \ref{lemma:translate} 
now implies
\be\label{nu_and_nextnu_2}
\Delta(\nu_{n-1},\nu_{n})=2
=\Delta(\nu''_{n-1},\nu''_{n}),\quad\forall n\in\Z
\ee
Since $\nu'$ is not a translate of $\nu$, both terms 
on the right-hand side of \eqref{app_triangle} are now 
nonzero integers.
The first equality in \eqref{nu_and_nextnu_2} for $n=n_0$,  
combined with \eqref{app_triangle}, then yields 
\be\label{wenowhave_triangle}
\Delta(\nu_{n_0-1},\nu')=\Delta(\nu',\nu_{n_0})=1
\ee
We now distinguish the following cases.\newline\newline
(1) If $\nu''_{n_1}\geq\nu_{n_0}$,  by
\eqref{sandwich_n1} we have 
$\nu''_{n_1-1}\leq\nu_{n_0-1}\leq\nu_{n_0}\leq \nu''_{n_1}$; 
by \eqref{dist_triangle_1},
\[\Delta(\nu''_{n_1-1},\nu''_{n_1})
=\Delta(\nu_{ n''_1-1},\nu_{n_0-1})
+\Delta(\nu_{n_0-1},\nu_{n_0})+\Delta(\nu_{n_0},\nu''_{n_1})\]
 From \eqref{sandwich_n1}, \eqref{nu_and_nextnu_2} with $n=n_1$,  
 and \eqref{dist_equal}, we obtain  $\nu''_{n_1}=\nu_{n_0}$. \newline\newline
(2) If $\nu_{n_0-1}\leq\nu''_{n_1}\leq\nu_{n_0}$, 
we distinguish whether (a) $\nu_{n_0-1}\leq\nu''_{n_1}\leq\nu'$ 
or (b) $\nu'\leq \nu''_{n_1}\leq \nu_{n_0}$. 
In the former case,  \eqref{dist_triangle_1}  
and \eqref{wenowhave_triangle} yield 
\[
1=\Delta(\nu_{n_0-1},\nu')
=\Delta(\nu_{n_0-1},\nu''_{n_1})+\Delta(\nu''_{n_1},\nu')
\]
and one of the terms on the r.h.s. must be $0$. 
Case (b) is similar.
\end{proof}
\subsection{Proofs of intermediate results from Subsections 
\ref{subsec:proof_inv_2}--\ref{subsec:proof_inv_3}}\label{subsec:inter_2}
 \begin{proof}[Proof of Proposition \ref{prop:extrema}]
\mbox{}\newline\newline
 \textit{Proof of (o).}  This follows from 
 \eqref{eq:G-from-rho_0}, \eqref{flux_tasep},
 and Lemma \ref{lemma_phi}.  \newline\newline
 \textit{Proof of (i).} 
For the following,
we rely on the expression for $G$ given in (\ref{eq:G-from-rho_0}).
 Therefore $G\geq 0$.   
In cases \eqref{unless} and \eqref{unless_2}, $G$ is identically $0$ 
(the former follows from  example \ref{ex:sym}). 
We henceforth exclude these cases. 
If $q=0$, the conclusion follows from  example \ref{ex:periodic}. 
If $q>0$   (that is $r>0$), which we henceforth assume,  
$G$ is continuously differentiable.
First, $G'$ vanishes at  least  once because $G(0)=G(2)=0$,   cf. \textit{(o)}.    
 Next, 
\begin{equation}
G(\rho)=  (\gamma_0+\gamma_1)\frac{\rho}{2} 
\left(1-\frac{\rho}{2}\right) 
+  (\gamma_0-\gamma_1)(1-\rho) \varphi(\rho)
-  (\gamma_0+\gamma_1)\varphi(\rho)^2  \label{eq:Gviaphi}
\end{equation}
with
\begin{eqnarray}
\varphi(\rho)
&=& \frac{1}{2}\left(\frac{r+1}{r-1}\right)
\left(1-\sqrt{\psi(\rho)}\right)\mbox{ if }r\neq 1\nonumber\\
&  =  &  0  \mbox{ if } r=1\label{eq:phi}\\
\label{eq:psi}
\psi(\rho)&=& 1+ \left(\frac{r-1}{r+1}\right)^2 \rho(\rho-2)
\end{eqnarray}  
Note that   $0<\psi(\rho)\le 1$, where the first inequality 
follows from $r>0$.  We then compute
\begin{eqnarray}\label{eq:psi-deriv}
\psi'(\rho)&=& \left(\frac{r-1}{r+1}\right)^2 2(\rho-1)
\\\label{eq:phi-deriv}
\varphi'(\rho)&=&-\frac{1}{2}\left(\frac{r-1}{r+1}\right)
\frac{(\rho-1)}{\sqrt{\psi(\rho)}}\\\label{eq:phi-deriv2}
\varphi''(\rho)&=& -2r\left(\frac{(r-1)}{(r+1)^3}\right)
\psi(\rho)^{-3/2}\\\label{eq:phi-deriv3}
\varphi^{(3)}(\rho)
&=& 6r(\rho-1)\left(\frac{(r-1)^3}{(r+1)^5}\right)\psi(\rho)^{-5/2}
\end{eqnarray}
\begin{eqnarray}
\nonumber
G'(\rho)&=& (\gamma_0+\gamma_1)\frac{1}{2}(1-\rho) 
 + (\gamma_0-\gamma_1)\left[- \varphi(\rho)+(1-\rho)\varphi'(\rho)\right]
\nonumber\\ &&\qquad 
-  2(\gamma_0+\gamma_1)\varphi(\rho)\varphi'(\rho)\label{eq:G-deriv}\\
G''(\rho)&=& -\frac{1}{2}(\gamma_0+\gamma_1) 
 + (\gamma_0-\gamma_1) \left[- 2\varphi'(\rho)+(1-\rho)\varphi''(\rho)\right]
\nonumber\\ &&\qquad 
-  2(\gamma_0+\gamma_1)[\varphi'(\rho)^2+\varphi(\rho)\varphi''(\rho)]
\label{eq:G-deriv2}
\end{eqnarray}
\begin{eqnarray}
G^{(3)}(\rho)&=&  + (\gamma_0-\gamma_1) 
\left[- 3\varphi''(\rho)+(1-\rho)\varphi^{(3)}(\rho)\right]
\nonumber\\&& \qquad 
-  (\gamma_0+\gamma_1)[6\varphi'(\rho)\varphi''(\rho)
+2\varphi(\rho)\varphi^{(3)}(\rho)] \nonumber\\ 
&=& 6r \frac{(r-1)^2}{(r+1)^4}\psi(\rho)^{-5/2}\times
\nonumber\\&& \qquad
\left[ (\gamma_0-\gamma_1)\frac{4r}{(r-1)(r+1)}
+(\gamma_0+\gamma_1)(1-\rho)\right]\label{eq:G-deriv3}
\end{eqnarray}
 Hence,  $G^{(3)}$ is identically $0$ if $r=1$ (that is $p=q$), 
 and has a constant sign if $\gamma_0+\gamma_1=0$ and $r\neq 1$. 
 Whereas if $\gamma_0+\gamma_1\neq 0$  and $r\neq 1$, 
we have that $G^{(3)}(\rho)$ changes sign  exactly once,  for the value 
\begin{equation}\label{eq:tilderho0}
\widetilde\rho_0:
=1 + \frac{\gamma_0-\gamma_1}{\gamma_0+\gamma_1}\frac{4r}{(r-1)(r+1)}
\end{equation} 
Therefore $G''$ is increasing before $\widetilde\rho_0$ and decreasing afterwards.  
Hence $G''$ changes sign at most twice and $G'$ changes sign 
at most three times.\newline\newline
\textit{Proof of (ii).}  
 If $q=0$, then $G(1)=0$ by \eqref{flux_degenerate}. 
If $q\neq 0$, the functions $\widetilde{\rho}_i$ in Lemma \ref{lemma_phi} 
are continuously differentiable on $[0,2]$, thus the same holds for $G$.
By \eqref{eq:Gviaphi}, \eqref{eq:phi}--\eqref{eq:psi}, 
\eqref{eq:phi-deriv} and \eqref{eq:G-deriv},
\[
G(1)= \frac{(\gamma_0+\gamma_1)\sqrt{r}}{(\sqrt{r}+1)^2},\quad
G'(1)=\frac{ \gamma_1-\gamma_0 }{2}\left(\frac{\sqrt{r}-1}{\sqrt{r}+1}\right),
\]
whence the desired conclusions.\newline\newline
\textit{Proof of (iii).} 
Here we obtain 
\[
%
G'(2)=-\frac{\gamma_0+r\gamma_1}{r+1}
\]
Under \eqref{wlog}, we have 
$\gamma_0+r\gamma_1\geq r(\gamma_0+\gamma_1)\geq 0$. 
The lower bound is positive if $r>0$ and $\gamma_0+\gamma_1> 0$. 
On the other hand, $\gamma_0+r\gamma_1=(1-r)\gamma_0>0$ if 
$\gamma_0+\gamma_1=0$ 
and $\gamma_0\neq 0$; and
$\gamma_0+r\gamma_1>0$ if $r=0$ and $\gamma_0>0$.
Finally, $\gamma_0+r\gamma_1=0$ if we have \eqref{unless}, or \eqref{unless_2}, 
or $q=\gamma_0=0$. 
\newline\newline
\textit{Proof of (iv).} This follows from \eqref{eq:Gviaphi}, 
\eqref{eq:phi} and \eqref{eq:psi}.\newline\newline
 {\em Proof of (v).} Without loss of generality, we assume 
  $\gamma_0=\gamma_1=1/2$. 
Then \eqref{eq:Gviaphi} becomes
\begin{eqnarray}\label{eq:Gviaphi-1/2}
G(\rho)&=& \frac{\rho}{2} \left(1-\frac{\rho}{2}\right) - \varphi(\rho)^2 
\end{eqnarray}  
and \eqref{eq:G-deriv} becomes
\begin{eqnarray}\label{eq:G-deriv-1/2}
G'(\rho)&=& \frac{1}{2}(1-\rho) 
-2\varphi(\rho)\varphi'(\rho)=(1-\rho)
\left(1-\frac{1}{2\sqrt{\psi(\rho)}}\right)\label{eq:G-deriv2-1/2}
\end{eqnarray}
We have that 
\begin{eqnarray*}
G'(1/2)&=&\frac{1}{2}\left(1-\frac{1}{2\sqrt{\psi(1/2)}}\right)\\
G'(1/2)&>&0 \Leftrightarrow \psi(1/2)>\frac{1}{4} 
\Leftrightarrow 1>\left(\frac{r-1}{r+1}\right)^2
\end{eqnarray*}
which is true. Then  after some computations, one can see that 
\begin{eqnarray}\nonumber
G(1/2)>G(1)&\Leftrightarrow& 4\psi(\frac{1}{2})\psi(1)<\left[-1+\frac{7}{4}\left(\frac{r-1}{r+1}\right)^2\right]^2\\
&\Leftrightarrow&  3 - 14 R - R^2
< 0 \label{eq:lacondition}
\end{eqnarray}
 where 
\[ R=\frac{1}{4}\left(\frac{r-1}{r+1}\right)^2 \]
Indeed, we have that
\begin{eqnarray}\label{eq:valeurs}
\psi(1/2)= 1-3R;\,\psi(1)=1-4R
\end{eqnarray}
so that
$\quad\psi(1/2)+\psi(1)=2-7R,\, \psi(1/2)-\psi(1)=R,\quad$ and
\begin{eqnarray}\nonumber
G(1/2)>G(1)&\Leftrightarrow&R< 
\left[\sqrt{\psi(1/2)}-\sqrt{\psi(1)}\right]\times
\left[2-\left(\sqrt{\psi(1/2)}+\sqrt{\psi(1)}\right)\right]
\\\nonumber
&\Leftrightarrow&R< 
2\left[\sqrt{\psi(1/2)}-\sqrt{\psi(1)}\right]+
\left[\psi(1)-\psi(1/2)\right]
\\\nonumber
&\Leftrightarrow&R
\left[\sqrt{\psi(1/2)}+\sqrt{\psi(1)}\right]<R
\\\nonumber
&\Leftrightarrow&
4\psi(1/2)\psi(1)
<\left[1-\left(\psi(1)+\psi(1/2)\right)\right]^2
\end{eqnarray} 
Solving  inequation \eqref{eq:lacondition} 
 with respect to $r$ 
gives the condition in \textit{(b)}. \\ \\
 {\em Proof of (vi).} 
Let $F(\rho):=G(\rho+1)-G(\rho)$. Note that
\begin{eqnarray}\label{eq:dvptF}
F(\rho)&=&\mathcal{F}(\rho+1)-\mathcal{F}(\rho)\qquad\hbox{with}\\\label{eq:calF}
\mathcal{F}(\rho)&=&(\gamma_0+\gamma_1)
\left[ -\frac{1}{4}(\rho-1)^2 - \varphi(\rho)^2\right]
 - (\gamma_0-\gamma_1)(\rho-1)\varphi(\rho)
\end{eqnarray}
{\em First case.} We assume $\gamma_0+\gamma_1\neq 0$. 
By \eqref{hom_g}, without loss of generality, 
we may consider  $\gamma_0=\mathfrak{a}$ and $\gamma_1=1-\mathfrak{a}$ with 
$\mathfrak{a}\in\R$. 
We have
\begin{eqnarray*}
 \mathcal{F}'(\rho)&=& -\frac{1}{2}(\rho-1)
-2\varphi(\rho)\varphi'(\rho) 
+ ( 2\mathfrak{a}-1)(1-\rho) \varphi'(\rho)-( 2\mathfrak{a}-1)\varphi(\rho)\\
F'(\rho)&=& -1 
+ \frac{1}{2}\left[\frac{\rho}{\sqrt{\psi(\rho+1)}}
- \frac{(\rho-1)}{\sqrt{\psi(\rho)}} \right]\\
&&+ \frac{( 2\mathfrak{a}-1)}{2}\left(\frac{r-1}{r+1}\right)
\left[\frac{\rho^2}{\sqrt{\psi(\rho+1)}}
- \frac{(\rho-1)^2}{\sqrt{\psi(\rho)}} \right]\\
&&+ \frac{( 2\mathfrak{a}-1)}{2}\left(\frac{r+1}{r-1}\right)
\left[\sqrt{\psi(\rho+1)}- \sqrt{\psi(\rho)}\right]
\end{eqnarray*}
 then
\begin{eqnarray*}
\mathcal{F}''(\rho)&=& -\frac{1}{2}
-2\varphi'(\rho)^2-2\varphi(\rho)\varphi''(\rho) 
+ ( 2\mathfrak{a}-1)(1-\rho) \varphi''(\rho) 
-2( 2\mathfrak{a}-1)\varphi'(\rho)\\
&=& -\frac{1}{2} -2\left[-\frac{1}{2}\left(\frac{r-1}{r+1}\right)
\frac{\rho-1}{\sqrt{\psi(\rho)}}\right]^2\\&&
+2r\left(\frac{r+1}{r-1}\right)\left(1-\sqrt{\psi(\rho)}\right)
\left(\frac{r-1}{(r+1)^3}\right)
\psi(\rho)^{-3/2}\\
&&-2r( 2\mathfrak{a}-1)(1-\rho)\left(\frac{r-1}{(r+1)^3}\right)\psi(\rho)^{-3/2}\\
&&+( 2\mathfrak{a}-1)\left(\frac{r-1}{r+1}\right)
\frac{(\rho-1)}{\sqrt{\psi(\rho)}}
\end{eqnarray*}
so that 
\begin{eqnarray}\label{F''1}
F''(\rho)&=&  
- \frac{1}{2}\left(\frac{r-1}{r+1}\right)^2
\left[\frac{\rho^2}{\psi(\rho+1)}
- \frac{(\rho-1)^2}{\psi(\rho)} \right]\\\label{F''2}
&&+\frac{2r}{(r+1)^2}
\left[\frac{1}{\psi(\rho+1)^{3/2}} - \frac{1}{\psi(\rho)^{3/2}}
- \frac{1}{\psi(\rho+1)} + \frac{1}{\psi(\rho)} \right]\\\nonumber
&&+ ( 2\mathfrak{a}-1)\left(\frac{r-1}{r+1}\right)
\left[\left(\frac{\rho}{\sqrt{\psi(\rho+1)}}
 +\frac{2r}{(r+1)^2}\frac{\rho}{\psi(\rho+1)^{3/2}}\right)\right.\\\label{F''3}
&&\qquad\qquad\qquad\qquad\qquad
\left.-\left(\frac{(\rho-1)}{\sqrt{\psi(\rho)}}
+ \frac{2r}{(r+1)^2} \frac{(\rho-1)}{\psi(\rho)^{3/2}}\right)\right]
\end{eqnarray}
We check the sign of each term.
\begin{eqnarray}\nonumber
f(\rho)&=& \frac{(\rho-1)}{\sqrt{\psi(\rho)}} \\\nonumber
f'(\rho)&=& \frac{1}{\psi(\rho)^{3/2}}\frac{4r}{(r+1)^2}>0 \\\nonumber
g(\rho)&=& f(\rho)^2\\\label{F''g'}
g'(\rho)&=& 2f(\rho)f'(\rho) <0
\quad\mbox{ for }\rho\in[0,1)\\\nonumber
\bar{f}(\rho)&=& \frac{1}{\psi(\rho)^{3/2}} - \frac{1}{\psi(\rho)}\\\label{F''barf'}
\bar{f}'(\rho)&=& \frac{\psi'(\rho)}{2\psi(\rho)^{5/2}[2\sqrt{\psi(\rho)}+3]}
\left[-5+4\left(\frac{r-1}{r+1}\right)^2\rho(\rho-2)\right] 
\geq  0\\\nonumber
h(\rho)&=& \frac{(\rho-1)}{\psi(\rho)^{3/2}}
\left[\psi(\rho)+\frac{2r}{(r+1)^2}\right]\\\label{F''h'}
h'(\rho)&=& \frac{1}{\psi(\rho)^{5/2}}\frac{2r}{(r+1)^2}\frac{12r}{(r+1)^2} >0
\end{eqnarray}
 (note that $\bar{f}'(\rho)=0$ if $r=1$, and $\bar{f}'(\rho)>0$ if $r\neq 1$).
  By  \eqref{F''g'} and  \eqref{F''barf'}, the first two terms  
 \eqref{F''1} and \eqref{F''2} of $F''(\rho)$ are non-negative. By \eqref{F''h'}
 the sign of the third term \eqref{F''3} of $F''(\rho)$  depends on the sign of
 $( 2\mathfrak{a}-1)(r-1)$: If $( 2\mathfrak{a}-1)(r-1)\geq 0$, then the third term of 
 $F''(\rho)$ will also be non-negative. \\ \\
Let us assume first that   $\mathfrak{a}\geq 1/2,\,r\leq 1$. 
Hence $F''(\rho)> 0$  for $\rho\in[0,1)$.  
Then
\begin{eqnarray}\label{F'0}
F'(0)&=& -\frac{1}{2}-\frac{ 2\mathfrak{a}-1}{2}
\left[\left(\frac{r-1}{r+1}\right)+
\left(\frac{\sqrt{r}-1}{\sqrt{r}+1}\right)\right]<0\\\label{F'1}
F'(1)&=& -\frac{1}{2}+\frac{ 2\mathfrak{a}-1}{2}
\left[\left(\frac{r-1}{r+1}\right)+
\left(\frac{\sqrt{r}-1}{\sqrt{r}+1}\right)\right] <0 
\quad\hbox{(see below)}\\\label{F0}
F(0)&=& G(1)=\frac{\sqrt{r}}{(\sqrt{r}+1)^2} >0 \\\label{F1}
F(1)&=& -G(1) <0
\end{eqnarray}
 Thus $F$ being non-increasing the equation $F(\rho)=0$ has a unique solution.
 We now show that $F'(1)<0$. We write $X=\sqrt{r}$, 
and we consider $X\geq 1$. 
\begin{eqnarray*}\nonumber
\mathfrak{f}(X)&:=&2(r+1)(\sqrt{r}+1)F'(1)\\
\label{F'1-1}
&=&
(4\mathfrak{a}-3) X^3 -X^2-X-(4\mathfrak{a}-3)\\
\label{F'1-2}
\mathfrak{f}(1)&=& -4<0\\\label{F'1-3}
\mathfrak{f}'(X)&=& 3(4\mathfrak{a}-3) X^2 -2X-1
\end{eqnarray*}
 If  $\mathfrak{a}=3/4 $, $\mathfrak{f}'(X)< 0$. 
Otherwise we solve $\mathfrak{f}'(X)=0$. 
\begin{eqnarray*}\label{F'1-4}
\delta&=& 4(3\mathfrak{a}-2)>0 \quad\hbox{for}\quad \mathfrak{a}>2/3\\\label{F'1-5}
X_\pm&=&\frac{1\pm\sqrt{\delta}}{3(4\mathfrak{a}-3)}\quad\hbox{for}\quad
\delta\geq 0 
\end{eqnarray*}
Then
\begin{itemize}
\item if $\mathfrak{a}<2/3$, $\delta< 0$, $\mathfrak{f}'(X)< 0$, 
$\mathfrak{f}$ is decreasing hence $F'(1)<0$.
\item if $2/3\leq \mathfrak{a}<3/4$, $\mathfrak{f}'(X)> 0$ 
for $X\in(X_-,X_+)$; but $X_\pm<0$, 
hence $\mathfrak{f}'(X)< 0$, $\mathfrak{f}$ is decreasing and 
$F'(1)<0$. 
\item if $\mathfrak{a}=3/4$, $\mathfrak{f}'(X)< 0$, hence 
$F'(1)<0$. 
\item if $\mathfrak{a}>5/6$, $X_-<0<X_+$ and $X_+>1$ because 
\[X_+<1\Leftrightarrow 9(\mathfrak{a}-1)(4\mathfrak{a}-3)>0\Leftrightarrow \mathfrak{a}\notin (3/4,1)\]
thus  $\mathfrak{f}'(X)< 0$, $\mathfrak{f}$ is decreasing hence 
$F'(1)<0$.
\item if $3/4<\mathfrak{a}<5/6$ we also have $X_+>1$, thus $F'(1)<0$.
 \item if $\mathfrak{a}=5/6$, then $X_+=1+\frac{2}{\sqrt{2}}>1$ 
hence $\mathfrak{f}$ is decreasing and  $F'(1)<0$. 
\end{itemize}
 Thanks to \eqref{symmetry_g},  the case $\mathfrak{a}\geq 1/2,\,r\leq 1$ 
 is also solved.
A similar reasoning enables to deal with 
$\mathfrak{a}\leq 1/2,\,r\leq 1$, then with  $\mathfrak{a}\leq 1/2,\,r\geq 1$.  \\ \\
 {\em Second case.} We assume $\gamma_0+\gamma_1=0$ 
 and $p\neq q$.  Without loss of generality, 
we can consider  $\gamma_0=1$.  
This amounts to repeating the computations 
of the first case keeping only in $F'(\rho)$ 
and $F''(\rho)$ those terms with the factor $( 2\mathfrak{a}-1)/2$, 
which we replace by $1$. 
This leads similarly to $F''(\rho)<0$ for $\rho\in[0,1)$. 
However, we now have $F'(0)<0$ and $F'(1)>0$. 
Thus there exists $\rho^*\in(0,1)$ such that $F$ 
is decreasing on $[0,\rho^*]$ and increasing on 
$[\rho^*,1]$. Besides,  \eqref{F0}--\eqref{F1} 
are now replaced by  $F(0)=F(1)=0$, cf.  \textit{(o)} 
and \textit{(ii)} of Proposition \ref{prop:extrema}. 
 This implies that $0$ and $1$ are the only solutions 
of the equation  $G(\rho+1)-G(\rho)=0$. \\ \\
 {\em Proof of (vii).}
If $q=0$, this follows from Remark \ref{remark_case_b} 
and Example \ref{ex:periodic}. We now assume $q>0$.
We write the function $G(\rho)$ in terms 
of a different variable. Recall definition \eqref{def_f_always} of $\mathcal F$, 
definition of $\widetilde{\rho}_0(.)$
in Lemma \ref{lemma_phi}, and expression \eqref{eq:G-from-rho_0} 
for $G$. We can then write
\[
G(\rho)=\gamma_0\rho_0(1-\rho_0)+\gamma_1\rho_1(1-\rho_1)
\]
where $(\rho_0,\rho_1)$ is the unique element of $\mathcal F$ 
such that $\rho_0+\rho_1=\rho$. By \eqref{def_f_always}, 
setting $r=q/p>0$, there is a unique
$\lambda\in[0,+\infty)$ such that
\be\label{density_twolane_fug}
\rho_0=\frac{r\lambda}{1+r\lambda},\quad\rho_1=\frac{\lambda}{1+\lambda}
\ee
It follows that
\be\label{flux_twolane_fug}
G(\rho)=\widetilde{G}(\lambda)=\lambda\left\{
\frac{\gamma_0r}{(1+r\lambda)^2}+\frac{\gamma_1}{(1+\lambda)^2}
\right\}
\ee
Then, nonzero solutions of the equation $\widetilde{G}(\lambda)=0$ are solutions of 
\be\label{quadra}
r(\gamma_0+\gamma_1r)\lambda^2+2r(\gamma_0+\gamma_1)\lambda+\gamma_0r+\gamma_1=0
\ee
Positive solutions of \eqref{quadra} correspond to solutions 
of $G(\rho)=0$ in $(0,2)$.
If $\gamma_0+\gamma_1 r=0$, that is 
%
$p\gamma_0+q\gamma_1=0$,
%
then since $p>q$, we have $\gamma_0+\gamma_1\neq 0$.  
The unique solution of \eqref{quadra} is
\[
\lambda=-\frac{\gamma_1+\gamma_0r}{2r(\gamma_0+\gamma_1)}
\]
 and $\widetilde{G}$ changes sign around this solution. 
Recalling \eqref{wlog}, we find that $\lambda>0$ 
if $q\gamma_0+p\gamma_1<0$, that is \eqref{cond_nob2}, otherwise $\lambda\leq 0$.
If $p\gamma_0+q\gamma_1\neq 0$,  then \eqref{quadra} is quadratic 
with  reduced discriminant 
\[
\Delta'=-r(1-r)^2\gamma_0\gamma_1
\]
Under \eqref{cond_nob2}, by \eqref{wlog} and $q>0$, we have 
$\gamma_0+\gamma_1 r>0$. By Remark \ref{remark_case_b}, 
we have $\Delta'>0$, hence two solutions 
$\lambda_1<\lambda_2$  around which $\widetilde{G}(\lambda)$ 
changes sign. These solutions are such that 
\[
\lambda_1\lambda_2=\frac{\gamma_0r+\gamma_1}
{r(\gamma_0+\gamma_1r)}<0,\quad\lambda_1+\lambda_2
=-2\frac{\gamma_0+\gamma_1}{\gamma_0+\gamma_1 r}\leq 0
\]
 where the first inequality follows from \eqref{cond_nob2} 
 and the second one from \eqref{wlog}.  
By the former, there is a 
positive solution to \eqref{quadra}.
If \eqref{cond_nob2} fails and $\Delta'\geq 0$,  since $\lambda_1\lambda_2\geq 0$ 
and $\lambda_1+\lambda_2\leq 0$,  there is no positive solution. 
Finally, if \eqref{cond_nob2} fails and $\Delta'<0$, there is no solution.
\end{proof}
 \begin{proof}[Proof of Lemma \ref{lemma:r_finite}]
 In cases \textit{(i)--(ii)} below, 
we always have $|\mathcal R_0|\leq 3$. 
The only case not covered below is $q=0<p$ 
and $\gamma_0\neq\gamma_1$. Then 
\eqref{flux_degenerate} and Definition \ref{def_set} 
show that $\mathcal R_0$ is reduced to two elements 
of $\mathcal B_1$.  \newline\newline
\textit{Proof of (i).}
 By Definition \ref{def_set}, for any $(\rho^-,\rho^+)$ 
in $\mathcal R_0$, $\rho=\min(\rho^-,\rho^+)$ must be a solution 
of the equation $G(\rho+1)-G(\rho)=0$. 
By \textit{(vi)} of Proposition \ref{prop:extrema},  
this equation has exactly one solution $\rho$ in $[0,1]$. 
This implies $\mathcal R_0\subset\{(\rho,\rho+1);(\rho+1,\rho)\}$. 
But condition \textit{(ii)} of Definition \ref{def_set} 
implies that $(\rho,\rho+1)$ and $(\rho+1,\rho)$ cannot 
both lie in $\mathcal R_0$.  Indeed, $G$ would then be 
constant on $[\rho,\rho+1]$, and the only situations 
where $G$ can be constant on a nontrivial interval 
are \eqref{unless}, \eqref{unless_2} and \eqref{unless_3}, 
which are excluded here.  \newline\newline 
  Since $G(0)=G(2)=0$  by \textit{(o)} 
of Proposition \ref{prop:extrema},  
in order to have  $\mathcal B_1\cap\mathcal R_0\neq\emptyset$,  
it is necessary to have $G(1)=0$. 
By \textit{(ii)} of Proposition \ref{prop:extrema}, 
this only occurs if $q=0$ or $\gamma_0+\gamma_1=0$.\newline\newline
\textit{Proof of (ii).} 
{\em First case:} $q>0$, $p\neq q$ and 
$\gamma_0+\gamma_1=0\neq\gamma_0\gamma_1$.
Similarly to \textit{(i)}, using \textit{(vi)} 
of Proposition \ref{prop:extrema}, we see that 
$\mathcal R_0\subset\mathcal B_1$.  By \textit{(o)} 
and \textit{(ii)} of Proposition \ref{prop:extrema}, 
$G$ only vanishes for $\rho\in\{0,1,2\}$;  thus by 
Definition \ref{def_set}, one of the points $(0,1)$ 
or $(1,0)$, and one of the points $(1,2)$ or $(2,1)$, 
lie in $\mathcal R_0$. And since $(\rho,\rho+1)$ 
and $(\rho+1,\rho)$ cannot both lie in 
$\mathcal R_0$, $\mathcal R_0$ contains two elements. \newline
{\em Second case:}  $q=0$ and $\gamma_0\neq\gamma_1$. 
Then  \eqref{flux_degenerate} and Definition \ref{def_set} 
shows that $\mathcal R_0$ is reduced to two elements of 
$\mathcal B_1$.
 \newline\newline
{\em Proof of (iii).}
 Assume first $r>0$. 
By \eqref{symmetry_g}, since $\gamma_0=\gamma_1$, 
we have $G(2-\rho)=G(\rho)$ for all $\rho\in[0,2]$. 
Thus $G(1/2)=G(3/2)$ and $G'(1)=0$.   Recalling \textit{(i)}, 
there can be no shock of amplitude $1$ other than 
$(1/2,3/2)$ or $(3/2,1/2)$; and at most one of these 
 lies in $\mathcal R_0$. 
If $G$ has a single extremum (which must be at $1$), 
by \textit{(iii)} of Proposition \ref{prop:extrema}, 
$G$  is bell-shaped and this  extremum is a maximum. 
 Thus $\mathcal R_0=\{(1/2,3/2)\}$.
If $G$ has more than one extremum, by symmetry it must have three. 
Still by \textit{(iii)} of Proposition \ref{prop:extrema}, 
the extremum at $1$ is then a local minimum and the other two 
are local maxima symmetric with respect to $1$. Since 
$G'(1/2)>0$ by \textit{(v)} of Proposition \ref{prop:extrema}, 
condition \textit{(ii)} of Definition \ref{def_set} 
cannot hold with $(\rho^-,\rho^+)=(3/2,1/2)$. 
On the other hand, this condition holds with 
$(\rho^-,\rho^+)=(1/2,3/2)$ if and only if 
$G(1/2)\leq G(1)$. The conclusion then follows from 
\textit{(v)} of Proposition \ref{prop:extrema}. 
 Finally, for $r=0$, $\mathcal R_0$ follows 
from \eqref{flux_degenerate} and Definition \ref{def_set} 
(recall \eqref{wlog}, implying here that $\gamma_0>0$ 
and $\gamma_1>0$). \newline\newline
{\em Proof of (iv).} 
For $(\mathfrak{a},r)\in[1/2,1]\times[1,+\infty)$, let us denote by 
$\rho(\mathfrak{a},r)$ the unique solution given by \textit{(vi)} 
of Proposition \ref{prop:extrema} of
 $F_{\mathfrak{a},1-\mathfrak{a},r}(\rho)=0$,  where 
 $F_{\mathfrak{a},1-\mathfrak{a},r}(\rho)
 :=G_{\mathfrak{a},1-\mathfrak{a},r}(\rho+1)-G_{\mathfrak{a},1-\mathfrak{a},r}(\rho)$. 
The proof of Proposition \ref{prop:extrema}, \textit{(vi)} 
showed that  $F'_{\mathfrak{a},1-\mathfrak{a},r}(\rho)<0$  for every $\rho\in[0,1]$. 
Besides, by \eqref{eq:Gviaphi}, \eqref{eq:phi} 
and \eqref{eq:psi},  $F_{\mathfrak{a},1-\mathfrak{a},r}$ 
is continuously differentiable with respect to $(\mathfrak{a},r)$. 
Thus the implicit function theorem implies that 
$(\mathfrak{a},r)\mapsto \rho(\mathfrak{a},r)$ is continuously differentiable.  Let
\[
I(\mathfrak{a},r)
:=\inf_{\rho\in [\rho(\mathfrak{a},r),1+\rho(\mathfrak{a},r)]} 
G_{\mathfrak{a},1-\mathfrak{a},r}(\rho),\quad
S(\mathfrak{a},r):=\sup_{\rho\in [\rho(\mathfrak{a},r),1+\rho(d,r)]} 
G_{\mathfrak{a},1-\mathfrak{a},r}(\rho) 
\]
We define
\be\label{def_openset}
\mathcal Z:=\left\{
(\mathfrak{a},r)\in[0,1]\times[0,1]:\,
I(\mathfrak{a},r)
< G_{\mathfrak{a},1-\mathfrak{a},r}[\rho(\mathfrak{a},r)]<S(\mathfrak{a},r)
\right\}
\ee
The set $\mathcal Z$ is an open subset of $[0,1]^2$ 
because $(\mathfrak{a},r)\mapsto \rho(\mathfrak{a},r)$ is continuous. 
By \textit{(iii)}, it contains  $\{1/2\}\times(0,r_0)$. 
Finally, by \textit{(ii)} of Definition \ref{def_set}, 
for $(\mathfrak{a},r)\in\mathcal Z$, neither $(\rho(\mathfrak{a},r),1+\rho(\mathfrak{a},r))$ 
nor $(1+\rho(\mathfrak{a},r),\rho(\mathfrak{a},r))$ lies in $\mathcal R_0$, 
thus $\mathcal R_0=\emptyset$. 
\end{proof}
\begin{proof}[Proof of Lemma \ref{lemma:nob2}]
\mbox{}\newline\newline
{\em Proof of (i).} By  \textit{(i)}  
of Proposition \ref{prop:step2}, 
$\mu$ is a shock measure of amplitude $2$, 
 that is either a $(0,2)$ or a $(2,0)$-shock measure.
The second possibility and \textit{(i)} of 
Proposition \ref{cor:step2} would imply that $(2,0)$ 
satisfies condition \textit{(ii)} of Definition \ref{def_set}, 
thus  by \textit{(o)} of Proposition \ref{prop:extrema},  
that the maximum of $G$ is  $0$;
whereas  \textit{(iii)} of Proposition \ref{prop:extrema}  (when $q>0$)   
and \eqref{flux_degenerate}   (when $q=0$)   imply that this maximum 
is positive under \eqref{wlog}.  \newline\newline
{\em Proof of (ii).} 
 By \textit{(vii)} of Proposition \ref{prop:extrema},  the equation 
$G(\rho)=0$ has a solution in $(0;2)$  and changes sign 
around this solution.  Since $G(0)=G(2)=0$, 
condition \textit{(ii)} of Definition \ref{def_set} cannot hold, 
and Proposition \ref{cor:step2} implies that a 
$(0,2)$-shock measure cannot exist. This and \textit{(i)} 
above imply the desired conclusion. 
\end{proof}
\begin{proof}[Proof of Proposition \ref{cor:step2}]
\mbox{}\newline\newline
 \textit{Proof of (i).}  Assume for instance $\rho^-<\rho^+$, 
 the case $\rho^->\rho^+$ 
 being similar. Let $r\in[\rho^-,\rho^+]$. 
 Let  $d\widetilde{\nu}(\eta,\xi)$  be a coupling 
 of $d\nu(\eta)$ and $d\nu_r(\xi)$ that is invariant 
 for the coupled generator \eqref{coupled_gen} 
 (it exists by \cite[Proposition 2.14 in Chapter VIII]{liggett2012interacting}).  
 Since $\widetilde{\nu}$ is supported on a compact space, 
 there exists an increasing  
 $\N$-valued sequence $x_n\to +\infty$ such that 
 $\tau_{-x_n-1}\widetilde{\nu}$ 
 and $\tau_{x_n}\widetilde{\nu}$ have  weak limits denoted 
 respectively by $\widetilde{\nu}_{-\infty}$ 
 and $\widetilde{\nu}_{+\infty}$. By \eqref{limits_mu} 
 and translation invariance 
 of $\nu_r$, $\widetilde{\nu}_{\pm\infty}$ is 
 a coupling of $\nu_{\rho^\pm}$ and $\nu_r$. 
 Since the coupled generator  $\widetilde{L}$ 
 given by \eqref{coupled_gen} 
  for the transition kernel  \eqref{eq:intensity_c} 
 is translation invariant in the $\Z$-direction, we have 
 $\widetilde{\nu}\in\widetilde{\mathcal I}\cap\widetilde{\mathcal S}$. 
 Hence, by  \eqref{ordered_coupling} in  
 the proof of \thmref{characterization_lemma}, $\widetilde{\nu}$ 
 is supported on ordered pairs $(\eta,\xi)$. On the other hand, 
 under $\widetilde{\nu}_{\pm\infty}$, empirical averages 
 (cf. \eqref{ergo_xi}) 
 exist by the law of large numbers and are given by 
 $M(\eta)=\rho^\pm$ and $M(\xi)=r$. 
 These averages must be ordered like $\eta$ and $\xi$, 
 hence $\widetilde{\nu}_{-\infty}$ 
 and $\widetilde{\nu}_{+\infty}$ are supported respectively 
 on  $E_-$ and $E_+$.  \newline\newline
 Let  $N\in\N$, $R_N:=(\Z\cap[-N,N])\times W$, and 
\be\label{sum_rect}
 \widetilde{F}_N(\eta,\xi):= D_{-N,N}(\eta,\xi)
=\sum_{i\in W}\sum_{ z\in\Z\cap[-N,N]}|\eta(z,i)-\xi(z,i)|
\ee 
 Since $\widetilde{\nu}\in\widetilde{\mathcal I}$, we have
\be\label{invtilde}
\int_{\mathcal X\times\mathcal X}\widetilde{L} 
\widetilde{F}_N(\eta,\xi)  d\widetilde{\nu}(\eta,\xi)=0
\ee 
By \cite[Lemma 2.4]{Liggett1976}, we have 
\begin{eqnarray}
\label{computation_1}
\widetilde{L}\widetilde{F}_N(\eta,\xi) 
& = & \sum_{x\not\in R_N,\,y\in R_N} p(x,y)J_{x,y}(\eta,\xi)\\
\label{computation_2}
& - & \sum_{x\in R_N,\,y\not\in R_N} p(x,y)J_{x,y}(\eta,\xi) \\
\label{computation_3}
& - & \sum_{x\in R_N,\,y\in R_N,\,x\neq y} 
[p(x,y)+p(y,x)] {\bf 1}_{E_{x,y}}(\eta,\xi)
\end{eqnarray}
 where  $E_{x,y}$ was defined in \eqref{def_dk}, and 
\begin{eqnarray}
J_{x,y}(\eta,\xi) & := & [\eta(x)(1-\eta(y))-\xi(x)(1-\xi(y))]\left\{
{\bf 1}_{\{\eta(x)\geq \xi(x),\,\eta(y)\geq \xi(y)\}}-
\right. 
\nonumber\\
& & \left.
{\bf 1}_{\{\eta(x)\leq \xi(x),\,\eta(y)\leq \xi(y)\}}
\right\}\label{coupled_current}
\end{eqnarray}
Let
\[
\widetilde{j}(\eta,\xi):=\sum_{x(0)\leq 0,y(0)>0}p(x,y)J_{x,y}(\eta,\xi)
-\sum_{x(0)\leq 0,y(0)>0}p(y,x)J_{y,x}(\eta,\xi)
\]
where $J_{x,y}$ is defined by \eqref{coupled_current}.
Then \eqref{computation_1}--\eqref{computation_2} 
can be written as  
$\tau_{-N-1}\widetilde{j}(\eta,\xi)-\tau_{N}\widetilde{j}(\eta,\xi)$. 
By \eqref {coupled_current} and \eqref{def_microflux_ladder}
\be\label{difference_currents-2}
\widetilde{j}(\eta,\xi)=j(\eta)-j(\xi)\mbox{ if }\eta\leq\xi,\quad
\widetilde{j}(\eta,\xi)=j(\xi)-j(\eta)\mbox{ if }\xi\leq\eta
\ee
The stationarity relation \eqref{invtilde} combined with 
\eqref{computation_1}--\eqref{computation_3} yields 
\be\label{boundary_flux}
\widetilde{\nu}(\tau_{-N-1}\widetilde{j})
-\widetilde{\nu}(\tau_{N}\widetilde{j})\geq 0
\ee 
Taking  $N=x_n$ and letting  $n\to+\infty$ yields
\be\label{boundary_flux_lim}
\widetilde{\nu}_{-\infty}(\widetilde{j})
-\widetilde{\nu}_{+\infty}(\widetilde{j})\geq 0
\ee
Under $\widetilde{\nu}_{\pm\infty}$, 
we can use \eqref{difference_currents-2} 
for ordered configurations. The marginals of  
$\widetilde{\nu}_{\pm\infty}$  then yield 
\be\label{entropy}
G(r)-G(\rho^-)\geq G(\rho^+)-G(r)
\ee
Since $r\in[\rho^-,\rho^+]$ is arbitrary,  
we first obtain  $G(\rho^+)=G(\rho^-)$ 
by letting $r=\rho^\pm$, and then
$G(\rho^+)=G(\rho^-)=\min_{r\in[\rho^-,\rho^+]}G(r)$.  \newline\newline
\textit{Proof of (ii). } This follows from \textit{(i)} above, 
and \textit{(ii)} of Proposition \ref{prop:step2}. \newline\newline
\textit{Proof of  (iii).}   By Lemma \ref{lemma:r_finite}, 
$\gamma_0+\gamma_1=0$ or $q=0$. 
Assume from now on that the latter holds. \newline\newline
(a) We assume first $\gamma_1\geq 0$. Then by \eqref{flux_degenerate} 
and Definition \ref{def_set}, if $\gamma_0$ and $\gamma_1$ 
are not both $0$, we have 
$\mathcal R_0\cap\mathcal B_1=\{(0,1);(1,2)\}$.  \newline\newline
We consider first $(\rho^-,\rho^+)=(0,1)$.  
We show that this case is impossible if $\gamma_1=0$, 
whereas if $\gamma_1>0$,  $\mu$ is one of the measures 
$\nu^{\bot,+\infty,j}$ in \eqref{def_bl_3_2}.   
  To this end, observe first that since $q=0<p$,  
  $\nu^0$ is the probability measure supported 
  on the empty configuration and $\nu^1$ is supported 
  on the configuration that is empty 
  on lane $0$ and full on lane $1$.  Since $\mu$ is a 
  $(0,1)$-shock measure, we have
\begin{eqnarray}
\label{shock_lane_0}
\lim_{x\to-\infty}\tau_x\eta^0_0=\mu_0,\quad 
\lim_{x\to+\infty}\tau_x\eta^0_0=\mu_0, & & \\ 
\label{shock_lane_1}
\lim_{x\to-\infty}\tau_x\eta^1_0=\mu_0,\quad 
\lim_{x\to+\infty}\tau_x\eta^1_0=\mu_1 
\end{eqnarray}
where  $\mu_\rho=\mu_{\Z,\rho}$  (recall \eqref{alpha_x} and Remark \ref{rk_prod}) 
denotes the product Bernoulli measure 
on $\{0,1\}^\Z$ with parameter $\rho$; for $\rho\in\{0,1\}$ as above, 
these are Dirac measures supported on the empty or full configuration. 
 As in the proof of  Proposition \ref{prop:bowtie}, 
we couple $\eta_.$ with   an ASEP  $\zeta_.^0$ on lane $0$ 
starting from $\zeta^0_0:=\eta^0_0$,
with jump rate $d_0$ to the right and $l_0$ to the left, that is 
\eqref{def_tasep}--\eqref{eq:generator of the Exclusion} 
with $(l,d)=(l_0,d_0)$. 
The limit \eqref{shock_lane_0} implies
\be\label{shock_zeta}
\lim_{n\to+\infty}\frac{1}{n}\sum_{x=1}^n\zeta^0_0(x)
=\lim_{n\to+\infty}\frac{1}{n}\sum_{x=-n}^1\zeta^0_0(x)=0
\ee
in probability. 
 Since the initial configuration satisfies \eqref{shock_zeta}, 
 $\zeta_t^0$ converges 
in law as $t\to+\infty$ to the Bernoulli invariant measure 
with zero density, 
that is the empty configuration;
this follows from \cite[Theorem 1]{bm} when $\gamma_0>0$, 
or \cite[Chapter VIII]{liggett2012interacting} when $\gamma_0=0$.
Since $\eta_t^0\leq\zeta_t^0$, 
the same limit  holds for $\eta_t^0$. By stationarity of $\mu$, 
this implies that 
under $\mu$, lane $0$ is almost surely empty. 
It follows that $(\eta^1_t)_{t\geq 0}$ 
is itself an autonomous  SEP.  Thus the marginal of $\mu$ on lane $1$ is 
an invariant measure for  SEP.  By \cite[Theorem 1.4]{Liggett1976},
it is a mixture of Bernoulli and blocking measures. 
Because of \eqref{shock_lane_1}, 
only blocking measures are present in the mixture.  
Note that this is only possible if $\gamma_1>0$. 
In this case, $\mu$ is a mixture of  
the invariant measures $\nu^{\bot,+\infty,j}$ in 
\eqref{def_bl_3_2} for $j\in\Z$. Since $\mu$ is extremal, 
it is one of them.\newline\newline
 Next, we consider $(\rho^-,\rho^+)=(1,2)$. This can be 
 reduced to the previous case 
 by Lemma \ref{lemma_sym}, considering the image of 
 $\eta_t$ by $\sigma\sigma'\sigma''$. 
 The resulting process has drift $\gamma'_0=\gamma_1$ on lane $0$, 
 and $\gamma'_1=\gamma_0$ 
 on lane $1$. The  image $\mu''$ of $\mu$  is a $(0,1)$-shock 
 measure invariant 
 for the transformed process. It follows from the above  that: \newline\newline
- If $\gamma_0>0$, 
$\mu''=\nu^{\bot,+\infty,j}$, thus $\mu=\nu^{\bot,j,-\infty}$, 
for some $j\in\Z$. \newline\newline
- If  $\gamma_0=0$,  that is $\gamma'_1=0$,  from the above discussion, 
it is impossible for $\mu''$ to be a $(0,1)$-shock measure, 
and thus for $\mu$ 
to be a $(1,2)$-shock measure.  \newline\newline
Putting together the cases $(\rho^-,\rho^+)=(0,1)$ 
and $(\rho^-,\rho^+)=(1,2)$, 
we  conclude that 
in case \textit{(iv)} of Theorem  \ref{cor:entire}, 
a $(\rho^-,\rho^+)$-shock measure with 
$(\rho^-,\rho^+)\in  \mathcal B_1$ 
lies in the set \eqref{def_bl_3_2}; 
whereas in case \textit{(vi)}  it lies in the set 
 \eqref{def_bl_3_1}.  In the former case 
 $\mathcal R'=\mathcal R_0\cap\mathcal B_1=\{(0,1);(1,2)\}$, 
 whereas in the latter case  
 $\mathcal R'=\{(0,1)\}\neq\mathcal R_0\cap 
 \mathcal B_1=\{(0,1);(1,2)\}$.\newline\newline
\textit{(b)} We consider now $\gamma_1<0<\gamma_0$. 
Here, by \eqref{flux_degenerate} 
and  Definition \ref{def_set},  we have 
$\mathcal R_0\cap\mathcal B_1=\{(1,0);(2,1)\}$. The case 
$(\rho^-,\rho^+)=(1,0)$ is treated like 
$(\rho^-,\rho^+)=(0,1)$ in (a) above;
except that on lane $1$ we have a $(1,0)$-shock with a negative drift. 
The case $(\rho^-,\rho^+)=(1,2)$ is deduced 
by  Lemma \ref{lemma_sym} and particle-hole symmetry
 (recall \eqref{sym_op}). 
\end{proof}
\begin{proof}[Proof of Lemma \ref{lemma_blocking}]
Let $\widetilde{\nu}$ denote a coupling of $\nu^1$ and $\nu^2$ 
such that $\widetilde{\nu}\in\widetilde{\mathcal I}$. 
Since $\nu^1$ and $\nu^2$ 
are supported on $\mathcal X_2$, $\widetilde{\nu}$ satisfies assumption 
\eqref{assumption_finite_disc} of Proposition \ref{prop_kilroy}. 
Since we excluded the case $l_0=l_1=q=0<p$,  
by Lemma \ref{lemma_wi}, 
$p(.,.)$ is weakly irreducible. Thus, by \eqref{p_ordered_conf} 
and the proof of \thmref{characterization_lemma}, 
$\widetilde{\nu}$ is supported on ordered pairs 
of configurations. Since $H_2$ 
is a nondecreasing function on $\mathcal X_2$ 
and has the same value under 
both marginals of $\widetilde{\nu}$, it follows that  
$\widetilde{\nu}$ is supported on  $E_3$  (defined in \eqref{equal}),  
whence the conclusion. 
\end{proof}
\begin{appendix}\label{append}
\section{Extensions}\label{app:ext}
\subsection{Results}\label{app:ext_res}
 We discuss below situations where our approach should still work 
to extend parts of our results with minor modifications, or with 
suitable extensions but without essentially new ideas. Some explanations
on the feasability of these extensions are given in 
 Appendix \ref{app:ext_proof}. 
\subsubsection{Non nearest-neighbour horizontal kernels}\label{subsec:nnnhk}
We may consider kernels of the form \eqref{restrict_kernel} 
in which the horizontal kernels $Q_i(.)$ are no longer 
assumed nearest-neighbour, but only weakly irreducible, 
cf. Definition \ref{def_irred}. 
The results of Theorems \ref{thm:characterization_lemma} 
and \ref{thm:Characterization_of_invariant_measures}  remain valid 
as such, because their proofs do not   require  the nearest-neighbour 
assumption. This is partly true for Theorems \ref{cor:entire} 
and \ref{thm:Characterization_of_invariant_measures_gen} 
with the following restrictions or modifications (only statements 
that do not carry over as such are mentioned). \newline\newline
{\em In Theorem \ref{cor:entire}.} \newline\newline
{\em Statement (o).} The proof  carries over if we assume that 
$Q_i(z)>Q_i(-z)$ for all $i\in W$ and $z>0$ such that $Q_i(z)>0$; 
this is an intermediate condition 
between the single-lane conditions in \cite[Theorem 5.1]{fls} 
and \cite[Theorem 1.4]{Bramson2002}.\newline\newline
{\em Statement (i).} The proof carries over under the assumption 
that the kernel $Q_i(.)$ on each lane is symmetric. \newline\newline
{\em Statement (ii).} This may be extended under the following 
assumption, automatically satisfied for nearest-neighbour kernels 
(see \cite{fls} for a similar condition for single-lane ASEP, 
or \cite{ligd} for $d$-dimensional ASEP): there exists a constant 
$\theta\in(0,+\infty]$ such that
\be\label{cond_fls}
\forall i\in W,\,z\in\Z,\quad \frac{Q_i(z)}{Q_i(-z)}=\theta^z
\ee
Under this condition blocking measures can still be constructed 
from \eqref{sol_rev_twolane}. \newline\newline
{\em Statement (iii).} Under condition \eqref{cond_fls}, 
the conditions $d_1=\lambda d_0$ and $l_1=\lambda l_0$  
for some $\lambda\in\R$  should be 
replaced by $Q_1(z)=\lambda Q_0(z)$ for all $z\in\Z$. \newline\newline
{\em Statements (iv)--(vii).} The proof  of the description 
of $Bl_1$, cf; \eqref{def_bl_3_2}, \eqref{def_bl_5}, \eqref{def_bl_3_1}, 
remains valid as long as the kernels $Q_i(.)$ satisfy assumptions 
of \cite{BMM} or \cite{Bramson2002} ensuring
existence of profile measures for the corresponding single-lane ASEP. 
Then  the  family of  measures $\{\widehat{\mu}_n, n\in\Z\}$ 
involved in the construction of $\nu^{\bot,i,j}$ is more generally 
the family of profile measures from \cite{Bramson2002}, instead of 
being defined by \eqref{blocking_config}, \eqref{cond_rev_meas}. 
The statement that $Bl_2=\emptyset$ in \textit{(iv)} and \textit{(vii)} 
can be generalized under the assumption (similar to \textit{(o)} 
above) that $Q_1(z)>Q_1(-z)$ for all $z>0$. \newline\newline
{\em In Theorem \ref{thm:Characterization_of_invariant_measures_gen}.}
As above, statement (2) and its proof carry over under the assumption 
that the kernel on each lane is symmetric, and the description of 
blocking measures in $\mathcal I_n$ based on \eqref{sol_rev_twolane} 
can be generalized under condition \eqref{cond_fls}. 
\subsubsection{Multilane models}\label{subsec:mm}
 While the ladder process (that is a 
vertically cyclic multilane ASEP) was discussed in Subsection 
\ref{subsec:rotation}, another natural multilane generalization of 
the two-lane model is the kernel \eqref{restrict_kernel}, 
where $W=\{0,\ldots,n-1\}$ and
\be\label{natural_multi}
q(i,j)=p{\bf 1}_{\{i<n-1,\,j=i+1\}}+q{\bf 1}_{\{i>0,\,j=i-1\}}
\ee
with a nearest-neighbour kernel $Q_i(.)$ and corresponding drift 
$\gamma_i$ on lane $i$.
Note that the cyclic model of Subsection \ref{subsec:rotation} 
contains the vertical kernels  (for which tagged particle motion 
is studied in \cite{zhang-lad}) 
\be\label{cyclic_multi}
q(i,j)=p{\bf 1}_{\{i<n-1,\,j=i+1\}}+p{\bf 1}_{\{i=n-1,\,j=0\}}
+q{\bf 1}_{\{i>0,\,j=i-1\}}+q{\bf 1}_{\{i=0,\,j=n-1\}}
\ee
and \eqref{natural_multi}--\eqref{cyclic_multi} coincide if and only if 
$n=2$ and $p=q$. \newline\newline
{\em Theorem \ref{thm:characterization_lemma}.} The result can be proved 
similarly for the multilane model \eqref{natural_multi} if we extend the 
definition of $\mathcal F$ as the set of $(\rho_0,\ldots,\rho_{n-1})\in[0,1]^n$ 
such that
\be\label{def_f_ext}
p\rho_i(1-\rho_{i+1})=q\rho_{i+1}(1-\rho_i),\quad\forall i\in\{0,\ldots,n-1\}
\ee
{\em Theorems 
\ref{thm:Characterization_of_invariant_measures}--\ref{cor:entire}.} Outside the case
\be\label{unless_ext}p=q\mbox{ and }\sum_{i=0}^{n-1}\gamma_i=0\ee
(that is, the extension of \eqref{unless}), the proofs of this paper 
could be extended to show the following statements, 
some of which are analogous to
Theorem \ref{thm:Characterization_of_invariant_measures_gen}. \newline\newline
1. Elements of $\mathcal I_e$ that are not homogeneous Bernoulli measures 
consist of finitely many (up to shifts) shock measures of amplitude 
$k\in\{1,\ldots,n\}$, with at most $k$ shock measures of amplitude $k$.\par\medskip
\noindent
2. Statement \textit{(1), (b)} of 
Theorem \ref{thm:Characterization_of_invariant_measures_gen} holds. \par\medskip
\noindent
3. If $q>0$, under the same condition as in statement 
\textit{(1), (c)}, the set of shock measures  with shock 
 ($\rho^-=0,\rho^+=n)$ consists of (up to shift) $n$ blocking measures. 
These are constructed similarly to Lemma \ref{def_blocking_h2} when 
$l_i>0$ for all $i$; now the conditioning on $H_2$ depends on the 
remainder of $H_2$ modulo $n$. When $l_i=0$ for all $i$, these 
measures are constructed as in Subsection \ref{sec:blocking}, 
case \textit{(ii), (b)} and \textit{1, (c)} 
of Theorem \ref{thm:Characterization_of_invariant_measures_gen}. 
As in \eqref{blocking_config_multi}--\eqref{blocking_meas_multi}, 
we have a family of blocking measures $\nu_i$ for $i=0,\ldots,n-1$. 
However the weights are no longer uniform as in 
\eqref{blocking_meas_multi}:  we now have
\[
\nu_i:=\sum_{A\subset\{0,\ldots,n-1\}:\,|A|=i}w_i(A)\delta_{\eta_A}
\]
We may identify a subset $A$ of $W$ with an exclusion configuration on $W$ 
for which $A$ is the set occupied sites. The weight $w_i(A)$ is then the 
probability of $A$  under the unique invariant measure with $i$ particles 
of the SEP on $W$ with jump kernel \eqref{natural_multi}.\par\medskip
\noindent
4. If $q=0$, results of Theorem \ref{cor:entire} could be extended 
as follows. First, if all lane drifts $\gamma_i$ are different, the set 
of non-homogeneous invariant measures contains no shock measure except 
blocking or partial blocking measures described below.\newline\newline
To construct blocking or partial blocking measures,
we must partition lanes into ``groups''. A group of $k\geq 2$ 
lanes consists of adjacent lanes that are totally asymmetric 
in the same direction (if such lanes exist). Other lanes are 
viewed as singleton groups. For a group containing $k\geq 2$ lanes, 
we can construct blocking measures
similar to those of Theorem \ref{cor:entire}, \textit{(iv), (b)}, 
cf. \eqref{def_bl2(iv)(b)}. On singleton groups we use blocking measures 
of single-lane ASEP as defined in \eqref{blocking_config}--\eqref{cond_rev_meas}.
Then partial blocking measures are defined by considering blocking 
measures on one group, setting downstream lanes to $1$ and 
upstream lanes to $0$. \newline\newline
 To state this precisely, we introduce more notation. Let $i_1=n-1$, 
 and for $k\geq 1$, define $i_{k+1}$ as follows: $i_{k+1}=i_k-1$ 
 unless lanes $i_k$ and $i_{k}-1$ are totally asymmetric in the 
 same direction, that is, $0=d_kl_k+d_{k+1}l_{k+1}<d_kd_{k+1}+l_kl_{k+1}$. 
 In this case we set $i_{k+1}$ to be the smallest $\ell<i_k$ such that 
 lanes $\ell$ to $i_k$ are totally asymmetric in the same direction. 
 For some $k=k_0$ we eventually reach $i_k=0$. 
Let $G_k=\{i_{k+1},\ldots,i_k\}$ be the $k$-th group, and let 
$s_k:=(j_{i_{k+1}},\ldots,j_{i_k})$ be a ``shift vector'' taking 
values in $\Z\cup\{\pm\infty\}$ such that 
$j_{i_{k+1}}\geq\cdots\geq j_{i_k}$.  Similarly to \eqref{def_bl2(iv)(b)}, 
we define  $\nu^{\bot,j_{i_{k+1}},\ldots,j_{i_k}}$ to be the Dirac 
mass on the configuration $\eta^{\bot,j_{i_{k+1}},\ldots,j_{i_k}}$ 
whose restriction to lane  $\ell\in\{i_{k+1},\ldots,i_k\}$
is $\eta^*_{j_\ell}$.  We can then define a family of 
(generally partial) blocking measures $\nu^{\bot,k,s_k}$ 
indexed by a group number $k=1,\ldots,k_0$ and a shift vector 
$s_k$, but excluding  group numbers $k$ such that  $|G_{k}|=1$ 
with drift $\gamma_{i_k}=0$.  Under the measure $\nu^{\bot,k,s_k}$:\newline\newline
(A) All lanes with numbers $i>i_k$ are fully occupied, 
i.e. $\eta(z,i)=1$ for all $z\in\Z$.\\
(B) All lanes with numbers $i<i_{k+1}$ are empty, i.e. 
$\eta(z,i)=0$ for all $z\in\Z$. \\
(C) Assume $G_k=\{i_k\}$. Then if $\gamma_{i_k}>0$, the restriction 
of $\nu^{\bot,k,s_k}$ to lane $i_k$ is $\widehat{\mu}_{n}$ 
with $n=j_{i_k}$, given by \eqref{blocking_config} or 
\eqref{cond_rev_meas}  with $l=l_{i_k},d=d_{i_k}$. 
If $\gamma_{i_k}<0$, the restriction of $\nu^{\bot,k,s_k}$ to lane 
$i_k$ is the image of $\widehat{\mu}_{n}$ with $n=j_{i_k}$ by the 
symmetry operator $\sigma$  (defined in \eqref{sym_op}).  \\
(D) Assume $|G_k|\geq 2$. Then if lanes in $G_k$ are totally 
asymmetric to the right,  the restriction of $\nu^{\bot,k,s_k}$ 
to lanes in $G_k$ is $\nu^{\bot,j_{i_{k+1}},\ldots,j_{i_k}}$. 
If it is to the left, the restriction is the image of 
$\nu^{\bot,j_{i_{k+1}},\ldots,j_{i_k}}$ by $\sigma$. \newline\newline
Notice that only one group at a time can carry actual ``blocking'' 
measures. The above measures are $(\rho^-,\rho^+)$-shock measures 
where $\rho^\pm$ are integers such that  $0\leq\rho^-<\rho^+\leq n$. 
If there is a single group, i.e. all lanes are totally asymmetric 
in the same direction (as in Theorem \ref{cor:entire}, \textit{(iv) (b)}), 
finite-valued shift vectors yield global blocking measures, i.e. 
$\rho^-=0$ and $\rho^+=n$. If there are at least two groups 
(as in Theorem \ref{cor:entire}, \textit{(iv) (b)}, \textit{(v)} 
and \textit{(vi)}) , only partial blocking measures  are obtained, 
i.e. $\rho^+-\rho^-<n$.
\subsubsection{Single-lane models with several particles per site.} \label{item_mis}
We may consider the Misanthrope's process introduced in \cite{coc}. 
Let us recall the definition of this model. Given some 
$K\in\N\setminus\{0\}$, the state space of this process is 
$\mathcal X=\{0,\ldots, K\}^\Z$, and its generator of the form
\be\label{gen_mis}
Lf\left(\eta\right)
=\sum_{x,y\in V}p\left(x,y\right)b[\eta(x),\eta(y)]
\left(f\left(\eta^{x,y}\right)-f\left(\eta\right)\right),
\ee
where $p(x,y)=P(y-x)$ is a kernel satisfying \eqref{cond_lig}, 
and $b(.,.):\{0,\ldots,K\}^2\to[0,+\infty)$ is a jump rate function 
satisfying the following assumptions: \newline\newline
(M1)  $b(0,.)=b(.,K)=0$\\
(M2) For every $n,m\in\{0,\ldots,K\}$, $b(.,m)$ and $b(n,.)$ 
are respectively nondecreasing and nonincreasing.\\
(M3) For every $n,m\in\{0,\ldots,K\}$,
\be\label{condinvar}
\frac{b(n,m)}{b(m+1,n-1)}=\frac{b(n,0)b(1,m)}{b(m+1,0)b(1,n-1)}
\ee 
 (M4) For every $n,m\in\{0,\ldots,K\}$,
\be\label{gradient_cond}
b(n,m)-b(m,n)=b(n,0)-b(m,0)
\ee
{\em Homogeneous product invariant measures.} In \cite{coc},
homogeneous product invariant measures are constructed. 
The one-site marginal of these invariant measures is an exponential 
family  $(\theta^\lambda)_{\lambda\geq 0}$  
of probability measures on $\N$
of the form
\be\label{fugacity}
 \theta^\lambda(n):=Z(\lambda)^{-1}\lambda^n\theta^1(n)
\ee
where $\lambda$ is the fugacity, $Z(\lambda)$ the normalizing factor, 
and  $\theta^1(.)$  depends explicitely 
on the jump rate function $b(.,.)$. 
Homogeneous invariant measures are product measures 
 $\nu^\lambda$  on $\mathcal X$ such that
\be\label{inv_mis}
 \nu^\lambda[\eta(x)=i]=\theta^\lambda(i),\quad i\in\N,\,x\in\Z 
\ee
For $\lambda\to+\infty$,  $\theta^\lambda$  converges weakly to 
$\delta_K$ and  $\nu^\lambda$  to the corresponding product 
measure under which each site has $K$ particles. We may 
thus define by extension  $\theta^{+\infty}$ and $\nu^{+\infty}$. 
Then, one can reparametrize the family 
 $(\nu^\lambda)_{\lambda\in[0,+\infty]}$  to get a family 
$(\nu_\rho)_{\rho\in[0,K]}$ of product invariant measures 
indexed by the mean density  of particles, i.e. such that 
the expectation of $\eta(x)$ under $\nu_\rho$ is $\rho$,  by setting
\be\label{repar_mis}
\nu_\rho:=\nu^{R^{-1}(\rho)}
\ee
where $R(\lambda)$ is the mean of $\theta^\lambda$ 
(which is increasing and continuous with respect to $\lambda$).
With these measures, a characterization theorem similar to 
Theorem \ref{thm:characterization_lemma} is given  in \cite{coc}  for 
$\mathcal I\cap\mathcal S$, under the assumption that 
the jump kernel $p(.,.)$ is weakly irreducible. \newline\newline
{\em Blocking measures.} In the case of nearest-neighbour jumps, 
 $P(1)=d$ and $P(-1)=l$,  explicit blocking measures can be obtained 
by letting the fugacity in \eqref{inv_mis} depend on the site 
as follows:
\be\label{blocking_mis}
\mu^c[\eta(x)=i]= \theta^{\lambda(x)}(i), \quad i\in\N,\,x\in\Z
\ee
with
\be\label{fugacity_blocking}
\lambda(x)=c\left(\frac{d}{l}\right)^x,\quad c>0
\ee
Such blocking measures are studied in \cite{bfj}  as a basis 
for deriving remarkable combinatorial identites.
Interestingly, though the Misanthrope's  and two-lane exclusion 
process look quite different, the particular structure 
\eqref{cond_rev_twolane}--\eqref{shift_rev_twolane} is found 
in both settings. Namely, the above blocking measures can be 
decomposed by conditioning on the analogue of \eqref{def_H_2}, 
that is here the conserved quantity (when initially finite)
\be\label{def_H_single}
H(\eta):=
\sum_{ x\in\Z:\,x\leq 0}
\eta(x)-\sum_{ x\in\Z: \, x>0}[1-\eta(x)]
\ee 
As in Lemma \ref{def_blocking_h2}, the  conditioned measure 
 \be\label{78bis}
\mu^c(.|H(\eta)=k)\ee 
 does not depend on $c>0$. \newline\newline
{\em Characterization of invariant measures.}
To our knowledge, there exists so far no characterization result 
for $\mathcal I$. As mentioned in the introduction and explained 
in Subsection \ref{subsec:ideas}, compared to what is known for ASEP, 
new problems are induced by the fact that several particles per site 
are allowed. 
With a suitable adaptation of our proofs, the following results may be 
obtained for this model in the line of Theorems \ref{thm:Characterization_of_invariant_measures_gen}--\ref{cor:entire}. \par\medskip 
\noindent
1. Extremal invariant measures that are not homogeneous 
product measures consist (up to shifts) of a finite set of shock 
measures with integer amplitude $k\in\{1,\ldots,K\}$. For each 
$k\in\{1,\ldots,K\}$, there are at most $k$ shock measures. 
For $k=K$, there are exactly $K$ shock measures, which are 
the above conditioned blocking measures  \eqref{78bis}. \par\medskip
\noindent
2. For $K=2$, a function $b(.,.)$ satisfying conditions 
 (M1)--(M4)  above is uniquely determined by the parameters 
$b(1,0)$, $b(2,0)\geq b(1,0)$, and $b(1,1)\leq b(1,0)$; 
then $b(2,1)=b(2,0)-b(1,0)$.
One can then obtain the following result similar to 
\textit{(iii)} of Theorem \ref{cor:entire}. When 
\be\label{similar_mis}
|b(2,0)-2b(1,0)|\mbox{  and }b(1,1)\mbox{ are small enough},
\ee
all extremal invariant measures are either homogeneous 
product measures or blocking measures (i.e., there is no shock 
measure of amplitude $1$).
An explanation of the link between condition \eqref{similar_mis} 
and the set $\mathcal Z$ in \textit{(iii)} of Theorem 
\ref{cor:entire} (i.e. the conditions 
that $d$ is close to $1/2$ and $r$ is small enough) is given 
 by Lemma \ref{lemma:mapping_flux} below. \par\medskip
\noindent
3. If $p(.,.)$ is weakly irreducible and symmetric, all extremal 
invariant measures are homogeneous product measures.
\subsection{Extensions of main ideas}\label{app:ext_proof}
We now comment on the robustness of the steps  
of proof outlined in Subsection \ref{subsec:ideas}  
with respect to the extensions mentioned in Appendix \ref{app:ext_res}. 
These steps mainly use the following general properties of the model:
(i) attractiveness property \eqref{eq:attra}--\eqref{attractive};
(ii) weak irreducibility property  (see Definition \ref{def_irred}) 
 for the global kernel \eqref{restrict_kernel} when $q>0$, 
see Lemma \ref{lemma:irred};
(iii) finite propagation property (Proposition \ref{prop:prop});
(iv) the characterization Theorem \ref{thm:characterization_lemma} 
for $(\mathcal I\cap\mathcal S)_e$;
(v) the fact that $(\mathcal I\cap\mathcal S)_e$ consists 
of product measures.\newline\newline
Besides these properties, we use a fairly explicit expression 
of the flux function $G(\rho)$ in Step four, and the  incomplete 
ordering relations $\eta\supinf\xi$ and $\eta\bowtie\xi$ introduced 
in the case $q=0$. The explicit expression of the flux is allowed 
by property (v), and its degree of precision is still improved when 
the number of lanes is $2$.
For the models in  Appendix \ref{app:ext_res}, we have: \par\bigskip
\noindent
1. {\em For all models.} 
The ingredients listed in Step one hold so long as the global kernel 
$p(.,.)$ (see \eqref{restrict_kernel} for multilane models) is weakly 
irreducible. This is the case for the multilane model 
\eqref{natural_multi} when $q>0$;
for the ladder process \eqref{cyclic_multi} even if $q=0$; 
for the model with finite-range horizontal kernels $Q_i(.)$ 
in \eqref{restrict_kernel}, if $q>0$ and each $Q_i(.)$ is assumed 
weakly irreducible;  for the Misanthrope's process  if the single-lane 
kernel $p(.,.)$ in \eqref{gen_mis} is weakly irreducible.\par\medskip
\noindent
2. {\em For non-nearest neighbour horizontal kernels.} 
Expression \eqref{eq:G-from-rho_0} of the flux function is still 
valid. When $q=0$, in Step six above, the use of 
\cite[Theorem 1.4]{Liggett1976} (which is restricted to nearest-neighbour 
kernels) can be replaced by the more general \cite{Bramson2002}.\par\medskip
\noindent
 3. {\em For non-nearest neighbour horizontal kernels and 
multilane models.} 
Suitable extensions of the incomplete ordering relations 
$\eta\supinf\xi$ and $\eta\bowtie\xi$ can be introduced when $q=0$.\par\medskip
\noindent
4. {\em For multilane models.} 
When $q=0$, an explicit expression of the form 
\eqref{flux_degenerate} for the flux when $q=0$  still holds. 
When $q>0$, we do not know how to obtain as detailed information 
on the flux function $G(\rho)$ as in Proposition \ref{prop:extrema}, 
because its expression is less explicit. However, 
one can still show that $G$ has finitely many extrema,  
which implies a weaker form of statement \textit{(vi)} 
in Proposition \ref{prop:extrema}: namely that the equation 
$G(\rho+k)-G(\rho)$ has finitely many solutions for any integer $k$. 
This allows us to infer in Step four above that the number of possible 
shock profiles is finite.\par\medskip
\noindent
5. {\em For Misanthrope's process.} 
The proof of Theorem \ref{cor:entire} \textit{(i)} is similar 
to the two-lane ASEP proof. Indeed using the symmetry of the jump kernel 
$p(.,.)$ in \eqref{gen_mis} and the gradient condition (M2), 
one can write the microscopic current as a gradient 
as in \eqref{current_diff}.\newline\newline
We next come to possible shock measures when $\gamma\neq 0$
 ($\gamma:=\sum_{z\in\Z}zp(z)$ denotes
the mean drift of the jump kernel). 
The flux function expressed as a function of fugacity is a 
ratio of two polynomials. Indeed, let
\be\label{density_fug}R(\lambda):=\mu_\lambda[\eta(0)]=
\frac{\sum_{k=0}^K k\lambda^k\theta_1(k)}{\sum_{k=0}^K \lambda^k\theta_1(k)}
\ee
denote the mean density as a function of fugacity. Then, 
\be\label{flux_mis_fug}
G(R(\lambda))=\gamma\mu_\lambda[b(\eta(0),\eta(1)]
=\gamma\frac{\sum_{k=0}^K\sum_{l=0}^Kb(k,l)
\lambda^{k+l}\theta_1(k)\theta_1(l)}
{\left(\sum_{k=0}^K\lambda^{k}\theta_1(k)\right)^2}
\ee
{}From this one can show that when $\gamma\neq 0$, $G(.)$ has 
finitely many extrema.
This leads (as above for multilane models) to the conclusion 
that there are finitely many possible shock profiles. The more complete result
under condition \eqref{similar_mis} can be obtained because for $K=2$, 
the misanthrope's flux is as explicit as that of the two-lane ASEP. \newline\newline
 More precisely, the following mapping  proven below  holds between fluxes of 
two-lane ASEP and two-particle misanthrope's process. 
\begin{lemma}\label{lemma:mapping_flux}
Let $K=2$. Without loss of generality, assume $\gamma=1$, 
$b(1,0)=1$, $b(2,0)=\alpha\geq 1$, $0<b(1,1)=\beta\leq 1$, 
$b(2,1)=\alpha-1$ (cf. \eqref{gradient_cond}).  Let
\[
G^M_{\alpha,\beta}(\rho):=\nu_{\rho,\alpha,\beta}[b(\eta(0),\eta(1))]
\]
denote the macroscopic flux function 
of the corresponding Misanthrope's process, where 
$\nu_{\rho,\alpha,\beta}$ is the product invariant measure of 
this process with mean density $\rho$,  see 
\eqref{repar_mis}, where we added notational dependence on $\alpha,\beta$. 
Denote by $G_{\gamma_0,\gamma_1,r}$ the flux function of the two-lane ASEP, 
cf. \eqref{symmetry_g}--\eqref{hom_g}.
Then we have
\be\label{mapping_fluxes}
G_{\gamma_0,\gamma_1,r}=G^M_{\alpha,\beta}
\ee
if the following relations hold:
\be\label{mapping_param_1}
\frac{r}{(1+r)^2}=\frac{\beta}{\alpha}
\ee
and
\be\label{mapping_param_2}
\gamma_0r+\gamma_1=1+r,\quad\gamma_0+\gamma_1=\alpha
\ee
In particular, for given $0<\beta\leq 1\leq\alpha$ such that 
 $\beta\leq {\alpha}/{4}$,  the system \eqref{mapping_param_1}--\eqref{mapping_param_2} 
has a unique solution $(r,\gamma_0,\gamma_1)$ such that $r\in(0,1]$ 
and $\gamma_0,\gamma_1\geq 0$.
\end{lemma}
We note that for $\gamma_0=\gamma_1$ and $r\to 0$, we obtain $\beta\to 0$ 
and $\alpha\to 2$. Thus the image of the set $\mathcal Z$ in 
\textit{(iii)} of Theorem \ref{cor:entire} is a neighbourhood of 
$(\alpha=2,\beta=0)$ excluding $\beta=0$. 
\begin{proof}[Proof of Lemma \ref{lemma:mapping_flux}]
For $n\in\{0,\ldots,K\}$, let 
\[
q(n):=\frac{b(n,0)}{b(1,n-1)}, \quad q(n)!:=\prod_{i=1}^n q(i)
\]
where by convention the empty product equals $1$. 
The one-site marginal of  $\nu^1$  is then given (\cite{coc}) by 
\[
\theta^1(n)=\frac{1}{q(n)!}
\]
Under the assumptions of the lemma, we have
\[
q(0)=0,\quad q(1)=1,\quad q(2)=\frac{\alpha}{\beta}
\]
Plugging this into \eqref{flux_mis_fug}, 
we obtain the density and flux of the 
misanthrope's process as functions of fugacity:
\begin{eqnarray}
R^M(\lambda) & = & \frac{\lambda+2\frac{\beta}{\alpha}\lambda^2}
{1+\lambda+\frac{\beta}{\alpha}\lambda^2} \label{density_mis_fug_2}\\
\widetilde{G}^M(\lambda) & = & \frac{\lambda+2\beta\lambda^2
+(\alpha-1)\frac{\beta}{\alpha}\lambda^3}
{\left(1+\lambda+\frac{\beta}{\alpha}\lambda^2\right)^2}\label{flux_mis_fig_2}
\end{eqnarray}
We want to match the above expressions with the density and flux 
of two-lane ASEP as functions of fugacity $\lambda$. These respectively 
correspond to 
$\rho_0+\rho_1$ in \eqref{density_twolane_fug} and 
$\widetilde{G}(\lambda)$ in \eqref{flux_twolane_fug}. 
They can be written  as follows,
first in $\lambda$, and then in $\Lambda:=(1+r)\lambda$:
\begin{eqnarray}
R(\lambda) & = & 
\frac{(1+r)\lambda+2r\lambda^2}{1+(1+r)\lambda+r\lambda^2}
=\frac{\Lambda+\frac{2r}{(1+r)^2}\Lambda^2}
{1+\Lambda+\frac{r}{(1+r)^2}\Lambda^2}=:S(\Lambda)\label{density_asfollows} \\
\widetilde{G}(\lambda) 
& = & \frac{(\gamma_0r+\gamma_1)\lambda+2(\gamma_0+\gamma_1)r\lambda^2
+r(\gamma_0+\gamma_1 r)\lambda^3}{\left[1+(1+r)\lambda+r\lambda^2\right]^2}\nonumber\\
& = & \frac{\frac{(\gamma_0r+\gamma_1)}{1+r}\Lambda+2(\gamma_0
+\gamma_1)\frac{r}{(1+r)^2}\Lambda^2+\frac{r}{(1+r)^3}(\gamma_0+\gamma_1 r)
\Lambda^3}{\left[1+\Lambda+\frac{r}{(1+r)^2}\Lambda^2\right]^2}\nonumber\\
& =: & \widetilde{H}(\Lambda)\label{flux_asfollows}
\end{eqnarray}
We see that $R^M=S$ and $\widetilde{G}^M=\widetilde{H}$ if 
\eqref{mapping_param_1} and \eqref{mapping_param_2} hold as well as
\[
\gamma_0+\gamma_1 r=(1+r)(\alpha-1)
\]
But the latter is actually a consequence of \eqref{mapping_param_2}. 
Finally, since $G=\widetilde{H}\circ S^{-1}$ and 
$G^M=\widetilde{G}^M\circ (R^M)^{-1}$, we obtain $G^M=G$.
\end{proof}
\section{Proof of Proposition \ref{prop:tight}}\label{app:tight}
First it is a standard fact for Markov processes that 
$M'\in\mathcal I$. We must prove that it is supported 
on the set 
$\mathcal X_{2,1}$ in \eqref{def_H_2_1}.
We consider the coupled process $(\eta_t,\xi_t)_{t\geq 0}$. 
Its distribution at time $t$ is denoted by $\overline{\nu}_t$, and we set
\[
\overline{M}_t:=\frac{1}{t}\int_0^t\overline{\nu}_s\,ds
\]
The family $(\overline{M}_t)_{t\geq 0}$ is tight because 
it is supported on the compact set $\mathcal X\times\mathcal X$. 
Thus there exists a subsequential weak limit $\overline{M}$ 
of  the  family such that $M'$ is the $\xi$-marginal of $\overline{M}$.
By attractiveness and particle conservation, $\xi_t$ is obtained 
from $\eta_t$ by adding a second-class particle at some site. 
This implies for any $t>0$,
\be\label{difference_H}
H_2(\xi_t)=H_2(\eta_t)+1=1
\ee
because $H_2(.)$ is a conserved quantity,
thus $M'_t$ is supported on $\mathcal X_{2,1}$.  
However since $H_2$ is not continuous on $\mathcal X_2$, 
it is not {\em a priori} true that $M'$ is supported on $\mathcal X_{2,1}$.
However we now prove that it is indeed the case.
We couple via a common Harris system 
(see Subsection \ref{subsec:couple}) the process $(\xi_t)_{t\geq 0}$ 
starting from $\xi$ and the processes $\left(\eta^{(n)}_t\right)_{t\geq 0}$ 
starting from $\eta^{(n)}:=\tau_n\eta$. We denote 
$\vec{\eta}:=\left(\eta^{(n)}\right)_{n\in\N}$, 
$\vec{\eta}_t:=\left(\eta^{(n)}_t\right)_{n\in\N}$ 
(so $\eta^{(0)}=\eta$ and $\eta^{(0)}_t=\eta_t$).
Let 
$\vec{\mathbb{P}}$ denote the law of the process 
$(\xi_t,\vec{\eta}_t)_{t\geq 0}$ starting from the random coupled 
configuration $(\xi,\vec{\eta})$. 
For $\eta\in\mathcal X_2$, we set
\[
X_1(\eta):=\inf\{x\in\Z:\,\eta(y,0)=\eta(y,1)=1\mbox{ for all }y\geq x\}\in\Z
\]
Note that $X_1(\eta)\geq X_0(\eta)$, cf. \eqref{def_leftmost}, and
\be\label{prob_shift}
\vec{\mathbb{P}}\left\{
\xi\leq\eta^{(n)}
\right\}\geq
\vec{\mathbb{P}}\left\{
X_1(\eta)-n\leq X_0(\eta)-1
\right\}\to 1\mbox{ as }n\to+\infty
\ee
The inequality above holds because if we shift  $\eta$ 
far enough to the left so that its fully occupied region 
$[X_1(\eta),+\infty)\cap\Z$ comes to the right of 
the second-class $\xi$-particle at $X_0(\eta)-1$, 
the shifted configuration becomes greater or equal than $\xi$. 
Then
\be\label{using_shift}
\vec{\mathbb{P}}\left\{
X_0(\xi_t)\geq X_0\left(\eta^{(n)}_t\right)
\right\}\geq\vec{\mathbb{P}}\left\{
\xi_t\leq\eta^{(n)}_t
\right\}\geq\vec{\mathbb{P}}\left\{
\xi\leq\eta^{(n)}
\right\}
\ee
where the first inequality follows from the fact that 
the position of the leftmost particle is a nonincreasing function, 
and the second one from attractiveness.
Since $\left(\eta^{(n)}_t\right)_{t\geq 0}$ is stationary, 
the family $\left\{X_0\left(\eta^{(n)}_t\right):\,t\geq 0\right\}$ is tight. 
We next write
\be\label{next_write}
\vec{\mathbb{P}}\left\{
X_0(\xi_t)<-A
\right\}\leq \vec{\mathbb{P}}\left\{
X_0(\xi_t)< X_0\left(\eta^{(n)}_t\right)
\right\}+\vec{\mathbb{P}}\left\{
X_0\left(\eta^{(n)}_t\right)<-A
\right\},
\ee
We use \eqref{prob_shift}, \eqref{using_shift}, \eqref{next_write} 
and the above mentioned tightness, let $A\to+\infty$ and then 
$n\to+\infty$, to obtain
\be\label{hence_tight}
\liminf_{A\to+\infty}\liminf_{t\to+\infty}
\vec{\mathbb{P}}\left\{X_0(\xi_t)\geq -A\right\}=1
\ee
On the other hand, since $\xi_t\geq\eta_t$ by attractiveness, we also have
\be\label{other_hand}
\vec{\mathbb{P}}\left\{
X_1(\xi_t)\leq X_1(\eta_t)
\right\}=1
\ee
Again using stationarity and thus tightness of the process 
$\left(X_1(\eta_t)\right)_{t\geq 0}$, we obtain
%
%
\be\label{hence_tight_2}
\liminf_{A\to+\infty}\liminf_{t\to+\infty}
\vec{\mathbb{P}}\left\{X_1(\xi_t)\leq A\right\}=1
\ee
Since for $\eta\in\mathcal X$,
\[
N(\eta):=\sum_{ x\in V:\,  x(0)\leq 0}
\eta(x)+\sum_{ x\in V: \, x(0)>0}[1-\eta(x)]
\leq 2[\max(0,X_1(\eta))+\max(0,-X_0(\eta)],
\]
by 
\eqref{hence_tight} and \eqref{hence_tight_2},
$(N(\xi_t))_{t\geq 0}$ is a tight family. 
This implies $M'$ is supported on $\mathcal X_2$.
For $A\in\N$, let 
\[
H_2^A(\eta):=H_2(\eta):=
\sum_{ x\in V:\,  x(0)\in[-A,0]}
\eta(x)-\sum_{ x\in V: \, x(0)\in[1,A]}[1-\eta(x)]
\]
Note that 
\be\label{truncate_H}
H_2=H_2^A\quad\mbox{on }\{\eta\in\mathcal X_2:\,-A\leq X_1(\eta)\leq X_2(\eta)\leq A\}
\ee
It follows from 
%
\eqref{hence_tight} and \eqref{hence_tight_2} that
\[
\lim_{A\to+\infty}\liminf_{t\to+\infty}\vec{\mathbb{P}}\left\{
H_2^A(\xi_t)=H_2^A(\eta_t)+1
\right\}=
\lim_{A\to+\infty}\liminf_{t\to+\infty}\overline{\nu}_t\left\{
H_2^A(\xi)=H_2^A(\eta)+1
\right\}
=1
\]
By Cesaro limit along a subsequence of $\overline{M}_t$ converging to $\overline{M}$,
\[
\lim_{A\to+\infty}\liminf_{t\to+\infty}\overline{M}_t\left\{
H_2^A(\xi)=H_2^A(\eta)+1
\right\}
=1=\lim_{A\to+\infty}\overline{M}\left\{
H_2^A(\xi)=H_2^A(\eta)+1
\right\}
\]
Further, since $\lim_{A\to +\infty}H_2^A(\eta)=H_2(\eta)$ 
for every $\eta\in\mathcal X_2$, and $H_2$ and $H_2^A$ 
take integer values, on $\mathcal X_2\times\mathcal X_2$, we have
\[
\limsup_{A\to+\infty}\left\{
H_2^A(\xi)=H_2^A(\eta)+1
\right\}=\left\{
H_2(\xi)=H_2(\eta)+1
\right\}
\]
It follows by Fatou's lemma that 
\[
\overline{M}\left\{
H_2(\xi)=H_2(\eta)+1
\right\}=1
\]
so $M'$ is indeed supported on $\mathcal X_{2,1}$.
\section{Proof of  Proposition  \ref{prop_kilroy}}\label{app:nodis}
Let us rewrite the coupled generator \eqref{coupled_gen} as
\be\label{coupled_gen_2}
\overline{L}f(\eta,\xi)=\sum_{(\eta',\xi')
\in\mathcal X\times\mathcal X}a[(\eta,\xi);(\eta',\xi')]\left[
f(\eta',\xi')-f(\eta,\xi)
\right]
\ee
where the rates $a[(\eta,\xi);(\eta',\xi')]$ are defined as follows. 
First, for any $(x,y)\in V$ such that $x\neq y$,
$a[(\eta,\xi);(\eta',\xi')]$ is given by
\be\label{coupled_rates}
\left\{
\ba{lll}
p(x,y)[\eta(x)(1-\eta(y))]\vee[\xi(x)(1-\xi(y))] 
& \mbox{if} & (\eta',\xi')=\left(\eta^{x,y},\xi^{x,y}\right)\\
p(x,y)[\eta(x)(1-\eta(y))-\xi(x)(1-\xi(y))]^+
 & \mbox{if} & (\eta',\xi')=\left(\eta^{x,y},\xi\right)\\
p(x,y)[\eta(x)(1-\eta(y))-\xi(x)(1-\xi(y))]^- 
& \mbox{if} & (\eta',\xi')=\left(\eta,\xi^{x,y}\right)
\ea
\right.
\ee
with the kernel $p(.,.)$ given by \eqref{eq:intensity_c}. 
Next, $a[(\eta,\xi);(\eta',\xi')]=0$ if there exists no $(x,y)\in V^2$ 
such that $x\neq y$ and $(\eta',\xi')\in\left\{
\left(\eta^{x,y},\xi^{x,y}\right), 
\left(\eta^{x,y},\xi\right), \left(\eta,\xi^{x,y}\right)
\right\}$.\newline\newline
If $a[(\eta,\xi);(\eta',\xi')]\neq 0$, 
we say there is a transition from $(\eta,\xi)$ to $(\eta',\xi')$.
 Recalling the notation $x\stackrel{k}{\rightarrow}y$ introduced before 
 Definition \ref{def_irred}, we shall prove the following. 
\begin{lemma}\label{lemma_nodis}
Let $\widetilde{\nu}\in\widetilde{\mathcal I}
\cap\widetilde{\mathcal S}$. 
Then \eqref{nodis} holds for every $(x,y)\in V\times V$ 
such that $x\neq y$, and 
$x\stackrel{k}{\rightarrow}y$ or $y\stackrel{k}{\rightarrow}x$ 
for some $k$.
\end{lemma}
\begin{proof}[Proof of Lemma \ref{lemma_nodis}]
We prove by induction on $k$ that \eqref{nodis} 
holds for every $(x,y)\in V\times V$ 
such that $x\neq y$ and $x\stackrel{k}{\rightarrow}y$.
Applying the statement to $(\xi,\eta)$ then shows that it holds for 
$(\eta,\xi)$ and $y\stackrel{k}{\rightarrow}x$. \newline\newline
 We now use the computation done between \eqref{sum_rect} 
 and  \eqref{coupled_current}.
 The sums in  \eqref{computation_1}--\eqref{computation_2} 
 are boundary contributions, 
 that we denote respectively by $\Gamma_N^i(\eta,\xi)$ 
 and $\Gamma_N^o(\eta,\xi)$.  
 Since $\widetilde{\nu}\in\widetilde{\mathcal I}$, we have
\be\label{invtilde-2}
\int_{\mathcal X\times\mathcal X}
\widetilde{L} F_N(\eta,\xi)  d\widetilde{\nu}(\eta,\xi)=0
\ee
 We have  to exploit \eqref{invtilde-2}; for this 
 we distinguish between the two assumptions: \newline\newline
\textit{First case.} We assume that 
$\widetilde{\nu}\in\widetilde{\mathcal S}$. 
Since $J_{((u+z,i),(v+z,j))}=\tau_z J_{(u,i),(v,j)}$ 
for all $u,v,z\in\Z$,  we have
\be\label{cancel}
\int_{\mathcal X\times\mathcal X}
\Gamma_N^i(\eta,\xi)d\widetilde{\nu}(\eta,\xi)
-\int_{\mathcal X\times\mathcal X}
\Gamma_N^o(\eta,\xi)d\widetilde{\nu}(\eta,\xi)=0
\ee
 \textit{Second case.} We assume 
 \eqref{assumption_finite_disc}. The latter with the inequalities
\begin{eqnarray*}
|\Gamma_N^i(\eta,\xi)| 
& \leq & \sum_{i\in W}l_i(|\eta(-N-1,i)
-\xi(-N-1,i)|+|\eta(-N,i)-\xi(-N,i)|) \\
|\Gamma_N^o(\eta,\xi)| & \leq & 
\sum_{i\in W}d_i(|\eta(N,i)-\xi(N,i)|+|\eta(N+1,i)-\xi(N+1,i)|)
\end{eqnarray*}
leads to
\be\label{cancel_2}
\lim_{N\to+\infty}\left\{
\int_{\mathcal X\times\mathcal X}
\Gamma_N^i(\eta,\xi)d\widetilde{\nu}(\eta,\xi)
-\int_{\mathcal X\times\mathcal X}
\Gamma_N^o(\eta,\xi)d\widetilde{\nu}(\eta,\xi)
\right\}=0
\ee 
 Using  \eqref{cancel} for all $N$,  
 we obtain that for every $(x,y)\in V^2$ such that $p(x,y)>0$,
\[
\int_{\mathcal X\times\mathcal X} 
D_{x,y}(\eta,\xi)d\widetilde{\nu}(\eta,\xi)=0
\]
This implies \eqref{nodis} for $k=1$. \newline\newline
Now assume \eqref{nodis} holds for $k-1$.
If $A$ is a subset of $\mathcal X\times\mathcal X$ and 
$(\eta,\xi)\in\mathcal X\times\mathcal X$, 
we write  $(\eta,\xi)\stackrel{n}{\rightarrow}A$  
if there exists a sequence of coupled 
configurations,  $(\eta_0,\xi_0)
=(\eta,\xi),\ldots,(\eta_n,\xi_n)=(\eta',\xi')$,  
such that $a[(\eta_i,\xi_i);(\eta_{i+1},\xi_{i+1})]>0$ 
for every  $i=0,\ldots,n-1$,  and $(\eta',\xi')\in A$.
Assume $A=A_0$ is a local set (that is, 
such that its indicator function is a local function) and 
\begin{eqnarray*}
A_n & := & \{(\eta,\xi)\in\mathcal X\times\mathcal X:\,
(\eta,\xi)\stackrel{n}{\rightarrow} A_0\}\\
A'_n & := & \{(\eta,\xi)\in\mathcal X\times\mathcal X:\,
(\eta,\xi)\stackrel{i}{\rightarrow} A_0\mbox{ for some }i\leq n\}
\end{eqnarray*} 
 Then \eqref{coupled_gen_2}--\eqref{coupled_rates} 
 implies that there exist positive 
 constants  $a_n,b_n$ such that
\be\label{genleadsto}
 \widetilde{L}{\bf 1}_{A_n}  \geq a_n{\bf 1}_{A_{n+1}}-b_n{\bf 1}_{A_n}
\ee
 Iterating \eqref{genleadsto} shows that if 
 $\widetilde{\nu}\in\widetilde{\mathcal I}$ and 
 $\widetilde{\nu}(A)=0$, then  $\widetilde{\nu}(A_n)=0$,  
 hence $\widetilde{\nu}(A'_n)=0$. 
For the induction step, we use  this as follows. 
Let $E^n$ denote the set of 
coupled configurations $(\eta,\xi)\in\mathcal X\times\mathcal X$ 
such that 
there is no pair of opposite discrepancies at sites $x,y\in V$ if 
$x\stackrel{i}{\rightarrow}y$ or $y\stackrel{i}{\rightarrow}x$ 
for any $i\leq n$.
We choose $A_0=E^{k-1}$ so that $\widetilde{\nu}(A_0)=0$ 
by the induction assumption. 
Then we claim that $E^k$ is contained in $A'_{k-1}$.  Indeed,
assume $x\stackrel{k}{\rightarrow}y$ and
$(\eta,\xi)\in E_{x,y}$.   Let $(x=x_0,\ldots,x_k=y)$ denote a $p$-path 
from $x$ to $y$. By the induction assumption, 
$\widetilde{\nu}$-almost surely, 
we have $\eta(x_i)=\xi(x_i)$ for all 
$i=1,\ldots,k-1$. If $\eta(x_1)=\xi(x_1)=0$, then
$(\eta^{x_0,x_1},\xi^{x_0,x_1})\in E_{x_1,y}$. 
Otherwise let $i^*$ be the maximum index $i$ such that 
$\eta(x_i)=\xi(x_i)=1$. Then
one can find a sequence of at most $k-1$ transitions 
leading from $(\eta,\xi)$ 
to some $(\eta',\xi')\in E_{x,x_{k-1}}$ as follows: 
\textit{(i)}  if $i^*<k-1$, 
the coupled particle at $x_{i^*}$ jumps from $x_{i^*}$ 
to $x_{k-1}$ along the path; 
\textit{(ii)} the coupled particle at $x_{k-1}$ 
exchanges with the $\xi$-discrepancy at $y$. 
\end{proof}
\end{appendix}
\noindent 
{\bf Acknowledgements.} 
This work  has been conducted within the FP2M federation 
(CNRS FR 2036) and was partially supported by laboratoire MAP5,
grants ANR-15-CE40-0020-02 and ANR-14-CE25-0011 (for C.B. and E.S.),
LabEx CARMIN (ANR-10-LABX-59-01).
G.A. was supported by the Israel Science Foundation grant \# 957/20.
O.B. was supported by EPSRC's EP/R021449/1 Standard Grant.
Part of this work was done during the stay of C.B, O.B. and E.S. at the
Institut Henri Poincar\'e (UMS 5208 CNRS-Sorbonne Universit\'e) -
Centre Emile Borel for
the trimester ``Stochastic Dynamics Out of Equilibrium''.
The authors thank these institutions for hospitality and support.
C.B., O.B. and E.S. thank
Universit\'{e} Paris Descartes for hospitality, as well as Villa Finaly
(where they attended the conference ``Equilibrium and 
Non-equilibrium Statistical Mechanics'').
%
\bibliographystyle{plain}
\bibliography{MLTnew}

\end{document}